\documentclass[11pt]{amsart}

\usepackage{latexsym}
\usepackage{amssymb}
\usepackage{amsmath}
\usepackage{color}

\usepackage{tikz}
\usetikzlibrary{arrows.meta,positioning,calc}

\definecolor{qblue}{RGB}{18,45,115}
\definecolor{qred}{RGB}{150,35,35}
\definecolor{qgray}{RGB}{90,90,90}

\tikzset{
  qbox/.style={
    draw=qblue,
    very thick,
    rounded corners=4pt,
    align=center,
    fill=blue!2,
    inner sep=7pt
  },
  tbox/.style={
    draw=qred,
    very thick,
    rounded corners=4pt,
    align=center,
    fill=red!2,
    inner sep=7pt
  },
  qarrow/.style={
    -{Latex[length=3mm,width=2.3mm]},
    very thick,
    qblue
  },
  tarrow/.style={
    -{Latex[length=3mm,width=2.3mm]},
    very thick,
    qred
  },
  qdoublearrow/.style={
    {Latex[length=3mm,width=2.3mm]}-{Latex[length=3mm,width=2.3mm]},
    very thick,
    qblue
  },
  garrow/.style={
    -{Latex[length=3mm,width=2.3mm]},
    thick,
    dashed,
    qgray
  },
  lab/.style={
    font=\footnotesize,
    align=center
  }
}

\usepackage[a4paper]{geometry}
\newtheorem{theorem}{Theorem}[section]
\newtheorem{lemma}[theorem]{Lemma}
\newtheorem{proposition}[theorem]{Proposition}
\newtheorem{corollary}[theorem]{Corollary}
\newtheorem{definition}[theorem]{Definition}
\newtheorem{example}[theorem]{Example}
\newtheorem{remark}[theorem]{Remark}

\newcommand\id{\mathop{\rm id}}

\newcommand\Tr{\mathop{\rm Tr}}

\newcommand{\A}{\mathcal{A}}
\newcommand{\M}{\mathcal{M}}
\newcommand{\B}{\mathcal{B}}

\newcommand{\op}{\operatorname{op}}

\newcommand{\cl}[1]{\mathcal{#1}}
\newcommand{\bb}[1]{\mathbb{#1}}

\newcommand{\N}{\mathcal{N}}

\addtocontents{toc}{\protect\setcounter{tocdepth}{1}}

\begin{document}

\title[Quantum Non-local Games]{Quantum non-local games: Quantum relations, projection lattices and rule operators}

\author[A. Chatzinikolaou]{Alexandros Chatzinikolaou}
\address{ Institute of Mathematics of the Polish Academy of Sciences\\ ul. Śniadeckich 8\\
00--656 Warszawa\\ Poland}
\email{achatzinik.math@gmail.com}



\date{\today}

\begin{abstract}
Quantum non-local games with quantum inputs or outputs have been formulated
in several different languages, including  rank-one and quantum XOR games, support maps between projection
lattices, probabilistic quantum hypergraphs, and Frobenius-algebraic rule
operators on finite quantum sets.  We give a unified operator-algebraic
comparison of these models by assigning to each rule its winning
transformation space: the operator space of transformations accepted with certainty.
We introduce \(\mathcal R\)-projection-test quantum games with finite-dimensional input and output von Neumann algebras and an arbitrary, possibly infinite-dimensional, referee von Neumann algebra \(\mathcal R\). Their winning transformation spaces are precisely operator spaces with a natural bimodule structure, equivalently rectangular quantum relations.
We compare this formalism with projection-lattice games, hypergraph quantum games, and the graphical rule-operator  definition. Todorov--Turowska's projection-lattice games capture exactly the reflexive winning bimodules.  Hypergraph
quantum games admit value-preserving projection-test realisations,
and, after passing to perfect transformations, describe the same reflexive
part as projection-lattice games. Using Daws' technique, we also translate rule operators 
to projections in tensor products of finite-dimensional von Neumann algebras. This identification places graphical rules in the same operator-bimodule framework and preserves both values and perfectness.
Finally, we compare concurrency with the synchronicity conditions of Goldberg and of Bochniak--Kasprzak--So{\l}tan. In concrete representations, BKS synchronicity is equivalent to concurrency, while Goldberg synchronicity implies concurrency  and becomes equivalent to it under the special Frobenius normalisation. The framework is illustrated by classical, rank-one, quantum XOR, colouring, and quantum graph homomorphism and isomorphism games.
\end{abstract}

\maketitle


\section{Introduction}
Over the past twenty years, non-local games \cite{CHTW04} have become central to quantum information theory, serving as a primary method for testing and quantifying entanglement and offering an elegant operational setting for studying Bell inequalities.  In these cooperative, one-round games, two separated players try to convince a referee that they can coordinate their answers without communication. By allowing quantum strategies, the resulting input-output behaviour is described by shared states and local measurement operators. Accordingly, the analysis of non-local games naturally involves operator spaces, representations, and tensor products, placing the subject at the interface of quantum information theory and non-commutative analysis \cite{PV16}.  Questions about the strength of various quantum models can  therefore be translated  into  questions  about tensor products, traces, and representations
of operator algebras.  This point of view was made
particularly vivid by the works \cite{JNPPSW11,Fri12,Oza13}, which related
Tsirelson's problem to the Connes Embedding Problem.  Further significant developments
include the proof that the set of finite-dimensional quantum
correlations is not closed \cite{Slo19} (see also \cite{DPP19}), and the
breakthrough theorem \(\operatorname{MIP}^*=\operatorname{RE}\)  \cite{JNVWY20}, which separated 
approximately quantum and quantum commuting models and gave a negative answer
to the Connes Embedding Problem.

At the operational level, a classical two-player non-local game is specified
by finite question sets \(X,Y\), finite answer sets \(A,B\), a probability
distribution \(\pi\) on \(X\times Y\), and a rule predicate
\( 
\lambda:A\times B\times X\times Y\to\{0,1\}.
\)
In each round, the referee samples \((x,y)\) according to \(\pi\), sends \(x\)
to Alice and \(y\) to Bob, and receives answers \(a\) and \(b\), respectively.
The players win precisely when \(\lambda(a,b,x,y)=1\).  Thus, a correlation
\(p(a,b|x,y)\) wins the game with probability
\[
\omega(\lambda,\pi,p)
=
\sum_{x,y,a,b}
\pi(x,y)\lambda(a,b,x,y)p(a,b|x,y).
\]
Although the players may agree on a joint strategy before the game begins,
they are not allowed to communicate during the game.  Different assumptions on the
resources available to them lead to the familiar hierarchy of correlation sets and hence to different values of the game~\cite{HP16}.

Quantum games arise naturally when the information exchanged between the
referee and the players is transmitted through quantum states.  In this
setting, the referee's questions, the players' answers, or both, may be
quantum systems rather than classical symbols.  Unlike in the classical case,
however, there is no single evident replacement for the Boolean rule predicate
and the question distribution.  This has led to several related formalisms.  In
cooperative quantum games, the referee prepares a quantum input and tests the
returned system with an accepting measurement \cite{LTW13}; rank-one quantum
games form a distinguished subclass with a close connection to operator spaces
\cite{CJPP15}; and quantum XOR games provide a structured
quantum-input/classical-output model \cite{RV15}.  Other approaches encode the
rules by support-valued maps between projection lattices
\cite{TT24,BHTT23,BHTT24}, by probabilistic quantum hypergraphs and projection quantum games
\cite{CLTT23}, or by rule operators between finite quantum sets
\cite{Goldberg25,Goldberg26}.    Recent works have also developed the
value theory of quantum non-local games through resource-theoretic,
operator-space, and operator-algebraic methods
\cite{CLTT23,AJP26}. Several partial extensions of non-local games to quantum settings have also been
studied, including semi-quantum games, steering and monogamy-of-entanglement
games, and extended non-local games \cite{Bus12,Fri12,TFKW13,JMRW16}. Different formulations of quantum games are also discussed from the graphical point of view in Goldberg's thesis \cite{Goldberg25}.

While these diverse formalisms successfully capture different operational aspects of quantum interactions, they retain distinct mathematical data and structures, making it previously unclear how they formally relate to one another. The primary goal of this work is to systematically compare these seemingly disparate approaches and establish a unified mathematical framework for quantum non-local games. To bridge these models, we focus on the structural core of the games: the rules that determine which quantum transformations are accepted with certainty. We demonstrate that across all these varied frameworks, the common underlying object governing perfect strategies is an operator space, or more precisely an operator bimodule, which we term the \emph{winning transformation space}.  we compare different rule formalisms through their winning transformation spaces and determine which correspondences preserve perfectness and, where applicable, game values. To carry out the comparison we allow the referee to possess a possibly infinite-dimensional von Neumann algebra $ \cl R$.

The prototypical operational model for us is the cooperative quantum game
description of Leung--Toner--Watrous \cite{LTW13}, which we call a \emph{projection-test quantum
game}.  In this model, there are finite-dimensional input Hilbert spaces \(H_X,H_Y\), output
Hilbert spaces \(H_A,H_B\), and a referee space \(H_R\).  The referee prepares
a fixed input state
\( 
\psi\in (H_X\otimes H_Y)\otimes H_R,
\)
sends the \(H_X\)-system to Alice and the \(H_Y\)-system to Bob, and keeps the
\(H_R\)-system.  After the players act on the systems they receive with a transformation \(T:H_X \otimes H_Y \to  H_A\otimes H_B\), according to a prescribed strategy, and the referee tests
\((H_A\otimes H_B)\otimes H_R\) against an accepting projection \(Q\). They win precisely when 
\( 
Q^{\perp}(T\otimes 1_{H_R}) \psi =0.
\)
The winning probability of  a strategy $\Gamma : S^1(H_X) \otimes S^1(H_Y) \to  S^1(H_A) \otimes S^1(H_B)$ for the game $G=(\psi,Q)$  is given by
\[ 
\omega_{G}(\Gamma)= \Tr \left( ( \Gamma \otimes \id )(\psi \psi^*) Q\right).
\]

In Section \ref{sec:projection} of the paper we develop this point of view for projection-test
games. For the comparison developed here, we pass from this concrete operational
picture to a support-level formulation over finite-dimensional (although an extension to infinite dimensions may be possible using the recent framework of~\cite{Daws25}) von Neumann
algebras $\cl X \subseteq \B(H_X)$, $\cl Y \subseteq \B(H_Y)$, $\cl A \subseteq \B(H_A)$ and $ \cl B\subseteq \B(H_B)$.   We also consider an arbitrary (allowing infinite dimensions) von Neumann algebra $\cl R \subseteq \B(H_R)$ for the referee, and let
\( 
P\in (\cl X \otimes \cl Y)\otimes \cl R\) and
\( 
Q\in (\cl A \otimes \cl B)\otimes \cl R
\)
be projections. We call $ G=(P,Q)$, an \emph{$ \cl R$-projection-test quantum game}.  The associated winning transformation space is
\[
\mathcal U_{P,Q}
=
\{T\in \B(H_X \otimes H_Y,H_{A} \otimes H_B):
Q^\perp(T\otimes 1_{\cl R})P=0\}.
\]
Since \(P\) and \(Q\) lie in
\((\cl X \otimes \cl Y)\otimes \cl R\) and \((\cl A \otimes \cl B)\otimes \cl R\), respectively,
\(\mathcal U_{P,Q}\) is an \((\cl A \otimes \cl B)'\)-\((\cl X \otimes \cl Y)'\) bimodule.  Conversely, every
\((\cl A \otimes \cl B)'\)-\((\cl X \otimes \cl Y)'\) bimodule
\(\mathcal U\subseteq \B(H_X \otimes H_Y,H_{A} \otimes H_B)\) admits a canonical
generalised projection-test realisation 
(Theorem \ref{thm:abstract-projection-test}). In finite dimensions, these bimodules are precisely quantum relations from
\(\cl X \otimes \cl Y\) to \(\cl A \otimes \cl B\) in rectangular form.  The general theory of quantum relations between two von Neumann
algebras was recently developed
by Daws \cite{Daws26}.

The winning transformation space is
not itself the physical strategy set;  rather, if \(\Gamma\) is a channel
between the corresponding systems and \(\mathcal X_\Gamma\) denotes its Kraus
space, or the associated Kraus bimodule in the von Neumann algebraic setting,
then
\( 
\Gamma \)  is perfect for the game
 if and only if \( \mathcal X_\Gamma\subseteq\mathcal U_{P,Q}.
\) 
In the terminology of Daws \cite[Section~4]{Daws26}, the Kraus
bimodule $\mathcal X_\Gamma$ is the quantum relation associated with the adjoint channel. Thus, perfectness
means precisely that the quantum relation of the strategy is contained
in the winning quantum relation.
Thus, the winning transformation space records the rule at the Kraus level,
while the chosen strategy class determines which channels are available as
strategies.  Our emphasis is therefore on the rule side: we regard strategies
uniformly as channels and compare the corresponding Kraus-level perfectness
conditions.

This construction extends the classical non-local game picture.  Indeed, if
\(\lambda:A\times B\times X\times Y\to\{0,1\}\) is a classical rule predicate, then the corresponding winning space is the masa-bimodule
\[
\mathcal U_{\Lambda_{\lambda}}
=
\operatorname{span}
\{\varepsilon_{(a,b),(x,y)}:((x,y),(a,b)) \in \Lambda_{\lambda}\}
\subseteq \B(\mathbb C^{X\times Y},\mathbb C^{A\times B}),
\]
associated with the relation:
\[
\Lambda_\lambda
=
\{((x,y),(a,b))\in (X\times Y)\times(A\times B):
\lambda(a,b,x,y)=1\}.
\]
Here $\varepsilon_{(a,b),(x,y)}=e_{(a,b)}e_{(x,y)}^*$ is the rectangular matrix unit sending the canonical basis vector $e_{(x,y)}$ to $e_{(a,b)}$.
Conversely, every such masa-bimodule determines a canonical relation, and
hence a classical rule predicate.  Thus projection-test quantum games provide
a non-commutative analogue of binary relations; viewing relations as
hypergraphs, this quantisation of classical non-local games was already
present in \cite{HT23,HT25}.  The correspondence between relations and
arbitrary masa-bimodules originates in \cite{EKS98}.

In Section \ref{sec:projection_lattice} we study \emph{projection-lattice quantum games}, in the sense of Todorov--Turowska~\cite{TT24}, which capture the
reflexive part of this operator-bimodule picture.  In the matrix-algebra
setting, such a game is encoded by a zero-preserving join-continuous map from
the projection lattice of the input algebra to the projection lattice of the
output algebra, and perfectness is expressed by requiring the output support
of each input projection to lie under the prescribed accepting support.  We extend this to concrete finite-dimensional von Neumann
algebras.  Following the theory of subspace maps and reflexive operator spaces
\cite{LS75,Erdos86}, a \(\mathcal N'\)-\(\mathcal M'\) bimodule
\(\mathcal U\subseteq \B(H_{\rm in},H_{\rm out})\) determines the support map
\[
\varphi_{\mathcal U}:\mathcal P(\mathcal M)\to\mathcal P(\mathcal N),
\qquad
\varphi_{\mathcal U}(P)
=
P_{[\mathcal U P H_{\rm in}]}.
\]
Conversely, a zero-preserving join-continuous map
\(\varphi:\mathcal P(\mathcal M)\to\mathcal P(\mathcal N)\) determines
\[
\mathcal U_\varphi
=
\{T\in \B(H_{\rm in},H_{\rm out}):
\varphi(P)^\perp TP=0
\text{ for all }P\in\mathcal P(\mathcal M)\}.
\]
The space \(\mathcal U_\varphi\) is always reflexive, whereas
\( 
\mathcal U_{\varphi_{\mathcal U}}=\operatorname{Ref}(\mathcal U).
\)
Thus, projection-lattice games describe precisely the reflexive winning
transformation spaces after restricting to rules
\(\varphi=\varphi_{\mathcal U_\varphi}\) (Proposition \ref{prop:bimo_lattice}).  This qualification is essential:
a general projection-test game need not have a reflexive winning space.  We
give a simple example (Example
\ref{ex:not_lattice}) whose winning space is the trace-zero subspace of
\(M_2\), and is therefore not representable by a Todorov--Turowska
projection-lattice game with the same winning transformations.

The probabilistic quantum hypergraph formalism of \cite{CLTT23} fits the same
picture in a measurable form (Subsection
\ref{subsec:hypergr}).  In that setting, the referee samples a pure input
state according to a probability measure \(\mu\) and associates to it an
accepted output projection.  We show that such a game can be realised as a
projection-test game over the referee space \(L^2(\mathbb P,\mu)\), with the
accepting projection acting as a multiplication operator, and that the
corresponding values agree (Proposition~\ref{prop:hypergraph-projection-test}).  We also compare this with the
projection-lattice formalism: after passing to winning transformation spaces,
hypergraph quantum games and projection-lattice games describe the same
reflexive part of the theory  (Proposition~\ref{prop:hypergraph-reflexive}).  This agreement is only at the level of perfect
transformations; hypergraph games carry additional measure data and hence a
value theory, whereas projection-lattice games record only support
constraints.

A further aim of the paper (Section \ref{sec:Gold}) is to connect the operator-bimodule description
with Goldberg's graphical definition of quantum games
\cite{Goldberg25,Goldberg26}.  The passage between the graphical and
operator-algebraic settings follows the same guiding principle as the
equivalence between finite quantum graphs and operator bimodules in
\cite{MRV18, Daws24}: after choosing faithful traces, the underlying finite quantum
sets are realised as \(L^2\)-spaces of finite-dimensional tracial von
Neumann algebras, and graphical rules are converted into concrete operator
spaces. Our methods are operator-algebraic and may be viewed as a tracial specialisation of Daws’ comparison between different approaches to quantum graphs \cite{Daws24} and of his subsequent development of the rectangular theory of quantum relations and adjacency operators \cite{Daws26}. In particular, in the appropriate
symmetric and reflexive case, we recover the correspondence between quantum graphs and their operator-bimodule, or quantum-relation,
presentations.

Let  $\cl X \subseteq \B(H_X)$, $\cl Y \subseteq \B(H_Y)$, $\cl A \subseteq \B(H_A)$ and $ \cl B\subseteq \B(H_B)$ be finite-dimensional von Neumann
algebras equipped with faithful traces.  Goldberg's formalism is graphical and
is phrased in terms of quantum sets, that is, special symmetric dagger
Frobenius algebras.  For the operator-algebraic comparison, we do not impose
the special normalisation; instead, we work with the \(L^2\)-spaces associated with the chosen traces. In this setting, a \emph{rule operator} is a linear operator 
\[  
\lambda:
L^2(\mathcal X)\otimes L^2(\mathcal Y)
\to
L^2(\mathcal A)\otimes L^2(\mathcal B),
\]
satisfying Frobenius idempotency and self-conjugacy (Definition \ref{def:gold_game}). For simplicity assume $\cl M = \cl X\otimes \cl Y$ and $ \cl N =\cl A \otimes \cl B$.
The main bridge is the linear isomorphism
\[
\Psi_{\M,\N}:
\B(L^2(\mathcal M),L^2(\mathcal N))
\longrightarrow
\mathcal N\otimes\mathcal M^{\op},
\qquad
\Psi_{\M,\N}(\widehat a\,\widehat b^{\,*})=a\otimes b^* .
\]
Here $\widehat a\in L^2(\mathcal N)$ denotes the vector represented by $a\in\mathcal N$, and $\widehat a\,\widehat b^{\,*}$ is the rank-one operator $\widehat x\mapsto\langle\widehat b,\widehat x\rangle\widehat a$, for $b\in\mathcal M$.
Under this transform, the Frobenius convolution of rule operators becomes
ordinary multiplication in \(\mathcal N\otimes\mathcal M^{\op}\), while the self-conjugacy operation becomes the adjoint (Theorem \ref{thm:goldberg-daws-transform}).  Hence, the two rule
equations for a  rule operator \(\lambda\) are transported exactly to
the conditions that \(\Psi_{\M,\N}(\lambda)\) be an idempotent self-adjoint element (Corollary \ref{cor:goldberg-rule-projection}).
Thus, \( 
Q_\lambda:=\Psi_{\M,\N}(\lambda)\in \mathcal N\otimes\mathcal M^{\op},
\) is a projection
and therefore determines an 
\(\mathcal N'\)-\(\mathcal M'\) operator bimodule (Proposition \ref{prop:rectangular-daws}).

This correspondence also preserves perfect play. Given channels in the sense of \cite{Goldberg26}, we associate quantum channels between the underlying von Neumann algebras (Proposition \ref{prop:L2-channel-algebra-channel}).  We show that the canonical  value of a correlation \(f\) agrees with the winning probability of
the corresponding channel $\Gamma_f$.  Moreover, \(f\) is perfect
for the rule operator quantum game if and only if $\Gamma_f$ is perfect for the
projection-test quantum game (Theorem \ref{thm:goldberg-projection-test}). 

In Section \ref{sec:synchronicity} we compare the resulting formalism with several quantum versions of
synchronicity.  In the classical setting, synchronicity says that equal
questions force equal answers.  In projection-lattice quantum games, this is
replaced by \emph{concurrency} \cite{BHTT23}, where the classical diagonal is
replaced by the projection onto the maximally entangled vector \(\Omega_X\).  We
extend this condition to arbitrary  finite-dimensional
von Neumann algebras by replacing the projection onto $ \Omega_X$ with
\( 
\Pi_{\mathcal X}=P_{\operatorname{vec}(\mathcal X')}.
\)
For channels, concurrency is equivalent to 
\( 
\Pi_{\mathcal A}^{\perp}T_i\Pi_{\mathcal X}=0
\)
for every Kraus operator $T_i$
(Proposition~\ref{prop:concurrency-equivalences}).
We compare this condition with the notion of synchronicity introduced by Goldberg
\cite{Goldberg26} and with the synchronicity condition of
Bochniak--Kasprzak--So{\l}tan \cite{BKS23}. Goldberg's synchronicity
always implies concurrency, and the converse holds under the special
Frobenius normalisation
(Proposition~\ref{prop:Goldberg-concurrency}). After passing to
concrete faithful representations with multiplicities, BKS
synchronicity is equivalent to concurrency, and hence to the same
Kraus support condition
(Proposition~\ref{prop:BKS-concurrency}). Thus BKS synchronicity and
concurrency coincide at the level of channels, while Goldberg
synchronicity implies them in general and coincides with them under
the special Frobenius normalisation.

The final section (Section \ref{sec:examples}) applies the common framework to several examples.  Classical
games become masa-bimodules.  Rank-one games, including coherent
state exchange, have reflexive winning spaces (Proposition \ref{prop:rank-ones-refl}).  Quantum XOR games
and quantum-to-classical or classical-to-quantum games give one-sided masa
modules, which are again reflexive (Proposition \ref{prop:one-sided-masa-modules-reflexive}).  We also compare the projection-lattice
and rule operator formulations of quantum graph homomorphism and isomorphism games
\cite{TT24,BHTT23,BHTT24,Goldberg26}, and treat the quantum-to-classical graph
homomorphism game of Brannan--Ganesan--Harris \cite{BGH22}.  In these graph
examples, the winning spaces are described by explicit block-support
constraints and are therefore reflexive.

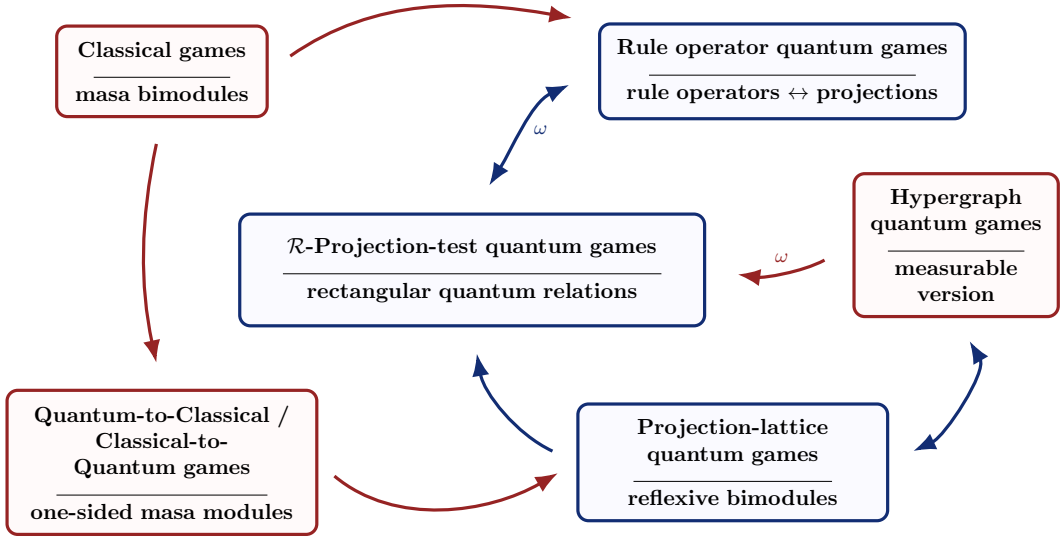
\begin{figure}[t]
\centering
\begin{tikzpicture}[scale=.82, transform shape, every node/.style={font=\small}]

\node[qbox, text width=7cm, minimum height=1.8cm] (pt) at (0,0)
{\bfseries $\cl R$-Projection-test quantum games\\[-.2em]
\rule{6.1cm}{0.35pt}\\[-.1em]
 rectangular quantum relations};

\node[tbox, text width=2.9cm] (cl) at (-5.0,3.2)
{\bfseries Classical games\\[-.2em]
\rule{2.2cm}{0.35pt}\\[-.1em]
masa bimodules};

\node[qbox, text width=5.4cm] (gold) at (5.0,3.2)
{\bfseries Rule operator quantum games\\[-.2em]
\rule{4.3cm}{0.35pt}\\[-.1em]
rule operators
\(\leftrightarrow\) projections};

\node[tbox, text width=2.8cm] (hyp) at (7.8,0.4)
{\bfseries Hypergraph quantum games\\[-.2em]
\rule{2.2cm}{0.35pt}\\[-.1em]
measurable version};

\node[tbox, text width=4.5cm] (qtc) at (-5.0,-3.1)
{\bfseries Quantum-to-Classical /\\
Classical-to-Quantum games\\[-.2em]
\rule{3.4cm}{0.35pt}\\[-.1em]
one-sided masa modules};

\node[qbox, text width=4.5cm] (pl) at (4.2,-3.1)
{\bfseries Projection-lattice quantum games\\[-.2em]
\rule{3.2cm}{0.35pt}\\[-.1em]
reflexive bimodules};

\draw[qarrow, shorten >=10pt, shorten <=10pt]
(pl.west) to[out=155,in=-80] (pt.south);

\draw[qdoublearrow, shorten >=12pt, shorten <=12pt]
(gold.west) to[out=205,in=58]
node[lab, midway, right] {\(\omega\)}
(pt.north);

\draw[tarrow, shorten >=12pt, shorten <=12pt]
(hyp.west) to[out=-152,in=-8]
node[lab, midway, above] {\(\omega\)}
(pt.east);

\draw[tarrow, shorten >=10pt, shorten <=10pt]
(cl.east) to[out=35,in=170] (gold.north west);

\draw[tarrow, shorten >=10pt, shorten <=10pt]
(cl.south) to[out=-102,in=102] (qtc.north);

\draw[tarrow, shorten >=8pt, shorten <=8pt]
(qtc.east) to[out=-40,in=-150] (pl.west);

\draw[qdoublearrow, shorten >=10pt, shorten <=10pt]
(hyp.south) to[out=-60,in=25]
node[lab, midway, right, text width=3.4cm, align=center]
{ \\[-.2em] }
(pl.east);

\end{tikzpicture}
\caption{The main quantum-game formalisms compared in this paper through associated operator spaces; \(\omega\)-labelled arrows also preserve the value theory.}
\label{fig:game-formalisms}
\end{figure}

\subsection{Contents} The paper is organised as follows. In  Section~2 we study projection-test
quantum games and relate their perfect strategies to winning
transformation spaces. In the finite-dimensional von
Neumann algebra setting, we show that these winning spaces are precisely operator bimodules. In Section~3 we compare this picture
with projection-lattice games and probabilistic quantum hypergraphs. In  Section~4 we
translate the graphical rule operators into projections in tensor
products of finite-dimensional von Neumann algebras, and compare the
corresponding values and perfectness conditions. In Section~5 we compare different versions of synchronicity.  Finally,  in Section~6 we collect the main examples and counterexamples. The main correspondences and value-preserving identifications are summarized in Figure~\ref{fig:game-formalisms}.

\subsection{Notation and standing conventions}

Throughout the paper, Hilbert spaces are finite dimensional unless explicitly
stated otherwise. For Hilbert spaces $H$ and $K$, $ H \otimes K$ denotes their Hilbert space tensor product. Throughout the paper at least one of the Hilbert spaces in the tensor product will be finite dimensional so that the algebraic tensor product is already closed.   If \(H\) and \(K\) are Hilbert spaces, \(\B(H,K)\) denotes
the space of linear operators from \(H\) to \(K\), and \(\B(H)=\B(H,H)\). An operator space is a subspace $ \mathcal U\subseteq \B(H,K)$. We
write \(1_H\) for the identity operator on \(H\), and \(\id_E\) for the
identity map on an operator space \(E\). If \(E\subseteq H\) is a subspace,
\(P_E\) denotes the orthogonal projection onto \(E\). Inner products are
denoted by \(\langle\cdot,\cdot\rangle\) and are linear in the second
variable. For \(\xi\in K\) and \(\eta\in H\), the rank-one operator
\(\xi\eta^*\in \B(H,K)\) is given by
\[
(\xi\eta^*)(h)=\langle \eta,h\rangle \xi,
\qquad h\in H.
\]
Since we assume finite dimensions, \(\B(H)\) is spanned by its rank-one
operators. In particular, if \((e_i)_{i=1}^n\) is an orthonormal basis of
\(H\), then \(\B(H)\cong M_n\). In general, if \(X\) is a finite set, we write $ \bb C^{X}:= \bb C^{|X|}$ and denote the canonical basis by $ \{e_x:x\in X\}$. Similarly, we denote
\(M_X:=M_{|X|}\) and write $ \{\varepsilon_{x,y}: x,y\in X\}$ for the canonical matrix units. For finite sets \(X,Y\), we sometimes write \(XY\) for \(X\times Y\),
\(\mathbb C^{XY}\) for \(\mathbb C^X\otimes\mathbb C^Y\), and
\(M_{XY}\) for \(M_X\otimes M_Y\). We write \(\varepsilon_{x,y}\) for the canonical matrix
units in \(M_X\). Similarly, \(D_n\) denotes the diagonal algebra and, for a
finite set \(X\), \(D_X:=D_{|X|}\). The conjugate Hilbert space of \(H\) is denoted by \(\overline H\). Its
elements are written \(\overline\xi\), \(\xi\in H\), with scalar multiplication
\( 
\lambda\overline\xi=\overline{\overline\lambda \xi}.
\) A state in a Hilbert space \(H\) is a unit
vector. For \(H=\bb C^n\), we write
\(  
\Omega
=
n^{-1/2}\sum_{i=1}^n e_i\otimes \overline e_i
\in \bb C^n\otimes \overline{ \bb C^n}
\)
for the maximally entangled state.

We write 
\(S^1(H)\) for the space of trace-class operators. The usual unnormalised trace on
\(S^1(H)\) is denoted by \(\Tr\). 
The vectorisation map
\[
\operatorname{vec}(\xi\eta^*)=\xi\otimes\overline\eta,
\]
is the isometric isomorphism between the Hilbert-Schmidt operators and the Hilbert space tensor product. In particular, in finite dimensions we shall also regard it
as a linear isomorphism $\operatorname{vec}:\B(H,K)\rightarrow K\otimes\overline H$.
Quantum channels are written in the Schrödinger picture. Thus a channel from
\(H\) to \(K\) is a completely positive trace-preserving map
\( 
\Gamma:S^1(H)\to S^1(K).
\)

When \(H\) is finite dimensional, we often identify \(S^1(H)\)
and \(\mathcal B(H)\) as vector spaces, while keeping track of their
different roles. In particular, for a finite set \(X\), we continue to write
\(M_X\) for the matrix space \(\mathcal B(\mathbb C^X)\), also viewed as the
trace-class space \(S^1(\mathbb C^X)\) through this identification.

If \(\mathcal A\) is a \(C^*\)-algebra, its opposite \(C^*\)-algebra is
denoted by \(\mathcal A^{\op}\). As a vector space and an involutive normed
space it is the same as \(\mathcal A\), but the multiplication is reversed.
We write \(x^{\op}\in\mathcal A^{\op}\) for the element corresponding to
\(x\in\mathcal A\); thus
\( 
x^{\op}y^{\op}=(yx)^{\op}\) and
\((x^{\op})^*=(x^*)^{\op}.
\)
When no confusion arises, we suppress the symbol \({}^{\op}\) in the
elements of \(\A^{\op}\).

A von Neumann algebra is a  \(*\)-subalgebra
\(\mathcal M\subseteq \B(H)\) such that $ \cl M= \cl M''$.  Its
commutant is
\[
\mathcal M'
=
\{a\in \B(H):am=ma\text{ for all }m\in\mathcal M\}.
\]
All von Neumann algebras appearing in the sequel are assumed to be non-degenerate.
We denote the positive cone of \(\mathcal M\) by \(\mathcal M_+\) and its projection lattice by $\mathcal P(\mathcal M)$. If $ \M \subseteq \B(H)$, $ \N \subseteq \B(K)$ are von Neumann algebras, $ \M \bar \otimes \N$ denotes their normal spatial tensor product.  Throughout the paper, at least one of the von Neumann algebras appearing in the tensor product will be finite dimensional, so the algebraic tensor product is already closed and thus coincides with the normal spatial one. Hence we suppress the bar. 
Similarly, for Hilbert spaces.

\section{Projection-test quantum  games} \label{sec:projection}

\subsection{Cooperative quantum games and operator spaces} \label{subsec:projectiontest}

Quantum-input/quantum-output variants of non-local games have been studied in several forms. The model considered here follows Leung--Toner--Watrous, where a two-player cooperative quantum game is specified by a referee who sends quantum registers to the players, receives quantum registers in return, and applies a two-outcome measurement to the joint returned system \cite{LTW13}. Variants of this framework have also been studied in the rank-one setting \cite{CJPP15}, in finite-rank quantum games \cite{CLTT23}, and more recently in \cite{AJP26}.

Here, we focus on the input-output transformation viewpoint.  We suppress the
particular physical form of the players' local operations and retain the
linear transformation tested by the referee.  This leads first to the
operator-space model following the description of \cite{LTW13}, and then to its
finite-dimensional von Neumann algebraic generalisation. In this model, there are finite-dimensional input Hilbert spaces \(H_X,H_Y\), output
Hilbert spaces \(H_A,H_B\), and a referee space \(H_R\). Thus we write $H_{\rm in}=H_X \otimes H_Y$, and $H_{\rm out}= H_A \otimes H_B$ 
and view the referee's state as
\( 
\psi\in H_{\rm in}\otimes H_R.
\)
The accepting measurement is represented by a projection
\( 
Q\in\mathcal P(\B(H_{\rm out}\otimes H_R)).
\)
A \emph{projection-test quantum game} is therefore a pair
\[
G=(\psi,Q),
\qquad
\psi\in H_{\rm in}\otimes H_R,
\qquad
Q\in\mathcal P(\B(H_{\rm out}\otimes H_R)).
\]
The game is played as follows.
\begin{enumerate}
    \item The referee prepares the input state
    \[
    \psi\in H_{\rm in}\otimes H_R.
    \]

    \item The referee sends the \(H_{\rm in}\)-register to the players and
    keeps the \(H_R\)-register.

    \item The players apply a transformation according to some strategy
    \[
    T:H_{\rm in}\to H_{\rm out}.
    \]
    Since they do not act on the referee's private register, the total
    transformation is
    \[
    T\otimes 1_{H_R}:
    H_{\rm in}\otimes H_R\to H_{\rm out}\otimes H_R.
    \]

    \item The referee receives the \(H_{\rm out}\)-register and measures 
    \[
    (T\otimes 1_{H_R})\psi
    \]
    with the two-outcome measurement \(\{Q,Q^\perp\}\).

    \item The players win precisely when the outcome lies in the accepting
    subspace \(Q(H_{\rm out}\otimes H_R)\). 
\end{enumerate}
The operator space of the winning/accepted transformations, namely, the \emph{winning transformation space} is
\[
\mathcal U_{\psi,Q}
=
\{T\in \B(H_{\rm in},H_{\rm out}):
Q^\perp(T\otimes 1_{H_R})\psi=0\}.
\]

\begin{proposition}\label{prop:operator-space-projection-test}
Let \(H_{\rm in}\) and \(H_{\rm out}\) be finite-dimensional Hilbert spaces,
and let \( \mathcal U\subseteq \B(H_{\rm in},H_{\rm out})\) be an operator
space. Then \(\mathcal U\) is the winning transformation space of a quantum game. More precisely, there are a referee Hilbert space \(H_R\), a
unit vector \(\psi\in H_{\rm in}\otimes H_R\), and a projection
\(Q \in \B(H_{\rm out}\otimes H_R)\) such that
\( 
\mathcal U_{\psi,Q_{}}=\mathcal U.
\)
\end{proposition}

\begin{proof}
Let \(d=\dim H_{\rm in}\), choose an orthonormal basis
\((e_i)_{i=1}^d\) of \(H_{\rm in}\), and set
\(H_R=\overline{H_{\rm in}}\). Put
\( 
\psi=\Omega =d^{-1/2}\sum_{i=1}^d e_i\otimes \overline e_i
\in H_{\rm in}\otimes\overline{H_{\rm in}},
\)
to be the maximally entangled state.
We use the vectorisation
\( 
\operatorname{vec}:\B(H_{\rm in},H_{\rm out})
\rightarrow
H_{\rm out}\otimes\overline{H_{\rm in}}\), \( 
\operatorname{vec}(T)=\sum_{i=1}^d Te_i\otimes\overline e_i .
\)
Then
\[
(T\otimes 1_{H_R})\psi
=
d^{-1/2}\operatorname{vec}(T).
\]
Let \(Q\) be the projection of
\(H_{\rm out}\otimes\overline{H_{\rm in}}\) onto
\(\operatorname{vec}(\mathcal U)\). Then
\( 
Q^{\perp}(T\otimes 1_{H_R})\psi=0
\)
if and only if
\( 
\operatorname{vec}(T)\in\operatorname{vec}(\mathcal U).
\)
Since vectorisation is injective, this is equivalent to \(T\in\mathcal U\).
Hence \(\mathcal U_{\psi,Q}=\mathcal U\).
\end{proof}

Operationally, the projection-test realisation of an operator space
\(\mathcal U\subseteq \B(H_{\rm in},H_{\rm out})\) is a vectorised membership
test for transformations.  
The maximally entangled input state ensures that no tested information is lost: applying \(T\) to one leg produces the vectorisation of the whole transformation \(T\).  The accepting projection \(Q=P_{\operatorname{vec}(\mathcal U)}\) then tests whether this vectorised transformation lies in the prescribed operator space. Hence, the projection-test game accepts exactly those transformations \(T\) that belong to \(\mathcal U\).

Let \(H\) and \(K\) be finite-dimensional Hilbert spaces. Every vector
\(\psi\in H\otimes K\) admits a Schmidt decomposition 
\[
\psi=\sum_{i=1}^r s_i\,\xi_i\otimes\eta_i,
\qquad s_i>0,
\]
where \((\xi_i)_i\) and \((\eta_i)_i\) are orthonormal families in \(H\) and
\(K\), respectively. The integer \(r\) is called the \emph{Schmidt rank} of
\(\psi\).

If \(K = \overline{H}\),  we say that \(\psi\) has
\emph{full Schmidt rank} 
 if it admits a Schmidt decomposition
\( 
\psi=\sum_{i=1}^{\dim H}
s_i\,\xi_i\otimes\overline{\eta_i}\), with  \( s_i>0,
\)
where \((\xi_i)_i\) and \((\eta_i)_i\) are orthonormal bases of \(H\).

\begin{lemma}
Let \(H\) be  a finite-dimensional Hilbert space.
A vector \(\psi\in H\otimes\overline H\) has full Schmidt rank if and only if
there exists an invertible operator \(D\in \B(H)\) such that
\[
\psi=(D\otimes 1_{\overline H})\Omega .
\]
\end{lemma}

\begin{proof}
Every vector \(\psi\in H\otimes\overline H\) can be written uniquely as
\( 
\psi=\sum_{i,j}a_{ij}e_i\otimes\overline e_j .
\)
Define \(D\in \B(H)\) by
\( 
D e_j=(\dim H)^{1/2}\sum_i a_{ij}e_i .
\)
Then
\[
(D\otimes 1)\Omega
=
(\dim H)^{-1/2}\sum_j D e_j\otimes\overline e_j
=
\sum_{i,j}a_{ij}e_i\otimes\overline e_j
=
\psi .
\]
Thus, every vector has the form \((D\otimes1)\Omega\). It remains to identify full Schmidt rank.
Let
\( 
D=\sum_{k=1}^r s_k u_kv_k^*
\)
be a singular value decomposition, where \(r=\operatorname{rank}D\) and
\(s_k>0\). 
We have
\[
(D\otimes1)\Omega
=
(\dim H)^{-1/2}\sum_{k=1}^r s_k u_k\otimes\overline{v_k}.
\]
This is a Schmidt decomposition, up to the normalising scalar. Hence, the
Schmidt rank of \(\psi\) is equal to \(\operatorname{rank}D\). Therefore
\(\psi\) has full Schmidt rank if and only if \(D\) is invertible.
\end{proof}

\begin{remark} \label{rmk:fullrank} \rm
The maximally entangled vector in Proposition
\ref{prop:operator-space-projection-test} is only a canonical choice. It suffices to pick a vector with full Schmidt rank. Indeed, let 
\( 
\psi\in H_{\rm in}\otimes \overline{H_{\rm in}}
\)
be a unit vector with full Schmidt rank. The previous lemma gives an invertible
operator \(D\in \B(H_{\rm in})\) such that
\[
\psi=(D\otimes 1)\Omega.
\]
Thus, for every \(T\in \B(H_{\rm in},H_{\rm out})\),
\[
(T\otimes1)\psi
=
(TD\otimes1)\Omega
=
(\dim H_{\rm in})^{-1/2}\operatorname{vec}(TD).
\]
The test with input state \(\psi\) therefore sees \(T\) through the operator
\(TD\). Since \(D\) is invertible, no information about \(T\) is lost. 
Consequently, if
\(\mathcal U\subseteq \B(H_{\rm in},H_{\rm out})\) is an operator space and
\(Q_{\psi,\mathcal U}\) is the projection onto
\( 
\operatorname{vec}(\mathcal U D)
=
\{\operatorname{vec}(SD):S\in\mathcal U\},
\)
then
\( 
Q_{\psi,\mathcal U}^{\perp}(T\otimes1)\psi=0\)   if and only if  \(\operatorname{vec}(TD)\in\operatorname{vec}(\mathcal U D)\)
if and only if \(TD\in \mathcal U D\), and since \(D\) is invertible this is
equivalent to \(T\in\mathcal U\). Hence
\( 
\mathcal U_{\psi,Q_{\psi,\mathcal U}}=\mathcal U.
\)
Thus, any full-Schmidt-rank input state gives a realisation of the operator space as a winning transformation space. We note that full Schmidt rank is the quantum analogue of full support of a classical
question distribution, while the maximally entangled state plays the role of the uniform distribution in the classical case. 
\end{remark}

We conclude that the realisation of an operator space as a projection-test game is not unique.  As we saw in the previous remark one may
replace the maximally entangled vector by any full-Schmidt-rank input vector,
provided the accepting projection is adjusted accordingly.    The same operator space may
therefore admit many different projection-test presentations.


\subsection{Classical non-local games}\label{subsec:class}

Let \(G=(X,Y,A,B,\lambda)\) be a classical non-local game, and let \(\pi\)
be a probability distribution on \(X\times Y\) with full support. Put
\( 
H_{\rm in}=\mathbb C^X\otimes\mathbb C^Y\),
 \(H_{\rm out}=\mathbb C^A\otimes\mathbb C^B.
\)
The corresponding quantum game uses a referee register
\( 
H_R=\mathbb C^X\otimes\mathbb C^Y,
\)
which stores a coherent copy of the classical question pair. We place this
register on the right leg and set
\[
\psi_\pi
=
\sum_{x\in X,\ y\in Y}
\sqrt{\pi(x,y)}\,
(e_x\otimes e_y)_{\rm in}\otimes (e_x\otimes e_y)_{R}
\in H_{\rm in}\otimes H_R.
\]

For each \((x,y)\in X\times Y\), let \(P_{x,y}\in \B(H_{\rm out})\) be the
projection onto the winning answer subspace
\( 
\operatorname{span}\{e_a\otimes e_b:\lambda(a,b,x,y)=1\}.
\)
The accepting projection is
\[
Q_\lambda
=
\sum_{x\in X,\ y\in Y}
P_{x,y}\otimes \varepsilon_{(x,y),(x,y)}
\in \B(H_{\rm out}\otimes H_R).
\]
Thus, \(Q_\lambda\) is block diagonal with respect to the referee's classical
question register. In the \((x,y)\)-block, it tests whether the produced
answer lies in the winning answer subspace for that same question.

Let \(T\in \B(H_{\rm in},H_{\rm out})\). Then
\[
Q_\lambda^\perp(T\otimes 1_{H_R})\psi_\pi
=
\sum_{x,y}\sqrt{\pi(x,y)}\,
P_{x,y}^\perp T(e_x\otimes e_y)
\otimes(e_x\otimes e_y)_{R}.
\]
Since the vectors \((e_x\otimes e_y)_{H_R}\) are mutually orthogonal and
since \(\pi\) has full support, we have
\[
Q_\lambda^\perp(T\otimes 1_{H_R})\psi_\pi=0
\quad\Longleftrightarrow\quad
P_{x,y}^\perp T(e_x\otimes e_y)=0
\quad\text{for all }(x,y)\in X\times Y.
\]
Equivalently,
\( 
T(e_x\otimes e_y)\in P_{x,y}H_{\rm out}\)
for all \((x,y)\in X\times Y.
\)

By the previous equivalence, the winning transformation space of the quantum game is therefore
\[
\mathcal U_{\psi_\pi,Q_\lambda}
=
\{T\in \B(H_{\rm in},H_{\rm out}):
T(e_x\otimes e_y)\in P_{x,y}H_{\rm out}
\text{ for all }(x,y)\in X\times Y\}.
\]
Hence, this is precisely the operator space associated with the winning
relation:
\[
\mathcal U_{\psi_\pi,Q_\lambda}
=
\operatorname{span}
\{\varepsilon_{(a,b),(x,y)}:\lambda(a,b,x,y)=1\}.
\]
Indeed, the condition above says exactly that for any $ T \in \mathcal U_{\psi_\pi,Q_\lambda}$ the matrix coefficient 
\(\langle e_a\otimes e_b,T(e_x\otimes e_y)\rangle\) vanishes whenever
\(\lambda(a,b,x,y)=0\).



\subsection{Operator-algebraic generalisation} \label{subsec:abst_proj}

The quantum game formalism is compatible with finite-dimensional von Neumann algebraic systems once rank-one vector tests are replaced by general projection tests. Operationally, this allows the players' systems to be modelled by arbitrary finite-dimensional von Neumann algebras, rather than only full matrix algebras. In particular, the formalism includes quantum systems with superselection rules, equivalently classical mixtures of quantum systems \cite{Wea21,KMP04}.

Let
\(\mathcal X,\mathcal Y,\mathcal A,\mathcal B\) be finite-dimensional von
Neumann algebras acting on Hilbert spaces
\(H_X\), \(H_Y\), \(H_A\), \(H_B\), respectively.  Let also \(\cl R \subseteq \B(H_R)\) be a possibly infinite-dimensional von Neumann algebra. An $\cl R$-\emph{projection-test quantum game} from $(\mathcal X,\cl Y)$ to $ (\cl A,\cl B)$ is a pair $ G= (P,Q)$, where 
\[
P\in\mathcal P((\cl X \otimes \cl Y) \otimes \cl R),
\qquad
Q\in\mathcal P((\cl A \otimes \cl B)\otimes \cl R).
\]

The projection \(P\) specifies the allowed input subspace, while \(Q\)
specifies the accepting output subspace. In the generalised version,  the \emph{winning transformation space} is
\[
\mathcal U_{P,Q}
=
\{T\in \B(H_{X} \otimes H_Y,H_{A} \otimes H_B):
Q^\perp(T\otimes 1_{\cl R})P=0\}.
\]

When $ \cl X= \B(H_{X})$, $ \cl Y= \B(H_Y)$, $ \cl A= \B(H_A)$, $ \cl B= \B(H_{B})$,   $ \cl R= \B(H_R)$ and $ P= \psi\psi^*$ for a unit vector $ \psi \in (H_{X} \otimes H_Y) \otimes H_{R}$ this reduces to the projection-test quantum game $ G=(\psi,Q)$ introduced in Subsection \ref{subsec:projectiontest}.

When $\cl X= D_X$ and $ \cl Y= D_Y$ we say that the $\cl R$-projection-test quantum game $G$ is  \emph{classical-to-quantum}; when $ \cl A=D_A$ and $\cl B= D_B$ we say that $G$ is  \emph{quantum-to-classical}; when both the latter hold we say that $G$ is  \emph{classical}.

As we saw in Subsection~\ref{subsec:class}, every classical non-local game gives rise to a 
projection-test quantum game whose winning transformation space consists of the matrix units supported on the rule function. In particular, we may choose a realisation as a classical $\cl D_{X\times Y}$-projection-test quantum game. Indeed, let \(G=(X,Y,A,B,\lambda)\) be a classical non-local game.
For \((x,y)\in X\times Y\), let \(P_{x,y}\in D_A\otimes D_B\) be the projection onto
\( 
    \operatorname{span}
    \left\{
        e_a\otimes e_b:
        \lambda(a,b,x,y)=1
    \right\}.
\)
We may then choose
\[
    P_{\mathrm{cl}}
    =
    \sum_{x\in X,\,y\in Y}(\varepsilon_{x,x}\otimes\varepsilon_{y,y})\otimes \varepsilon_{(x,y),(x,y)}
    \in
    (D_X\otimes D_Y)\otimes D_{X\times Y}
\]
and
\[
    Q_\lambda
    =
    \sum_{x\in X,\,y\in Y}
        P_{x,y}\otimes \varepsilon_{(x,y),(x,y)}    \in
    (D_A\otimes D_B)\otimes D_{X\times Y}.
\]
For \(T\in B(H_{\mathrm{in}},H_{\mathrm{out}})\), we have
\[
    Q_\lambda^\perp
    (T\otimes 1_{\mathcal R})P_{\mathrm{cl}}
    =
    \sum_{x,y}
        P_{x,y}^{\perp}T(\varepsilon_{x,x}\otimes\varepsilon_{y,y})\otimes \varepsilon_{(x,y),(x,y)}.
\]
Since the projections \(\varepsilon_{(x,y),(x,y)}\) are mutually orthogonal, it follows
that
\[
    Q_\lambda^\perp
    (T\otimes 1_{\mathcal R})P_{\mathrm{cl}}=0
    \quad\Longleftrightarrow\quad
    P_{x,y}^{\perp}T(e_x\otimes e_y)=0
    \quad
    \text{for all }(x,y)\in X\times Y.
\]
Consequently,
\[
    U_{P_{\mathrm{cl}},Q_\lambda}
    =
    \operatorname{span}
    \left\{
        \varepsilon_{(a,b),(x,y)}:
        \lambda(a,b,x,y)=1
    \right\},
\]
where the last equality holds for every full-support distribution
\(\pi\).
Thus every classical non-local game admits a classical
\(D_{X\times Y}\)-projection-test realisation that preserves both its
winning transformation space and its distribution-dependent value. 
One may verify that the converse is also true.

\medskip

Let \(H\) be a finite-dimensional Hilbert space.  If
\(\mathcal M\subseteq \B(H)\) is a  von Neumann algebra, we represent
its opposite algebra on the conjugate Hilbert space \(\overline H\) by
\[
a^{\op}\,\overline\xi=\overline{a^*\xi},
\qquad a\in\mathcal M,\ \xi\in H.
\]
Thus \(\mathcal M^{\op}\subseteq \B(\overline H)\).  With this convention, one
has
\( 
(\mathcal M^{\op})'=(\mathcal M')^{\op}\)
on \(\overline H.
\)
Consequently, if
\(\mathcal M\subseteq \B(H)\) and \(\mathcal N\subseteq \B(K)\), then \( \mathcal N\otimes\mathcal M^{\op} \subseteq \B(K \otimes \overline{H})\) and 
\( 
\big(\mathcal N\otimes\mathcal M^{\op}\big)'
=
\mathcal N'\otimes(\mathcal M')^{\op}
\)
by Tomita's theorem \cite[Theorem 5.9, Chapter IV]{TakesakiI}.

The following fact is standard (see for example \cite{Watrous2018}). We include a proof for completeness.
\begin{lemma}\label{lem:vec-bimodule-action-general}
Let \(H,K\) be finite-dimensional Hilbert spaces, and let
\( 
\operatorname{vec}:\B(H,K)\rightarrow K\otimes\overline H
\) 
be the canonical vectorisation, determined by
\( 
\operatorname{vec}(\xi\eta^*)=\xi\otimes\overline\eta\), \(\xi\in K,\ \eta\in H.
\) 
For \(a\in \B(H)\), let \(a^{\op}\in \B(\overline H)\) act by
\( 
a^{\op}\overline\eta=\overline{a^*\eta}.
\)
Then, for every \(b\in \B(K)\), \(a\in \B(H)\), and
\(T\in \B(H,K)\), one has
\[
\operatorname{vec}(bTa)
=
(b\otimes a^{\op})\operatorname{vec}(T).
\]
\end{lemma}

\begin{proof}
It suffices to verify the identity for rank-one operators
\(T=\xi\eta^*\).  For \(h\in H\),
\[
b(\xi\eta^*)a(h)
=
b\big(\langle \eta,ah\rangle \xi\big)
=
\langle a^*\eta,h\rangle b\xi.
\]
Hence,
\( 
b(\xi\eta^*)a=(b\xi)(a^*\eta)^*.
\)
Applying vectorisation gives
\[
\operatorname{vec}\big(b(\xi\eta^*)a\big)
=
b\xi\otimes\overline{a^*\eta}
=
(b\otimes a^{\op})(\xi\otimes\overline\eta)
=
(b\otimes a^{\op})\operatorname{vec}(\xi\eta^*).
\]
The result follows by linearity.
\end{proof}

We show the analogue of Proposition \ref{prop:operator-space-projection-test} for operator spaces carrying a bimodule structure. In the following think of the underlying spaces as $H_{\rm in} := H_X \otimes H_Y$, $ H_{\rm out}:= H_A \otimes H_B$, \( \cl M: =\cl X \otimes \cl Y \subseteq \B(H_{\rm in})\) and $ \cl N:= \cl A \otimes \cl B \subseteq \B(H_{\rm out})$, although the result holds in general.

\begin{theorem} \label{thm:abstract-projection-test}
Let \(\mathcal M\subseteq \B(H_{\rm in})\) and
\(\mathcal N\subseteq \B(H_{\rm out})\) be finite-dimensional von Neumann
algebras. A subspace
\(\mathcal U\subseteq \B(H_{\rm in},H_{\rm out})\) is a
\(\mathcal N'\)-\(\mathcal M'\) bimodule if and only if there exist a von Neumann algebra \( \cl R\subseteq \B(H_R)\) and projections
\( 
P\in\mathcal P(\mathcal M\otimes \cl R) \), and
\(Q\in\mathcal P(\mathcal N\otimes \cl R),
\)
such that $ \mathcal U= \mathcal U_{P,Q}$. In particular, in the converse direction one may choose
\[
\cl R= \cl M^{\rm op},
\qquad
Q=P_{\operatorname{vec}(\mathcal U)}, \qquad P= P_{\operatorname{vec}(\mathcal M')}.
\]
\end{theorem}
\begin{proof}
Suppose first that \(P\in\mathcal M\otimes \cl R\) and
\(Q\in\mathcal N\otimes \cl R\).  If
\(a\in\mathcal M'\), \(b\in\mathcal N'\), and
\(T\in\mathcal U_{P,Q}\), then \(P\) commutes with
\(a\otimes 1_{\cl R}\), while \(Q^\perp\) commutes with
\(b\otimes 1_{\cl R}\).  Hence,
\[
Q^\perp((bTa)\otimes 1_{\cl R})P
=
(b\otimes 1_{\cl R})Q^\perp(T\otimes 1_{\cl R})P(a\otimes 1_{\cl R})
=
0.
\]
Thus, \(bTa\in\mathcal U_{P,Q}\), and so
\(\mathcal U_{P,Q}\) is an
\(\mathcal N'\)-\(\mathcal M'\) bimodule.

Conversely, suppose that
\(\mathcal U\subseteq \B(H_{\rm in},H_{\rm out})\) is a
\(\mathcal N'\)-\(\mathcal M'\) bimodule.  Set
\(\cl R = \cl M^{\rm op}\), and let
\( 
P=P_{\operatorname{vec}(\mathcal M')} \) and
\( Q=P_{\operatorname{vec}(\mathcal U)}.
\)
We first check that these projections belong to the required algebras.  
If \(x'\in\mathcal N'\), \(y'\in\mathcal M'\), and \(T\in\mathcal U\), then by Lemma \ref{lem:vec-bimodule-action-general}
\[
(x'\otimes (y')^{\op})\operatorname{vec}(T)
=
\operatorname{vec}(x'Ty')
\in
\operatorname{vec}(\mathcal U),
\]
because \(\mathcal U\) is an \(\mathcal N'\)-\(\mathcal M'\) bimodule.  Hence
\(\operatorname{vec}(\mathcal U)\) is invariant under the finite-dimensional
\(*\)-algebra
\(\mathcal N'\otimes(\mathcal M')^{\op}\).  Therefore it is reducing for this
algebra, and its orthogonal projection \(Q=P_{\operatorname{vec}(\mathcal U)}\)
commutes with \(\mathcal N'\otimes(\mathcal M')^{\op}\).  Thus,
\[ 
Q\in
\big(\mathcal N'\otimes(\mathcal M')^{\op}\big)'
=
\mathcal N\otimes\mathcal M^{\op}
\subseteq
\mathcal N\otimes \B(\overline{H_{\rm in}}).
\]
Similarly, since \(\mathcal M'\subseteq B(H_{\rm in})\) is an
\(\mathcal M'\)-\(\mathcal M'\) bimodule, the same argument gives
\[ 
P=P_{\operatorname{vec}(\mathcal M')}
\in
\big(\mathcal M'\otimes(\mathcal M')^{\op}\big)'
=
\mathcal M\otimes\mathcal M^{\op}
\subseteq
\mathcal M\otimes \B(\overline{H_{\rm in}}).
\]
It remains to show that \(\mathcal U=\mathcal U_{P,Q}\). Let \(T\in \B(H_{\rm in},H_{\rm out})\).  Since the range of \(P\) is
\(\operatorname{vec}(\mathcal M')\) and the range of \(Q\) is
\(\operatorname{vec}(\mathcal U)\), the condition
\(Q^\perp(T\otimes 1)P=0\) is equivalent to
\[
(T\otimes 1)\operatorname{vec}(a)\in\operatorname{vec}(\mathcal U),
\qquad a\in\mathcal M'.
\]
Using the vectorisation identity
\( 
(T\otimes 1)\operatorname{vec}(a)=\operatorname{vec}(Ta)\), \(a\in\mathcal M',
\)
this is equivalent to
\( 
Ta\in\mathcal U\),
\( a\in\mathcal M'.
\)
If this holds, then taking \(a=1\) gives \(T\in\mathcal U\).  Conversely, if
\(T\in\mathcal U\), then \(Ta\in\mathcal U\) for every
\(a\in\mathcal M'\), because \(\mathcal U\) is a right
\(\mathcal M'\)-module.  Hence \(Q^\perp(T\otimes 1)P=0\).

Therefore, \(\mathcal U_{P,Q}=\mathcal U\), with the stated choices of
\(\cl R\), \(P\), and \(Q\).  The proof is complete.
\end{proof}

If \(\mathcal M\subseteq \B(H)\) is a finite-dimensional von Neumann algebra, we shall write
\( 
\Pi_{\mathcal M}:=
P_{\operatorname{vec}(\mathcal M')}
\in \mathcal M\otimes\mathcal M^{\op}
\subseteq \B(H\otimes\overline H),
\)
where \(P_{\operatorname{vec}(\mathcal M')}\) denotes the orthogonal projection
onto the subspace
\( 
\operatorname{vec}(\mathcal M')\subseteq H\otimes\overline H.
\)

\begin{remark}\label{rem:many-projection-test-realisations} \rm
Theorem~\ref{thm:abstract-projection-test} recovers
Proposition~\ref{prop:operator-space-projection-test} as the special case
where 
\( 
\mathcal M=\B(H_{\rm in})\),
\( \mathcal N=\B(H_{\rm out})
\) and $ \cl R= \M^{\rm op} = \cl B(\overline H_{\rm in})$.
Indeed, then
\(\mathcal M'=\mathbb C1_{H_{\rm in}}\) and
\(\mathcal N'=\mathbb C1_{H_{\rm out}}\), so that every subspace
\(\mathcal U\subseteq \B(H_{\rm in},H_{\rm out})\) is automatically an
\(\mathcal N'\)-\(\mathcal M'\) bimodule.  The construction in the theorem
takes \(H_R=\overline{H_{\rm in}}\) and
\( 
P=P_{\operatorname{vec}(\mathbb C1)}
=
\Omega\Omega^*,
\)
where
\( 
\Omega=(\dim H_{\rm in})^{-1/2}\operatorname{vec}(1_{H_{\rm in}})
\in H_{\rm in}\otimes\overline{H_{\rm in}}
\)
is the maximally entangled state.  It also takes
\(Q=P_{\operatorname{vec}(\mathcal U)}\).  Thus, the condition
\( 
Q^\perp(T\otimes 1_{\overline{H_{\rm in}}})P=0
\)
is equivalent to
\(
Q^\perp(T\otimes 1_{\overline{H_{\rm in}}})\Omega=0,
\)
and this is precisely the projection-test realisation of
\(\mathcal U\) from Proposition~\ref{prop:operator-space-projection-test}.
\end{remark}

\subsection{Strategies and winning transformations}
\label{subs:strategies_coop}

Strategies for general projection-test games are quantum channels between preduals, and perfectness is expressed in terms of support projections. We recall the relevant definitions.  If \(\mathcal M\) is a von Neumann algebra, its predual \(\mathcal M_*\) is
the Banach space of normal linear functionals on \(\mathcal M\). It is
characterised by the canonical identification
\( 
(\mathcal M_*)^*=\mathcal M.
\)
The positive cone \(\mathcal M_*^+\) consists of the normal positive
functionals, and a normal state is an element \(\rho\in\mathcal M_*^+\) such
that \(\rho(1_{\mathcal M})=1\). In finite dimensions, every linear functional
is normal, so \(\mathcal M_*=\mathcal M^*\). We write
the duality between \(\mathcal M_*\) and \(\mathcal M\) as
\[
\langle \rho,x\rangle=\rho(x),
\qquad
\rho\in\mathcal M_*,
\quad
x\in\mathcal M.
\]
 If \(\mathcal M\) is a von Neumann algebra and
\(\rho\in\mathcal M_*^+\), its support projection \(s(\rho)\) is the smallest
projection \(e\in\mathcal M\) such that \(\rho=\rho e=e\rho\), where
\( 
(\rho e)(x)=\rho(ex)\) and
\((e\rho)(x)=\rho(xe).
\)
Thus \(s(\rho)\leq P\) if and only if \(\rho=\rho P=P\rho\).
We say that a normal state \(\rho\) is supported in \(P\) if
\(s(\rho)\leq P\).

In the operational formulation, a strategy is a
quantum channel, that is, a completely positive, trace-preserving map
\[
\Gamma: S^1(H_{\rm in}) \to  S^1(H_{\rm out}),
\]
usually required to belong to a prescribed resource class, such as local
operations with shared entanglement, one-way communication, unrestricted
joint operations, or a no-signalling class. The referee evaluates \(\Gamma\)
by applying \(\Gamma\otimes{\rm id}_{S^1(H_R)}\) to the input state  and then
measuring the output with the accepting projection. 
 
For a projection-test \(G=(\psi,Q)\), the winning probability is
\[
\omega_G(\Gamma)
=
\left\langle
(\Gamma\otimes{\rm id}_{S^1(H_R)})(\psi\psi^*),Q
\right\rangle .
\]
If \(\Gamma(\rho)=\sum_i T_i\rho T_i^*\), then
\[
(\Gamma\otimes{\rm id}_{S^1(H_R)})(\psi\psi^*)
=
\sum_i
(T_i\otimes1)\psi\psi^*(T_i\otimes1)^* .
\]
Thus, \(\Gamma\) wins perfectly, i.e., $ \omega_{G}(\Gamma)=1$, if and only if
\(Q^\perp(T_i\otimes1)\psi=0\) for every \(i\). Equivalently, if
\(\mathcal X_\Gamma=\operatorname{span}\{T_i:i\}\) denotes the Kraus space of $ \Gamma$,
then
\[
\Gamma\text{ wins perfectly for }G=(\psi,Q)
\quad\Longleftrightarrow\quad
\mathcal X_\Gamma\subseteq\mathcal U_{\psi,Q}.
\]
For a proof of the claims above  see \cite[Lemma 5.2]{Hoefer25}.


We now extend this to the setting of general finite-dimensional von Neumann algebras $\M\subseteq \B(H_{\rm in})$ and $\N \subseteq \B(H_{\rm out})$.
Strategies for quantum input-output games over finite quantum spaces were
formulated in the Heisenberg picture by Bochniak--Kasprzak--Sołtan
\cite{BKS23}. We use the equivalent Schrödinger-picture predual formulation.
Thus, a \emph{quantum channel} is a completely positive state-preserving map
\[
\Gamma:\mathcal M_*
\to
\mathcal N_* .
\]
Equivalently, its dual is a normal unital completely positive map
\(\Gamma^*:\mathcal N\to  \mathcal M\),
defined by
\[
\langle\Gamma(\rho),z\rangle
=
\langle\rho,\Gamma^*(z)\rangle .
\]
 When $ \M = \B(H)$, we identify the predual of \(\mathcal B(H)\) with
\(S^1(H)\) through the pairing
\( 
\langle \rho,x\rangle=\Tr(\rho x)\),
\( \rho\in S^1(H)\), \(x\in \mathcal B(H)
\).

We shall use the following Kraus-space notation. Let
\(\mathcal M\subseteq \B(H_{\rm in})\) and
\(\mathcal N\subseteq \B(H_{\rm out})\) be finite-dimensional von Neumann
algebras, and let \(\Gamma:\mathcal M_*\to\mathcal N_*\) be a quantum channel. If
\( 
\Gamma^*(y)=\sum_{i=1}^r v_i^*yv_i\), \(y\in\mathcal N,
\)
then the predual channel is determined by
\[
\langle \Gamma(\rho),y\rangle
=
\langle \rho,\sum_{i=1}^r v_i^*yv_i\rangle,
\qquad
\rho\in\mathcal M_*,\ y\in\mathcal N .
\] 
With a slight abuse of notation we write $ \Gamma(\rho)=\sum_{i=1}^r v_i\rho v_i^*$ for the above expression.
The Kraus bimodule of \(\Gamma\) is the
\(\mathcal N'\)-\(\mathcal M'\) bimodule
\[
\mathcal X_\Gamma
=
\operatorname{span}\{n'v_i m':
n'\in\mathcal N',\ m'\in\mathcal M',\ 1\leq i\leq r\}
\subseteq \B(H_{\rm in},H_{\rm out}).
\]
In the notation of \cite{Daws26}, this bimodule is precisely the quantum relation associated with the adjoint 
map $\Gamma^*$. In particular, it is independent of the chosen
Kraus decomposition; see also \cite[Proposition~6.1]{KL26}.

Now consider an $\cl R$-projection-test quantum game \(G=(P,Q)\), and projections
\( 
P\in\mathcal P(\mathcal M\otimes \cl R)\) and
\( 
Q\in\mathcal P(\mathcal N\otimes \cl R).
\)
Given a strategy $\Gamma$ and a normal state $ \rho$ supported in $ P$, the \emph{winning probability} is 
\[
\omega_{G}(\Gamma,\rho)= \big\langle(\Gamma\otimes{\rm id}_{\cl R_{*}})(\rho),Q\big\rangle.
\]
A strategy \(\Gamma\) is \emph{perfect for \(G\)} if every normal state supported in
\(P\) is sent by
\(\Gamma\otimes{\rm id}_{\cl R_{*}}\) to a normal state supported in \(Q\). Equivalently,
for every normal state \(\rho\) with \(s(\rho)\leq P\),
\[
\big\langle(\Gamma\otimes{\rm id}_{\cl R_{*}})(\rho),Q^\perp\big\rangle=0.
\]
For a Kraus decomposition \(\Gamma(\rho)=\sum_iT_i\rho T_i^*\), this is
equivalent to
\( 
Q^\perp(T_i\otimes1_{\cl R})P=0\)
for every \(i\).
Hence, the quantum channel
\(
\Gamma\)  is perfect for \( G=(P,Q)\) 
if and only if
\( \mathcal X_\Gamma\subseteq\mathcal U_{P,Q}.
\)


\section{Projection-lattice games} \label{sec:projection_lattice}

For a  von Neumann algebra \(\M\subseteq \B(H)\), we denote by
\(
\mathcal{P}(\M)
\)
its projection lattice. The order is given by \(P\leq Q\) if \(PH\subseteq QH\), and the join
\( 
\vee_i P_i
\)
is the projection onto the closed linear span of the subspaces \(P_iH\).  The meet \(P\wedge Q\) is the largest projection below both \(P\) and \(Q\), equivalently, the projection onto \(PH\cap QH\). A map 
\( 
\varphi:\mathcal{P}(\M)\to\mathcal{P}(\N)
\)
is called \emph{zero-preserving} if $\varphi(0)=0$ and \emph{join-continuous}  if 
\[
\varphi\left(\bigvee_{i\in I} P_i\right)=\bigvee_{i\in I}\varphi(P_i)
\]
for every family \((P_i)_{i\in I}\subseteq\mathcal{P}(\M)\).
Equivalently, in finite dimensions,
\( 
\varphi(P\vee Q)=\varphi(P)\vee\varphi(Q)\),
\(  P,Q\in\mathcal{P}(\M).
\)

Join-continuous zero-preserving maps were first considered by  Erdos \cite{Erdos86} in the study of  reflexive operator spaces introduced by Loginov--Shulman  \cite{LS75}. These maps are equivalent to bilattices as defined by Shulman--Turowska in \cite{ST04}.

Let
\(H_1,H_2\) be Hilbert spaces and write \(\mathcal P_i=\mathcal P(\B(H_i))\).
If \(\mathcal U\subseteq \B(H_1,H_2)\) is a subspace, define
\[
\varphi_{\mathcal U}:\mathcal P_1\to\mathcal P_2,\qquad
\varphi_{\mathcal U}(Q)=P_{[\mathcal U QH_1]},
\]
where \(P_{[\mathcal U QH_1]}\) denotes the projection onto the closed linear span of
\(\{Tx:T\in\mathcal U,\ x\in QH_1\}\). Then \(\varphi_{\mathcal U}\)
preserves arbitrary joins and sends \(0\) to \(0\). Conversely, if
\(\varphi:\mathcal P_1\to\mathcal P_2\) is zero-preserving and join-continuous,
set
\[
\mathcal{U}_{\varphi}
=
\{T\in \B(H_1,H_2):\varphi(P)^\perp TP=0
\text{ for all }P\in\mathcal P_1\}.
\]
The reflexive hull of an operator space \(\mathcal U\) is
\[
\operatorname{Ref}\mathcal U
=
\{T\in \B(H_1,H_2):T\xi\in[\mathcal U\xi]
\text{ for every }\xi\in H_1\}.
\]
A subspace $ \mathcal  U$ is called \emph{reflexive} if $ \mathcal U= \operatorname{Ref}\mathcal U$. It can be seen that \(\operatorname{Ref}\mathcal U
=\mathcal U_{\varphi_{\mathcal U}}\). Hence \(\mathcal U\) is
reflexive precisely when
\( 
\mathcal U=\mathcal U_{\varphi_{\mathcal U}}.
\)

Thus,   a subspace
\(\mathcal U\subseteq \B(H_1,H_2)\) is reflexive if and only if
\(\mathcal U=\mathcal{U}_{\varphi}\) for some zero-preserving
join-continuous map \(\varphi:\mathcal P_1\to\mathcal P_2\).

\subsection{Classical rules as lattice maps} \label{sec:class_latt}

Let \(X,Y,A,B\) be finite sets and let \(D_X,D_Y,D_A,D_B\) denote the
corresponding diagonal algebras, acting on
\(\mathbb C^X,\mathbb C^Y,\mathbb C^A,\mathbb C^B\), respectively. We shall
freely identify \(D_X\otimes D_Y\) with \(D_{X\times Y}\). Thus every
projection in \(D_X\otimes D_Y\) is of the form
\( 
P_E=\sum_{(x,y)\in E} P_{x,y}\),
 \(E\subseteq X\times Y,
\)
where \(P_{x,y}=\varepsilon_{x,x}\otimes \varepsilon_{y,y}\) is the
rank-one projection onto \(\mathbb Ce_x\otimes \mathbb Ce_y\).
This gives a canonical identification
\[
\mathcal P(D_X\otimes D_Y)\cong \mathcal P(X\times Y),
\qquad
P_E\leftrightarrow E.
\]
Similarly, \(
\mathcal P(D_A\otimes D_B)\cong \mathcal P(A\times B).
\)
Under these identifications, the lattice operations on projections are just
the Boolean operations on subsets.
In particular, the atoms of \(\mathcal P(D_X\otimes D_Y)\) are precisely
the projections \(P_{x,y}\), corresponding to singleton question pairs
\(\{(x,y)\}\).

This observation of Todorov--Turowska~\cite[Section 10.3]{TT24} is the bridge between classical rules and
projection-lattice maps. A classical non-local game assigns, to each
question pair, a subset of admissible answer pairs:
\[
(x,y)\mapsto \lambda(x,y)
=
\{(a,b)\in A\times B:\lambda(a,b,x,y)=1\}.
\]
Equivalently, it assigns to each atom the output projection:
\[
P_{x,y} \mapsto \sum_{\lambda(a,b,x,y)=1}P_{a,b}.
\]
Since arbitrary projections in \(D_X\otimes D_Y\) are joins of atoms, the
rule extends uniquely by joins. Thus a classical rule predicate is exactly
the atomic data of a zero-preserving join-continuous map
\[
\varphi_{}:\mathcal P(D_X\otimes D_Y)\to \mathcal P(D_A\otimes D_B).
\]

We now make this correspondence explicit. Let
\(G=(X,Y,A,B,\lambda)\) be a non-local game, where
\(\lambda:A\times B\times X\times Y\to\{0,1\}\) is the rule predicate. We
write
\[
\Lambda_\lambda
=
\{((x,y),(a,b))\in (X\times Y)\times(A\times B):
\lambda(a,b,x,y)=1\}.
\]
Thus \(\lambda\) determines the rectangular operator space
\[
\mathcal U_{\Lambda_{\lambda}}
=
\operatorname{span}\{\varepsilon_{(a,b),(x,y)}:
((x,y),(a,b))\in\Lambda_{\lambda}\}
\subseteq
\mathcal B(\mathbb C^{X\times Y},\mathbb C^{A\times B}).
\]
Equivalently, \(\mathcal U_{\Lambda_{\lambda}}\) is the
\((D_A\otimes D_B)\)-\((D_X\otimes D_Y)\)-bimodule associated with the
relation \(\Lambda_{\lambda}\).

\begin{proposition}\label{prop:classical-games-join-maps}
Let \(X,Y,A,B\) be finite sets. Classical rule functions \( \lambda\) are in one-to-one correspondence with
zero-preserving join-continuous maps
\( 
\varphi:\mathcal P(D_X\otimes D_Y)\to\mathcal P(D_A\otimes D_B).
\) 
Under this correspondence,
\( 
\mathcal{U}_\varphi=\mathcal U_{\Lambda_{\lambda}},
\)
the diagonal bimodule associated with the winning relation of the game.
\end{proposition}

\begin{proof}
Let
\(\varphi:\mathcal P(D_X\otimes D_Y)\to\mathcal P(D_A\otimes D_B)\)
be zero-preserving and join-continuous. Since
\(\mathcal P(D_X\otimes D_Y)\) is atomic, every \(P_E\) satisfies
\( 
P_E=\bigvee_{(x,y)\in E}P_{x,y}.
\)
Hence
\( 
\varphi(P_E)=\bigvee_{(x,y)\in E}\varphi(P_{x,y}).
\)
Thus \(\varphi\) is completely determined by its values on the atoms. Define
a rule predicate \(\lambda_\varphi:A\times B\times X\times Y\to\{0,1\}\) by
\[
\lambda_\varphi(a,b,x,y)=1
\quad\Longleftrightarrow\quad
P_{a,b}\leq \varphi(P_{x,y}).
\]
In words, the winning answers for the question \((x,y)\) are precisely the
answer atoms contained in the projection \(\varphi(P_{x,y})\).

Conversely, let \(G=(X,Y,A,B,\lambda)\) be a classical game. Define
\[
\varphi_\lambda(P_E)
=
\bigvee\{P_{a,b}:\exists (x,y)\in E
\text{ such that } \lambda(a,b,x,y)=1\}.
\]
Equivalently, under the power-set identification,
\[
\varphi_\lambda(E)
=
\{(a,b)\in A\times B:
\exists (x,y)\in E \text{ such that } \lambda(a,b,x,y)=1\}.
\]
It is immediate that \(\varphi_\lambda(0)=0\). Moreover, if \((E_i)_i\) is
any family of subsets of \(X\times Y\), then an atom \(P_{a,b}\) is below
\(\varphi_\lambda(P_{\cup_iE_i})\) precisely when there exist \(i\) and
\((x,y)\in E_i\) such that \(\lambda(a,b,x,y)=1\). This is equivalent to
\(P_{a,b}\leq\vee_i\varphi_\lambda(P_{E_i})\). Hence
\(\varphi_\lambda\) preserves arbitrary joins.

The two constructions are inverse. Starting from \(\varphi\), we have
\[
\varphi_{\lambda_\varphi}(P_E)
=
\bigvee_{(x,y)\in E}
\bigvee\{P_{a,b}:P_{a,b}\leq \varphi(P_{x,y})\}
=
\bigvee_{(x,y)\in E}\varphi(P_{x,y})
=
\varphi(P_E).
\]
Starting from \(\lambda\), the condition
\(\lambda_{\varphi_\lambda}(a,b|x,y)=1\) is exactly the condition
\(P_{a,b}\leq \varphi_\lambda(P_{x,y})\), which is equivalent to
\(\lambda(a,b,x,y)=1\).

It remains to identify the associated operator space. For
\(T\in\mathcal B(\mathbb C^{X\times Y},\mathbb C^{A\times B})\), the
condition \(T\in\mathcal{U}_{\varphi_\lambda}\) means
\[
\varphi_\lambda(P_E)^\perp T P_E=0,
\qquad
E\subseteq X\times Y.
\]
Applying this to the atom \(P_{x,y}\), we see that
\(T(e_x\otimes e_y)\) must lie in the range of
\(\varphi_\lambda(P_{x,y})\). Equivalently,
\[
\langle T(e_x\otimes e_y),e_a\otimes e_b\rangle=0
\qquad
\text{whenever } \lambda(a,b,x,y)=0.
\]
Thus the matrix coefficients of \(T\) are supported on the operator-space associated with the winning relation. Conversely,
if all forbidden matrix coefficients vanish, then for every
\(E\subseteq X\times Y\), the subspace
\(T(P_E(\mathbb C^{X\times Y}))\) is contained in
\(\varphi_\lambda(P_E)(\mathbb C^{A\times B})\), and hence
\(\varphi_\lambda(P_E)^\perp TP_E=0\). Therefore
\[
\mathcal{U}_{\varphi_\lambda}
=
\operatorname{span}
\{\varepsilon_{(a,b),(x,y)}:\lambda(a,b,x,y)=1\}
=
\mathcal U_{\Lambda_{\lambda}}.
\]
This is exactly the diagonal bimodule determined by the winning relation.
\end{proof}


\subsection{The Todorov--Turowska quantum non-local games}

Let \(X,Y,A,B\) be finite sets and put
\( 
H_{\rm in}=\mathbb C^X\otimes\mathbb C^Y\), and
\( H_{\rm out}=\mathbb C^A\otimes\mathbb C^B.
\)
In the Todorov--Turowska setting \cite{TT24}, a quantum
non-local game is encoded by a zero-preserving join-continuous map
\[
\varphi:\mathcal P(M_X\otimes M_Y)\to
\mathcal P(M_A\otimes M_B).
\]

Operationally, \(\varphi\) should be read as a support-valued rule. If the
input is supported in a subspace \(PH_{\rm in}\), where
\( 
P\in\mathcal P(M_X\otimes M_Y),
\)
then the rule declares that the output of a perfectly winning transformation
must be supported in
\( 
\varphi(P)H_{\rm out}.
\)
Thus, a transformation \(T:H_{\rm in}\to H_{\rm out}\) obeys the rule if
\[
T(PH_{\rm in})\subseteq \varphi(P)H_{\rm out},
\qquad
P\in\mathcal P(M_X\otimes M_Y).
\]
Equivalently,
\[
\varphi(P)^\perp TP=0,
\qquad
P\in\mathcal P(M_X\otimes M_Y).
\]
The winning transformation space associated with \(\varphi\) is therefore
\[
\mathcal U_\varphi
=
\{T\in \B(H_{\rm in},H_{\rm out}):
\varphi(P)^\perp TP=0
\text{ for all }P\in\mathcal P(M_X\otimes M_Y)\}.
\]
In words, \(\mathcal U_\varphi\) consists of those transformations that send
every quantum input support to the output support prescribed by the rule.

The connection with the projection-test quantum-game picture is as follows.  In a
projection-test quantum game, one starts from an input state \(\psi\) and an
accepting projection \(Q\), and obtains the winning transformation space
\(\mathcal U_{\psi,Q}\) from the condition
\(Q^\perp(T\otimes 1_{H_R})\psi=0\).
In the Todorov--Turowska formulation, the starting point is instead the
projection-lattice rule \(\varphi\).  This rule first determines the support
space
\( 
\mathcal U_\varphi
\).
Vectorising these support constraints gives a projection-test quantum-game
realisation.  Namely, take \(H_R=\overline{H_{\rm in}}\), let
\(\Omega=(\dim H_{\rm in})^{-1/2}\sum_i e_i\otimes\overline e_i\), where $ (e_i)$ is an orthonormal basis in $ H_{\rm in}$, and set
\[
Q_\varphi
=
P_{\operatorname{vec}(\mathcal U_\varphi)}
=
\bigwedge_{P\in\mathcal P(M_X\otimes M_Y)}
\bigl(\varphi(P)^\perp\otimes P^{\op}\bigr)^\perp .
\]
Since \((T\otimes 1)\Omega=(\dim H_{\rm in})^{-1/2}\operatorname{vec}(T)\), we
have
\( 
Q_\varphi^\perp(T\otimes 1)\Omega=0\) 
if and only if 
\( T\in\mathcal U_\varphi.
\)
Thus, every Todorov--Turowska rule gives a projection-test quantum game, but with
a special accepting projection: \(Q_\varphi\) is the vectorised projection
onto a transformation space defined by support constraints.

Conversely, the support rule recovered from \(Q_\varphi\) is the Erdos support
map of \(\mathcal U_\varphi\).  Namely,
\(\varphi_{\mathcal U_\varphi}(P)=P_{[\mathcal U_\varphi P H_{\rm in}]}\).
Therefore \(Q_\varphi\) recovers the  rule
\(\varphi_{\mathcal U_\varphi}\), and the original rule \(\varphi\) is
recovered exactly when \(\varphi=\varphi_{\mathcal U_\varphi}\).

Thus, the projection-lattice formalism is the support-level description of
the reflexive part of the operator space picture: for every operator space
\(\mathcal U\), one has
\(\mathcal U_{\varphi_{\mathcal U}}=\operatorname{Ref}\mathcal U\).

\subsection{Operator-algebraic generalisation}
Before we move on to examples, we find it useful to pass to the more general von Neumann algebraic setting. In this way our definitions will capture all special cases of classical-to-classical, classical-to-quantum, quantum-to-classical and quantum-to-quantum situations as well as allow us to compare the different formalisms effectively.  

Let
\(\mathcal X,\mathcal Y,\mathcal A,\mathcal B\) be finite-dimensional von
Neumann algebras acting on Hilbert spaces
\(H_X\), \(H_Y\), \(H_A\), \(H_B\). A \emph{projection-lattice quantum game} from \( (\mathcal X,\mathcal Y)\) to \( (\mathcal A,\mathcal B)\) is a zero-preserving
join-continuous map
\[
\varphi:\mathcal P(\mathcal X\otimes \mathcal Y)\to\mathcal P(\mathcal A\otimes \mathcal B).
\]
The associated operator space is
\[
\mathcal U_\varphi
=
\{T\in \B(H_X\otimes H_Y,H_A \otimes H_B):\varphi(P)^\perp TP=0
\text{ for all }P\in\mathcal P(\mathcal X\otimes \mathcal Y)\}.
\]
We show that reflexive bimodules are in  correspondence with projection-lattice quantum games.
\begin{proposition} \label{prop:bimo_lattice}
Let \(\mathcal M\subseteq \B(H)\) and \(\mathcal N\subseteq \B(K)\) be von Neumann algebras. Given any zero-preserving join-continuous map
\( 
\varphi:\mathcal P(\mathcal M)\to\mathcal P(\mathcal N),
\)
the operator space $ \mathcal U_{\varphi}\subseteq \B(H,K)$ is a reflexive \(\mathcal N'\)-\(\mathcal M'\) bimodule. Conversely, 
if $ \mathcal U \subseteq \B(H,K)$ is a reflexive \(\mathcal N'\)-\(\mathcal M'\) bimodule, then the map $ \varphi_{\mathcal U}: \mathcal P(\mathcal M) \to \mathcal P(\mathcal N)$ is zero-preserving and join-continuous and $ \mathcal U=\mathcal U_{\varphi_{\mathcal U}}$.

\end{proposition}
\begin{proof}
Let \(\varphi:\mathcal P(\mathcal M)\to\mathcal P(\mathcal N)\) be zero-preserving and join-continuous, and let $  \mathcal{U}_\varphi$ be the corresponding operator space.
If \(n'\in\mathcal N'\), \(m'\in\mathcal M'\), \(T\in\mathcal U_\varphi\),
and \(P\in\mathcal P(\mathcal M)\), then
\[
n'Tm'PH
=
n'TPm'H
\subseteq n'TPH
\subseteq n'\varphi(P)K
\subseteq \varphi(P)K.
\]
Thus, \(\mathcal{U}_\varphi\) is an \(\mathcal N'\)-\(\mathcal M'\) bimodule. Next, we show that \(\mathcal{U}_\varphi\) is reflexive. Let \(T\in\operatorname{Ref}\mathcal{U}_\varphi\). If \(P\in\mathcal P(\mathcal M)\) and \(\xi\in PH\), then
\[
T\xi\in[\mathcal{U}_\varphi\xi]\subseteq[\mathcal{U}_\varphi PH]\subseteq \varphi(P)K.
\]
Hence \(TPH\subseteq\varphi(P)K\) for every \(P\), so \(T\in\mathcal U_{\varphi}\). Therefore
\(\operatorname{Ref}\mathcal{U}_\varphi=\mathcal{U}_\varphi\).

Conversely, let \(\mathcal U\subseteq \B(H,K)\) be a reflexive
\(\mathcal N'\)-\(\mathcal M'\) bimodule and consider $ \varphi_{\mathcal U}$.
Since \(\mathcal U\) is a left \(\mathcal N'\)-module, the subspace
\([\mathcal UPH]\) is invariant under \(\mathcal N'\), hence reducing for
\(\mathcal N'\). Therefore \(\varphi_{\mathcal U}(P)\in\mathcal P((\mathcal N')')=\mathcal P(\mathcal N)\).
Clearly \(\varphi_{\mathcal U}(0)=0\). If \(P=\bigvee_i P_i\), then
\[
[\mathcal UPH]
=
\left[\sum_i\mathcal UP_iH\right],
\]
so \(\varphi_{\mathcal U}\) preserves joins.

Finally, \(\mathcal U\subseteq\mathcal{U}_{\varphi_{\mathcal U}}\) is immediate. For the reverse inclusion, let
\(T\in\mathcal{U}_{\varphi_{\mathcal U}}\) and \(\xi\in H\). Let \(P_{[\mathcal M'\xi]}\) be the projection onto \([\mathcal M'\xi]\). Then
\(P_{[\mathcal M'\xi]}\in\mathcal P(\mathcal M)\), and
\[
T\xi\in T P_{[\mathcal M'\xi]}H
\subseteq [\mathcal U P_{[\mathcal M'\xi]}H]
=
[\mathcal U\mathcal M'\xi]
=
[\mathcal U\xi].
\]
Thus \(T\in\operatorname{Ref}\mathcal U=\mathcal U\). Hence
\( 
\mathcal{U}_{\varphi_{\mathcal U}}=\mathcal U.
\)
This proves the claimed correspondence.
\end{proof}

\subsection{Strategies and winning transformations}

The  interpretation of the strategies for projection-test quantum games of Subsection \ref{subs:strategies_coop}  also applies to projection-lattice games.

If
\(\varphi:\mathcal P(M_X\otimes M_Y)\to\mathcal P(M_A\otimes M_B)\) is a
Todorov--Turowska rule, then a physical strategy is a quantum channel \(\Gamma : M_X \otimes M_Y \to M_A \otimes M_B\), belonging to one of the no-signalling
correlation resource classes defined in \cite{TT24}.

The rule declares \(\Gamma\) perfect when
\(\langle\Gamma(P),\varphi(P)^\perp\rangle=0\) for every projection
\(P\in\mathcal P(M_X\otimes M_Y)\). For a Kraus decomposition
\(\Gamma(\rho)=\sum_iT_i\rho T_i^*\), this is equivalent to
\(\varphi(P)^\perp T_iP=0\) for all \(P\) and \(i\). Hence
\( 
\Gamma\) is perfect for \( \varphi\) 
 if and only if \( 
\mathcal X_\Gamma\subseteq\mathcal U_\varphi .
\)
The projection-lattice rule therefore determines the Kraus-level winning
space, while the chosen resource class determines which channels are allowed
as physical strategies. Since by Proposition \ref{prop:operator-space-projection-test}, there exist $ \psi, Q$ such that  $ \mathcal U_{\varphi}= \mathcal U_{\psi,Q}$, the winning conditions coincide for the two models.

Similarly, for a generalised projection-lattice rule
\( 
\varphi:\mathcal P(\mathcal X\otimes\mathcal Y)
\to
\mathcal P(\mathcal A\otimes\mathcal B),
\)
a strategy \(\Gamma : (\mathcal X \otimes \mathcal Y)_* \to (\mathcal A \otimes \mathcal B)_{*}\) is perfect if every input state supported in \(P\) is
sent to an output state supported in \(\varphi(P)\), for every
\(P\in\mathcal P(\mathcal X\otimes\mathcal Y)\). Equivalently, for every
normal state \(\rho\) with \(s(\rho)\leq P\), we have 
\( 
\langle\Gamma(\rho),\varphi(P)^\perp\rangle=0.
\) 
For a Kraus decomposition, this is equivalent to
\( 
\varphi(P)^\perp T_iP=0\), for all 
\( P\in\mathcal P(\mathcal X\otimes\mathcal Y)\) and all \(  i,
\)
and therefore
\( 
\Gamma\)  is perfect for \( \varphi\) 
if and only if
\(\mathcal X_\Gamma\subseteq\mathcal U_\varphi.
\)


\subsection{Hypergraph quantum games} \label{subsec:hypergr}
We now recall the quantum-game formalisms considered in \cite{CLTT23}.  The authors study projection quantum games, including rank-one and finite-rank quantum games, and hypergraph quantum games arising from probabilistic quantum hypergraphs.  From the perspective of the present paper, these are all instances of the same projection-test formalism once the referee Hilbert space is allowed to be arbitrary, and possibly infinite-dimensional.

Recall from Subsection~\ref{sec:class_latt} that a classical non-local game
$G=(X,Y,A,B,\lambda)$ determines a winning relation
\( 
\Lambda_{\lambda}
\subseteq
(X\times Y) \times (A\times B),
\)
equivalently a zero-preserving join-continuous map
\[
\phi_\lambda:
\mathcal P(D_X\otimes D_Y)
\to
\mathcal P(D_A\otimes D_B).
\]
For each question atom
$P_{x,y}=\varepsilon_{x,x}\otimes \varepsilon_{y,y}$, the projection
$\phi_\lambda(P_{x,y})$ is the diagonal projection onto the span of the
accepted answer atoms:
\[
\phi_\lambda(P_{x,y})
=
\sum_{\lambda(a,b,x,y)=1}P_{a,b}.
\]
If the classical game is equipped with a question distribution $\pi$ on
$X\times Y$, then $\pi$ may be regarded as the atomic probability measure
\[
\mu_\pi
=
\sum_{(x,y)\in X\times Y}
\pi(x,y)\,\delta_{P_{x,y}}
\]
on the pure state space of $M_X\otimes M_Y$, identified with the rank-one
projections on $\mathbb C^X\otimes\mathbb C^Y$.

A classical non-local game together with a question distribution $(G,\pi)$
can therefore be viewed as a \emph{probabilistic quantum hypergraph} $( \phi_{\lambda},\mu_\pi)$ supported on
the classical question atoms. 

\begin{remark}\rm
The terminology ``hypergraph'' comes from the classical case as follows.  A
rule predicate \(\lambda:A\times B\times X\times Y\to\{0,1\}\) determines a
hypergraph
\( 
\mathbb H_\lambda=(V,E)\), where 
\( V=A\times B,
\)
whose hyperedges are the winning answer sets
\[
E_{x,y}
=
\{(a,b)\in A\times B:\lambda(a,b,x,y)=1\},
\qquad (x,y)\in X\times Y.
\]
Strictly speaking, this is an \(X\times Y\)-indexed hypergraph on the vertex set $A \times B$.  The same labelled data can equivalently be written as the binary relation
from $X\times Y$ to $A\times B$,
\[
\Lambda_\lambda
=
\{((x,y),(a,b))\in (X\times Y)\times(A\times B):
\lambda(a,b,x,y)=1\}.
\]
In this paper, we mainly use this relational form, because it passes directly
to the associated operator space of winning transformations.
\end{remark}

The quantum definition replaces  diagonal classical pure inputs by
arbitrary pure input states.  Let
\[
\mathbb P_{XY}
=
\{\xi\xi^*:\xi\in \mathbb C^X\otimes\mathbb C^Y,\ \|\xi\|=1\}
\]
be the pure state space of $M_X\otimes M_Y$, identified with the rank-one
projections on $\mathbb C^X\otimes\mathbb C^Y$ and equipped with the
operator-norm topology.  A \emph{hypergraph quantum game} over
$(X,Y,A,B)$ is a  pair
\( 
(\phi,\mu),
\)
where $\mu$ is a regular Borel probability measure on $\mathbb P_{XY}$ and
\[
\phi:\mathbb P_{XY}\to \mathcal P(M_A\otimes M_B)
\]
is Borel measurable.  Thus, the referee samples a pure input state
$p\in\mathbb P_{XY}$ according to $\mu$, the players apply a transformation $T: \bb C^{X} \otimes \bb C^Y \to \bb C^{A} \otimes \bb C^B$ and the projection $\phi(p)$
specifies the accepted output support.

\begin{remark} \rm
It is important to distinguish this rule from the projection-lattice maps of
the previous subsections.  There, the rule is encoded as a zero-preserving
join-continuous map on the relevant projection lattice.  In the probabilistic
quantum hypergraph formalism, the rule is instead a Borel map defined only
on the pure state space $\mathbb P_{XY}$.  It is not required to preserve
joins, and the emphasis is on averaging the payoff with respect to $\mu$. We will see the precise relation between the two approaches in Proposition~\ref{prop:hypergraph-reflexive}.
\end{remark}

We associate with a hypergraph quantum game $(\phi,\mu)$ an
almost-everywhere version of the winning transformation space.  Let
\( 
H_{\rm in}=\mathbb C^X\otimes\mathbb C^Y\) and \( 
H_{\rm out}=\mathbb C^A\otimes\mathbb C^B.
\)
For $T\in\B(H_{\rm in},H_{\rm out})$ and $p=\xi\xi^*\in\mathbb P_{XY}$,
the pointwise winning condition is
\( 
\phi(p)^\perp T\xi=0.
\)
This condition is independent of the choice of unit vector $\xi$ representing
$p$.  Equivalently,  the winning transformation space associated with $(\phi,\mu)$ is
\[
\mathcal U_{\phi,\mu}
=
\left\{
T\in\B(H_{\rm in},H_{\rm out}):
\int_{\mathbb P_{XY}}
\operatorname{Tr}\bigl(TpT^*\phi(p)^\perp\bigr)\,d\mu(p)=0
\right\},
\]
since \( \|\phi(p)^\perp T\xi\|^2
=
\operatorname{Tr}\bigl(TpT^*\phi(p)^\perp\bigr)\).

Let $\Gamma:M_{X} \otimes M_Y\to M_{A} \otimes M_B$ be a quantum channel.
The \emph{winning probability} of the strategy $\Gamma$ for the hypergraph quantum game $(\phi,\mu)$  is
\[
\omega_{\phi,\mu}(\Gamma)
=
\int_{\mathbb P_{XY}}
\operatorname{Tr}\bigl(\phi(p)\Gamma(p)\bigr)\,d\mu(p).
\]

Projection-test quantum games with $ H_R= \ell^2(I)$ were called projection quantum games in  \cite{CLTT23}.
The connection between hypergraph quantum games and projection-test quantum games in the finitely supported case was observed in
\cite[Remark~4.13]{CLTT23}.

\begin{proposition}\label{prop:hypergraph-projection-test}
Let \((\phi,\mu)\) be a hypergraph quantum game.  Then there is a
projection-test quantum game \(G=(\psi,Q)\), with referee space
\(L^2(\mathbb P_{XY},\mu)\), such that
\( 
\mathcal U_{\psi,Q}=\mathcal U_{\phi,\mu}.
\)
Moreover, for every quantum channel \(\Gamma\), 
\[
\omega_{\phi,\mu}(\Gamma) =\omega_{\psi,Q}(\Gamma).
\]
\end{proposition}
\begin{proof}
 If \(\mu\) is discrete, say
\( 
\mu=\sum_{i\in I}\mu_i\,\delta_{\xi_i\xi_i^*}\),
\( 
\sum_{i\in I}\mu_i=1,
\)
where \(I\) is finite or countable, then this averaged family can be packaged
as a single projection-test quantum game.  Let \(H_R=\ell^2(I)\), with
orthonormal basis \(\{e_i:i\in I\}\), and set
\[
\psi=\sum_{i\in I}\sqrt{\mu_i}\,\xi_i\otimes e_i,
\qquad
Q=
\sum_{i\in I}
\phi(\xi_i\xi_i^*)\otimes e_ie_i^*,
\]
where the second sum is understood in the strong operator topology.
Then \(Q\) is a block-diagonal projection in
\( 
\mathcal P\big((M_A\otimes M_B)\otimes \B(H_R)\big),
\)
and the projection-test \((\psi,Q)\)  has the same winning transformation space. 

For a general regular Borel probability measure \(\mu\), the same construction
can be written using vector-valued \(L^2\)-spaces.  Let
\(H_R=L^2(\mathbb P_{XY},\mu)\).  We identify
\( 
\mathbb C^{XY}\otimes H_R
\cong
L^2(\mathbb P_{XY},\mu;\mathbb C^{XY})
\)
and similarly
\( 
\mathbb C^{AB}\otimes H_R
\cong
L^2(\mathbb P_{XY},\mu;\mathbb C^{AB}).
\)
Since \(\mathbb P_{XY}\) consists of rank-one projections, we choose a
measurable section
\( 
\mathbb P_{XY}\ni p\longmapsto \xi_p\in\mathbb C^{XY}.
\)
Thus \(\|\xi_p\|=1\) and \(p=\xi_p\xi_p^*\).

Let us briefly explain why such a section exists.  Fix an orthonormal basis
\(e_1,\dots,e_n\) of \(\mathbb C^{XY}\).  For \(p\in\mathbb P_{XY}\), let
\(j(p)\) be the first index such that \(pe_{j(p)}\neq0\), and define
\( 
\xi_p=\frac{pe_{j(p)}}{\|pe_{j(p)}\|}.
\)
This is well defined, since \(p\neq0\) is rank one.  The set on which
\(j(p)=j\) is
\( 
\{p:\|pe_i\|=0\text{ for }i<j,\ \|pe_j\|>0\},
\)
and is Borel, because the maps \(p\mapsto pe_i\) are continuous in the
operator norm.  On this set the map \(p\mapsto pe_j/\|pe_j\|\) is continuous.
Hence \(p\mapsto\xi_p\) is Borel.  Moreover, \(\xi_p\) is a unit vector in
\({\rm Ran}\,p\), and therefore \(p=\xi_p\xi_p^*\).

Define
\(\psi\in L^2(\mathbb P_{XY},\mu;\mathbb C^{XY})\) by
\( 
\psi(p)=\xi_p\),  for  \(p\in\mathbb P_{XY}.
\)
Since \(\|\xi_p\|=1\) and \(\mu\) is a probability measure, \(\psi\) is a unit
vector. Since \(M_{AB}\) is finite-dimensional and \(\phi\) is Borel with values
in projections, we may regard \(\phi\) as an element of
\(L^\infty(\mathbb P_{XY},\mu;M_{AB})\).  Let \(Q=M_\phi\) be the
corresponding multiplication operator on
\(L^2(\mathbb P_{XY},\mu;\mathbb C^{AB})\cong \mathbb C^{AB}\otimes H_R\);
that is,
\[
(Qf)(p)=\phi(p)f(p),
\qquad
f\in L^2(\mathbb P_{XY},\mu;\mathbb C^{AB}).
\]
Since \(\phi(p)^2=\phi(p)=\phi(p)^*\) for every \(p\), it follows
fibrewise that \(Q^2=Q=Q^*\).  Hence \(Q\) is a projection. 

Thus \((\psi,Q)\) is a projection-test quantum game with referee space
\(L^2(\mathbb P_{XY},\mu)\).  For \(T\in \B(\mathbb C^{XY},\mathbb C^{AB})\),
we have
\( 
\big(Q^\perp(T\otimes 1_{H_R})\psi\big)(p)
=
\phi(p)^\perp T\xi_p
\)
for \(\mu\)-almost every \(p\in\mathbb P_{XY}\).  Hence,
\[
\mathcal U_{\psi,Q}
=
\left\{
T\in \B(\mathbb C^{XY},\mathbb C^{AB}):
\phi(p)^\perp T\xi_p=0
\text{ for }\mu\text{-a.e. }p\in\mathbb P_{XY}
\right\}.
\]
This is precisely the almost-everywhere winning transformation space
associated with the hypergraph quantum game \((\phi,\mu)\). The space does not depend on the chosen measurable section.  

The values agree as well.    Let $\Gamma:M_{X} \otimes M_Y\to M_{A} \otimes M_B$ be a quantum channel and choose
a Kraus decomposition
\( 
\Gamma(\rho)=\sum_{i=1}^N T_i\rho T_i^*\), where 
\( 
T_i\in \B(\mathbb C^{XY},\mathbb C^{AB}).
\)
Since $\psi(p)=\xi_p$ and $Q=M_\phi$, we have
\( 
\bigl((T_i\otimes 1_{H_R})\psi\bigr)(p)=T_i\xi_p
\)
for $\mu$-almost every $p$.  Hence
\[
\begin{aligned}
\omega_{\psi,Q}(\Gamma)
&=
\operatorname{Tr}\bigl(Q(\Gamma\otimes {\rm id}_{S^1(H_R)})(\psi\psi^*)\bigr)
\\
&=\sum_{i=1}^N
\left\langle
Q(T_i\otimes 1_{H_R})\psi,\,
(T_i\otimes 1_{H_R})\psi
\right\rangle                                                   \\
&=
\sum_{i=1}^N
\int_{\mathbb P_{XY}}
\left\langle
\phi(p)T_i\xi_p,\,
T_i\xi_p
\right\rangle
\,d\mu(p)                                                        \\
&=
\int_{\mathbb P_{XY}}
\operatorname{Tr}
\left(
\phi(p)
\sum_{i=1}^N T_i p T_i^*
\right)
\,d\mu(p)                                                        \\
&=
\int_{\mathbb P_{XY}}
\operatorname{Tr}\bigl(\phi(p)\Gamma(p)\bigr)\,d\mu(p)
=
\omega_{\phi,\mu}(\Gamma).
\end{aligned}
\]
\end{proof}

One might alternatively use Proposition~\ref{prop:operator-space-projection-test}
to realise the hypergraph winning space as a projection-test game.  This
universal realisation, however, preserves only the winning
transformation space, and hence perfect play; it does not preserve the
hypergraph value in general.  By contrast, the
\(L^2(\mathbb P_{XY},\mu)\)-realisation of
Proposition~\ref{prop:hypergraph-projection-test} retains the measurable
family of tests and gives the value-preserving projection-test realisation.

We finish by recording the relation with projection-lattice games at
the level of perfect transformations.

\begin{proposition}\label{prop:hypergraph-reflexive}
Let \(H_{\rm in}=\bb C^{X}\otimes \bb C^Y\) and \(H_{\rm out}= \bb C^{A} \otimes \bb C^B\). The following hold: 
\begin{enumerate}
\item[(i)] If \((\phi,\mu)\) is a hypergraph quantum game, then
\(\mathcal U_{\phi,\mu}\) is reflexive.

\item[(ii)] Conversely, every reflexive operator space $\cl U\subseteq \B(H_{\rm in}, H_{\rm out})$ arises this way; there is a finitely supported hypergraph
quantum game \((\phi,\mu)\) such that
\[
\mathcal U_{\phi,\mu}
=
\mathcal U.
\]
\end{enumerate}
\end{proposition}
\begin{proof}
(i) Let \((\phi,\mu)\) be a hypergraph quantum game, and let
\(\mathcal U_{\phi,\mu}\) be its winning transformation space.  Then
\(\mathcal U_{\phi,\mu}\) is reflexive. 
Indeed,  choose
a basis \(T_1,\dots,T_r\) of \(\mathcal U_{\phi,\mu}\).  For each \(j\), set
\[
E_j=
\{p\in\mathbb P_{XY}:\phi(p)^\perp T_j\xi_p=0\}.
\]
This set is Borel, because \(p\mapsto \xi_p\) is Borel measurable and \(\phi\) is Borel
with values in the finite-dimensional projection lattice.  Since
\(T_j\in\mathcal U_{\phi,\mu}\), the set \(E_j\) is of measure one.  Hence
\(E=\bigcap_{j=1}^rE_j\) is of measure one.  For \(p\in E\), every element of
\(\mathcal U_{\phi,\mu}\) is a linear combination of the \(T_j\)'s and
therefore
\( 
\mathcal U_{\phi,\mu}\xi_p\subseteq \phi(p)H_{\rm out}.
\)
Now let \(S\in\operatorname{Ref}\mathcal U_{\phi,\mu}\).  By definition,
\(S\xi\in[\mathcal U_{\phi,\mu}\xi]\) for every \(\xi\in H_{\rm in}\).  In
particular, for every \(p\in E\),
\[
S\xi_p\in[\mathcal U_{\phi,\mu}\xi_p]\subseteq \phi(p)H_{\rm out}.
\]
Thus \(\phi(p)^\perp S\xi_p=0\) for \(\mu\)-almost every \(p\), and so
\(S\in\mathcal U_{\phi,\mu}\).  Hence
\(\operatorname{Ref}\mathcal U_{\phi,\mu}\subseteq\mathcal U_{\phi,\mu}\).
The reverse inclusion is automatic, and therefore
\(\mathcal U_{\phi,\mu}\) is reflexive.

(ii) If $\cl U = \B(H_{\rm in}, H_{\rm out})$, choose any $ p_0 \in \bb P_{XY}$, set $\mu = \delta_{p_0}$, and let $ \phi(p)=1 $ for every $p$. Then, $\cl U_{\phi,\mu}= \B(H_{\rm in}, H_{\rm out})$.

Now assume that $\cl U \neq \B(H_{\rm in}, H_{\rm out})$. Since $\cl U$ is reflexive,  by Larson's 
rank-one preannihilator criterion \cite{Lar82} (see also
\cite[Corollary 9.3]{Erdos86}), there exist non-zero vectors \(\xi_j\in H_{\rm in}\) and
\(\eta_j\in H_{\rm out}\), \(1\leq j\leq r\), such that
\[
\mathcal U
=
\{T\in\B(H_{\rm in},H_{\rm out}):
\langle\eta_j,T\xi_j\rangle=0,\ 1\leq j\leq r\}.
\]
Let \(p_1,\dots,p_m\) be the distinct rank-one projections onto the lines
\(\mathbb C\xi_j\).  For each \(\ell\), set
\( 
F_\ell= \mathrm{span}\{ \eta_j:\xi_j\in p_\ell H_{\rm in}\}
\subseteq H_{\rm out}.
\)
Define
\[
\mu=\frac1m\sum_{\ell=1}^m\delta_{p_\ell},
\qquad
\phi(p_\ell)=P_{F_\ell^\perp},
\]
and define \(\phi\) arbitrarily off \(\{p_1,\dots,p_m\}\), say
\(\phi(p)=1\). The map \(\phi\) is Borel, since it is constant off the finite Borel set
\(\{p_1,\dots,p_m\}\).  Then the hypergraph winning condition at the atom \(p_\ell\)
is
\( 
\phi(p_\ell)^\perp Tp_\ell=0.
\)
Since \(p_\ell H_{\rm in}\) is one-dimensional, this is equivalent to
\( 
\langle\eta_j,T\xi_j\rangle=0
\)
for every \(j\) with \(\xi_j\in p_\ell H_{\rm in}\).  Taking all atoms
\(p_\ell\), we obtain
\( 
\mathcal U_{\phi,\mu}=\mathcal U.
\)
\end{proof}


\section{Rule-operator definition} \label{sec:Gold}

In this section we discuss Goldberg's definition of quantum games
\cite{Goldberg26, Goldberg25}. In classical non-local games, the question and answer
systems are finite sets. Goldberg's quantisation replaces these finite sets by
quantum sets in the sense of Musto--Reutter--Verdon \cite{MRV18}; see also
Kornell's related formulation \cite{Kor20}. Categorically, a finite quantum
set is a special symmetric dagger Frobenius algebra, namely a finite-dimensional
Hilbert space equipped with compatible multiplication and unit maps. Such
Frobenius algebras give a categorical presentation of finite-dimensional
\(C^*\)-algebras \cite{Vic11}; see also \cite[Corollary 5.34]{HV19}.

A \emph{dagger Frobenius algebra} is a finite-dimensional Hilbert space \(X\)
equipped with linear maps \(m:X\otimes X\to X\) and
\(\eta:\mathbb C\to X\), called the multiplication and unit, respectively,
such that \(m\) is associative and unital:
\[
m(m\otimes 1_X)=m(1_X\otimes m),
\qquad
m(\eta\otimes 1_X)=1_X=m(1_X\otimes \eta).
\]
It is also required to satisfy the Frobenius identity
\[
(1_X\otimes m)(m^*\otimes 1_X)
=
m^*m
=
(m\otimes 1_X)(1_X\otimes m^*).
\]
A dagger Frobenius algebra is called \emph{special} if \(mm^*=1_X\), and
\emph{symmetric} if \(\eta^*m\sigma=\eta^*m\), where
\(\sigma:X\otimes X\to X\otimes X\) is the flip map. A \emph{finite quantum
set} is a special symmetric dagger Frobenius algebra.

For the operator-algebraic comparison, it is useful to work with a slightly
more flexible concrete realisation. After choosing faithful traces, we
realise the relevant Frobenius algebras as Hilbert spaces \(L^2(\mathcal R)\)
associated with finite-dimensional tracial von Neumann algebras
\(\mathcal R\). In this realisation, the Frobenius structure is induced by the
multiplication map \(m_{\mathcal R}\), the unit map \(\eta_{\mathcal R}\), and
their adjoints. We use this concrete form to translate Goldberg's rule
operators into projections, and then into operator bimodules.

Goldberg's original formalism uses the \emph{special normalisation}
\(m_{\mathcal R}m_{\mathcal R}^*=1_{L^2(\mathcal R)}\). In the comparison
below, we allow arbitrary faithful traces. Thus we work with
symmetric dagger Frobenius algebras that need not be special, obtaining a
trace-dependent extension of Goldberg's rule-operator equations. The special
normalisation recovers Goldberg's original setting. The operator-algebraic
techniques employed in this section are in the same spirit as those employed by Daws
\cite{Daws24} in comparing different definitions of quantum graphs.

We recall the necessary background.  In this section, a tracial von Neumann
algebra is a pair \((\mathcal M,\tau_{\mathcal M})\), where \(\mathcal M\)
is finite dimensional and \(\tau_{\mathcal M}\) is a faithful positive
tracial functional; thus
\[
\tau_{\mathcal M}(xy)=\tau_{\mathcal M}(yx),
\qquad
\tau_{\mathcal M}(x^*x)>0 \quad (x\neq 0).
\]
 The opposite algebra is equipped with the trace
\(\tau_{\mathcal M}^{\op}\) given by
\( 
\tau_{\mathcal M}^{\op}(x^{\op})=\tau_{\mathcal M}(x)\),
\( x\in\mathcal M.
\)

The associated \(L^2\)-space is obtained by imposing a Hilbert space structure on $\cl M$,  from the inner product
\[
\langle \widehat x,\widehat y\rangle
=
\tau_{\mathcal M}(x^*y),
\qquad x,y\in\mathcal M,
\]
where \(\widehat x\) denotes the vector represented by \(x\).  Since
\(\mathcal M\) is finite dimensional, the map \(x\mapsto\widehat x\)
identifies \(\mathcal M\) linearly with \(L^2(\mathcal M)\). 
Unless stated
otherwise, traces are not assumed to be normalised.

We write \(\Tr_n\) for the unnormalised trace on \(M_n\). Every
finite-dimensional von Neumann algebra is a finite-dimensional
\(C^*\)-algebra, hence, up to \(*\)-isomorphism, a finite direct sum
\[
\mathcal M=\bigoplus_\alpha M_{n_\alpha}.
\]
If \(1_\alpha\) is the central projection corresponding to the summand
\(M_{n_\alpha}\), then every faithful tracial functional on \(\mathcal M\)
has the form
\[
\tau_{\mathcal M}
=
\sum_\alpha c_\alpha \operatorname{Tr}_{n_\alpha},
\qquad c_\alpha>0.
\]

When a finite-dimensional von Neumann algebra \(\mathcal M\) is faithfully
represented on a Hilbert space \(H\), we shall also use its spatial
decomposition. Namely, up to unitary equivalence,
\[
H
=
\bigoplus_\alpha
(\bb C^{n_\alpha}\otimes E_\alpha),
\qquad
\mathcal M
=
\bigoplus_\alpha
(M_{n_\alpha}\otimes 1_{E_\alpha}),
\]
where the \(E_\alpha\)'s are non-zero multiplicity spaces. In this
representation,
\[
\mathcal M'
=
\bigoplus_\alpha
(1_{n_\alpha}\otimes B(E_\alpha)).
\]

Let \((\mathcal R,\tau_{\mathcal R})\) be a finite-dimensional von Neumann
algebra equipped with a faithful tracial functional. We write \(L^2(\mathcal R)\)
for the induced Hilbert space. The product and unit of \(\mathcal R\)
induce maps
\[
m_{\mathcal R}:L^2(\mathcal R)\otimes L^2(\mathcal R)\to L^2(\mathcal R),
\qquad
\eta_{\mathcal R}:\mathbb C\to L^2(\mathcal R),
\]
given by
\[
m_{\mathcal R}(\widehat x\otimes\widehat y)=\widehat{xy},
\qquad
\eta_{\mathcal R}(\alpha)=\alpha\widehat{1_{\mathcal R}}.
\]
Thus \(\eta_{\mathcal R}^*(\widehat x)=\tau_{\mathcal R}(x)\). With these
maps, \(L^2(\mathcal R)\) is a dagger Frobenius algebra. Moreover, it is
symmetric precisely because \(\tau_{\mathcal R}\) is tracial:
\[
\eta_{\mathcal R}^*m_{\mathcal R}(\widehat x\otimes\widehat y)
=
\tau_{\mathcal R}(xy)
=
\tau_{\mathcal R}(yx)
=
\eta_{\mathcal R}^*m_{\mathcal R}\sigma(\widehat x\otimes\widehat y).
\]

Specialness, however, depends on the choice of trace. More precisely, suppose
that
\[
\mathcal R=\bigoplus_\alpha M_{n_\alpha},
\qquad
\tau_{\mathcal R}=\sum_\alpha c_\alpha\operatorname{Tr}_{n_\alpha},
\qquad c_\alpha>0.
\]
If \(1_\alpha\) denotes the central projection onto the summand
\(M_{n_\alpha}\), viewed also as the corresponding orthogonal projection on
\(L^2(\mathcal R)\), then
\[
m_{\mathcal R}m_{\mathcal R}^*
=
\sum_\alpha \frac{n_\alpha}{c_\alpha}1_\alpha .
\]
Hence this concrete Frobenius algebra is special precisely when
\(c_\alpha=n_\alpha\) for every \(\alpha\). More generally, the trace
\(\tau_{\mathcal R}\) is a \emph{\(\delta\)-form}, meaning that
\(m_{\mathcal R}m_{\mathcal R}^*=\delta^2 1_{L^2(\mathcal R)}\), precisely
when the ratios \(n_\alpha/c_\alpha\) are independent of \(\alpha\).

\subsection{Rule operators, projections and bimodules}

Throughout this work, the symbols \(m_{\mathcal{XY}}\) and
\(m_{\mathcal{AB}}\) denote the multiplication maps of the tensor-products \(L^2(\mathcal X)\otimes L^2(\mathcal Y)\) and
\(L^2(\mathcal A)\otimes L^2(\mathcal B)\), respectively. Thus, under the
identifications
\( 
L^2(\mathcal{XY})\cong L^2(\mathcal X)\otimes L^2(\mathcal Y)\),
\( 
L^2(\mathcal{AB})\cong L^2(\mathcal A)\otimes L^2(\mathcal B),
\)
we have
\( 
m_{\mathcal{XY}}
=
(m_{\mathcal X}\otimes m_{\mathcal Y})
(1\otimes\sigma\otimes 1)
\)
and
\( 
m_{\mathcal{AB}}
=
(m_{\mathcal A}\otimes m_{\mathcal B})
(1\otimes\sigma\otimes 1).
\)
Equivalently,
\( 
m_{\mathcal{XY}}
\bigl((\widehat{x_1}\otimes\widehat{y_1})
\otimes
(\widehat{x_2}\otimes\widehat{y_2})\bigr)
=
\widehat{x_1x_2}\otimes\widehat{y_1y_2},
\)
and similarly for \(m_{\mathcal{AB}}\). In particular, the necessary flip of
the middle tensor factors is built into the notation \(m_{\mathcal{XY}}\) and
\(m_{\mathcal{AB}}\).

\begin{definition} \label{def:gold_game} \rm
Let
\((\mathcal X,\tau_X),(\mathcal Y,\tau_Y),(\mathcal A,\tau_A),(\mathcal B,\tau_B)\)
be finite-dimensional tracial von Neumann algebras. A \emph{rule operator quantum game} or simply \emph{rule operator} from \((\mathcal X,\mathcal Y)\) to
\((\mathcal A,\mathcal B)\) is a linear map
\( 
\lambda:
L^2(\mathcal X)\otimes L^2(\mathcal Y)
\to
L^2(\mathcal A)\otimes L^2(\mathcal B)
\)
satisfying
\begin{enumerate}
    \item[(i)] $m_{\mathcal{AB}}(\lambda\otimes\lambda)m_{\mathcal{XY}}^*=\lambda$;
    \item[(ii)] $(1_{L^2(\mathcal{XY})}\otimes \eta_{\mathcal{AB}}^{*}m_{\mathcal{AB}})
(1_{L^2(\mathcal{XY})} \otimes\lambda\otimes 1_{L^2(\mathcal{AB})})
(m_{\mathcal{XY}}^{*}\eta_{\mathcal{XY}}\otimes
1_{L^2(\mathcal{AB})})
=\lambda^{*}.$
\end{enumerate}

\end{definition}

The first condition says that \(\lambda\) is idempotent for the Frobenius
convolution product induced by multiplication, while the second says that it is ``self-conjugate". For simplicity we may write $ 1_{\mathcal X\mathcal Y}$ for $\operatorname{1}_{L^2(\mathcal{XY})}$ or simply $1$ when it is clear from the context.

\begin{remark}\rm
In the classical case this construction recovers the usual non-local game
formalism. Suppose
$\mathcal X=D_X $, $ \mathcal Y = D_Y$, $ \mathcal A=D_A $ and $\mathcal B= D_B$. Then,  with respect to the canonical trace, $L^{2}(D_n) = \bb C^n$, $n \in \bb N$ and we   write
\( 
\lambda(e_{xy})=\sum_{a,b}\lambda(a,b,x,y)e_{ab}.
\)
Since \(m_{\mathcal{XY}}^*e_{xy}=e_{xy}\otimes e_{xy}\), the first Goldberg
condition gives
\[
m_{\mathcal{AB}}(\lambda\otimes\lambda)m_{\mathcal{XY}}^*e_{xy}
=
\sum_{a,b}\lambda(a,b,x,y)^2e_{ab}.
\]
Thus \(m_{\mathcal{AB}}(\lambda\otimes\lambda)m_{\mathcal{XY}}^*=\lambda\)
is equivalent to
\(\lambda(a,b,x,y)^2=\lambda(a,b,x,y)\) for all \(x,y,a,b\), and hence the
coefficients are zero or one.

The second Goldberg condition is automatic for such coefficients. Indeed,
\(m_{\mathcal{XY}}^*\eta_{\mathcal{XY}}(1)=\sum_{x,y}e_{xy}\otimes e_{xy}\),
while
\(\eta_{\mathcal{AB}}^*m_{\mathcal{AB}}(e_{cd}\otimes e_{ab})
=\delta_{(c,d),(a,b)}\). Hence, applied to \(e_{ab}\), the left-hand side of
the self-conjugacy axiom is
\[
\begin{aligned}
(1\otimes\eta_{\mathcal{AB}}^*m_{\mathcal{AB}})
(1\otimes\lambda\otimes1)
(m_{\mathcal{XY}}^*\eta_{\mathcal{XY}}\otimes1)(e_{ab})  & =
(1\otimes\eta_{\mathcal{AB}}^*m_{\mathcal{AB}})
\sum_{x,y,c,d}\lambda(c,d,x,y)e_{xy}\otimes e_{cd}\otimes e_{ab}  \\
& =
\sum_{x,y}\lambda(a,b,x,y)e_{xy}.
\end{aligned}
\]
On the other hand,
\(\lambda^*(e_{ab})=\sum_{x,y}\overline{\lambda(a,b,x,y)}e_{xy}\).
Thus the second condition says that the coefficients are real. Since the first
condition already forces them to be \(0\) or \(1\), it is automatic. Therefore
Goldberg rule operators on classical diagonal algebras are exactly classical
rule predicates.
\end{remark}

Following \cite{Daws26}, we will use the following identification. If \(a\in\mathcal N\) and
\(b\in\mathcal M\), write
\( 
\widehat a\,\widehat b^{\,*}:L^2(\mathcal M)\to L^2(\mathcal N)\), for the rank-one operator
\( 
(\widehat a\,\widehat b^{\,*})(\widehat x)
=
\langle \widehat b,\widehat x\rangle\,\widehat a
=
\tau_{\mathcal M}(b^*x)\widehat a .
\)
Define
\(\Psi_{\M,\N}:\B(L^2(\mathcal M),L^2(\mathcal N))\to
\mathcal N\otimes\mathcal M^{\op}\) on rank-one operators by
\[
\Psi_{\M,\N}(\widehat a\,\widehat b^{\,*})=a\otimes b^* .
\]
and note that $\Psi_{\M,\N}$ is a linear isomorphism. We write $ \Psi_{\M}:= \Psi_{\M,\M}$. Let
\( 
\sigma_{\M,\N}:\mathcal N\otimes\mathcal M^{\op}
\to
\mathcal M\otimes\mathcal N^{\op}
\)
be the anti-*-homomorphism given by
\(\sigma_{\mathcal M,\mathcal N}(a\otimes b^{\op})=b\otimes a^{\op}\). We write $ \sigma_{\M}:= \sigma_{\M,\M}$.

Item (i) in the following theorem follows as the tracial version of \cite[Theorem 5.3]{Daws26}. See also \cite[Theorem 7.7]{MRV18}, \cite[Proposition 5.5]{Daws24}. We provide the proof for completeness.
\begin{theorem}\label{thm:goldberg-daws-transform}
Let \((\mathcal M,\tau_{\mathcal M})\),  \((\mathcal N,\tau_{\mathcal N})\) be finite-dimensional tracial von Neumann
algebras. For \(S,T\in B(L^2(\mathcal M),L^2(\mathcal N))\), one has
\begin{enumerate}
\item[(i)]
\( 
\Psi_{\M,\N}\big(m_{\mathcal N}(S\otimes T)m_{\mathcal M}^*\big)
=
\Psi_{\M,\N}(S)\Psi_{\M,\N}(T);
\)
\item[(ii)]
\( 
\Psi_{\N,\M}\Big(
(1\otimes \eta_{\mathcal N}^*m_{\mathcal N})
(1\otimes T\otimes 1)
(m_{\mathcal M}^*\eta_{\mathcal M}\otimes 1)
\Big)
=
\sigma_{\mathcal M,\mathcal N}(\Psi_{\M,\N}(T));
\)
\item[(iii)] $\eta^*_{\mathcal N} T \eta_{\mathcal M}(1)= \tau_{\mathcal N} \otimes \tau_{\mathcal M^{\op}}(\Psi_{\M,\N}(T))$
\end{enumerate}
Moreover,
\( 
\sigma_{\mathcal N,\mathcal M}(\Psi_{\N,\M}(T^*))=\Psi_{\M,\N}(T)^* .
\)
\end{theorem}

\begin{proof}
It suffices throughout to check rank-one operators.

(i) Let \(S=\widehat a\,\widehat b^{\,*}\) and
\(T=\widehat c\,\widehat d^{\,*}\). For \(x\in\mathcal M\) and
\(y\in\mathcal N\),  we get
\begin{align*}
\big\langle \widehat y,
m_{\mathcal N}(S\otimes T)m_{\mathcal M}^*\widehat x
\big\rangle
&=
\big\langle \widehat b\otimes \widehat d,
m_{\mathcal M}^*\widehat x\big\rangle
\big\langle m_{\mathcal N}^* \widehat y,
\widehat a \otimes \widehat c\big\rangle\\
&=
\big\langle \widehat{bd},\widehat x\big\rangle
\big\langle \widehat y,\widehat{ac}\big\rangle
\\
&=
\big\langle \widehat y,\widehat{ac}\,\widehat{bd}^{\,*}\widehat x\big\rangle.
\end{align*}
Thus \(m_{\mathcal N}(S\otimes T)m_{\mathcal M}^*
=\widehat{ac}\,\widehat{bd}^{\,*}\). Hence
\[\Psi_{\M,\N}(m_{\mathcal N}(S\otimes T)m_{\mathcal M}^*)
=ac\otimes (bd)^*=ac\otimes d^*b^*
=(a\otimes b^*)(c\otimes d^*)=\Psi_{\M,\N}(S)\Psi_{\M,\N}(T),\] where the second tensor
factor is \(\mathcal M^{\op}\).

(ii) Let \(T=\widehat a\,\widehat b^{\,*}\), and write
\(m_{\mathcal M}^*\eta_{\mathcal M}(1)=\sum_i \xi_i\otimes \zeta_i\) as a
finite sum of simple tensors. For \(\widehat z\in L^2(\mathcal N)\), we have 
\begin{align*}
&(1\otimes \eta_{\mathcal N}^*m_{\mathcal N})
(1\otimes T\otimes 1)
(m_{\mathcal M}^*\eta_{\mathcal M}\otimes 1)(\widehat z) \\
&\qquad \qquad =
(1\otimes \eta_{\mathcal N}^*m_{\mathcal N})
\Bigl(\sum_i
\xi_i\otimes
\langle\widehat b,\zeta_i\rangle\,
\widehat a\otimes\widehat z
\Bigr)\\
&\qquad \qquad =
\sum_i
\xi_i\,\langle\widehat b,\zeta_i\rangle\,
\eta_{\mathcal N}^*m_{\mathcal N}
(\widehat a\otimes\widehat z)\\
&\qquad \qquad=
\sum_i
\xi_i\,\langle\widehat b,\zeta_i\rangle\,
\tau_{\mathcal N}(az)
=
\sum_i
\xi_i\,\langle\widehat b,\zeta_i\rangle\,
\langle\widehat{a^*},\widehat z\rangle .
\end{align*}
 Moreover,
by the defining property of \(m_{\mathcal M}^*\eta_{\mathcal M}(1)\), one has
\(\sum_i\xi_i\langle\widehat b,\zeta_i\rangle=\widehat{b^*}\). Indeed, pairing the left hand side with \(\widehat x\in L^2(\mathcal M)\) gives
\(\langle \widehat x\otimes\widehat b,m_{\mathcal M}^*\eta_{\mathcal M}(1)\rangle
=\langle\widehat{xb},\widehat 1\rangle
=\langle\widehat x,\widehat{b^*}\rangle\).
Hence  \[ (1\otimes \eta_{\mathcal N}^*m_{\mathcal N})
(1\otimes T\otimes 1)
(m_{\mathcal M}^*\eta_{\mathcal M}\otimes 1) =\widehat{b^*}\,\widehat{a^*}^{\,*}.\] Applying \(\Psi_{\N,\M}\), we get
\[
\Psi_{\N,\M}(\widehat{b^*}\,\widehat{a^*}^{\,*})
=
b^*\otimes a
=
\sigma_{\mathcal M,\mathcal N}(a\otimes b^*)
=
\sigma_{\mathcal M,\mathcal N}(\Psi_{\M,\N}(T)),
\]
which proves the claim (ii).

(iii)  Let
\(T=\widehat a \widehat{b}^*\), with \(a\in\mathcal N\) and
\(b\in\mathcal M\).  Then \(T\widehat x=\tau_{\mathcal M}(b^*x)\widehat a\)
and hence
\[
(\tau_{\mathcal N}\otimes\tau_{\mathcal M^{\op}})(\Psi_{\M,\N}(T))
=
\tau_{\mathcal N}(a)\tau_{\mathcal M}(b^*).
\]
On the other hand,
\[
\eta_{\mathcal N}^*T\eta_{\mathcal M}(1)
=
\eta_{\mathcal N}^*T\widehat 1_{\mathcal M}
=
\eta_{\mathcal N}^*(\tau_{\mathcal M}(b^*)\widehat a)
=
\tau_{\mathcal M}(b^*)\tau_{\mathcal N}(a).
\]
Thus the two sides agree on rank-one maps, and therefore, by linearity, on all
\(T\in \B(L^2(\mathcal M),L^2(\mathcal N))\).

Finally, for \(T^*=\widehat b\,\widehat a^{\,*}\), so
\[ \sigma_{\mathcal N,\mathcal M}(\Psi_{\N,\M}(T^*))
=\sigma_{\mathcal N,\mathcal M}(b\otimes a^*)=a^*\otimes b
=(a\otimes b^*)^*=\Psi_{\M,\N}(T)^*.\]
\end{proof}

\begin{corollary}\label{cor:goldberg-rule-projection}
Let
\(\mathcal M=\mathcal X\otimes\mathcal Y\) and
\(\mathcal N=\mathcal A\otimes\mathcal B\),  and let
\(\lambda:L^2(\mathcal M)\to L^2(\mathcal N)\) be a rule operator quantum game. Set
\(Q_{\lambda}=\Psi_{\M,\N}(\lambda)\in\mathcal N\otimes\mathcal M^{\op}\). Then
\(Q_{\lambda}\) is a projection.
\end{corollary}

\begin{proof}
The definition and Theorem~\ref{thm:goldberg-daws-transform}(i)
give
\[
Q_{\lambda}^2
=
\Psi_{\M,\N}\big(m_{\mathcal N}(\lambda\otimes\lambda)m_{\mathcal M}^*\big)
=
\Psi_{\M,\N}(\lambda)
=
Q_{\lambda} .
\]
It remains to prove self-adjointness. Applying
Theorem~\ref{thm:goldberg-daws-transform}(ii) to the second Goldberg condition
gives
\(\sigma_{\mathcal M,\mathcal N}(Q_{\lambda})=\Psi_{\N,\M}(\lambda^*)\). Hence, applying
\(\sigma_{\mathcal N,\mathcal M}\) and using the final  identity in
Theorem~\ref{thm:goldberg-daws-transform},
\[
Q_{\lambda}
=
\sigma_{\mathcal N,\mathcal M}\sigma_{\mathcal M,\mathcal N}(Q_{\lambda})
=
\sigma_{\mathcal N,\mathcal M}(\Psi_{\N,\M}(\lambda^*))
=
Q_{\lambda}^* .
\]
Thus \(Q_{\lambda}=Q_{\lambda}^*=Q_{\lambda}^2\).
\end{proof}

We now pass from the projection \(Q_{\lambda}\) to an operator bimodule.  Recall the discussion before Lemma \ref{lem:vec-bimodule-action-general}.  Let
\(\mathcal M\subseteq \B(H)\) and \(\mathcal N\subseteq \B(K)\) be faithful
finite-dimensional representations.  We represent \(\mathcal M^{\op}\subseteq \B(\overline{H})\). 
Let
\( 
\operatorname{vec}_{}:\B(H,K)\to K\otimes\overline H
\)
be the vectorisation determined by
\( 
\operatorname{vec}_{}(\xi\eta^*)=\xi\otimes\overline\eta\), \(  \xi\in K,\ \eta\in H.
\)
By Lemma~\ref{lem:vec-bimodule-action-general}, for
\(x'\in\mathcal N'\), \(y'\in\mathcal M'\), and \(T\in \B(H,K)\), we have
\[
\operatorname{vec}_{}(x'Ty')
=
\bigl(x'\otimes (y')^{\op}\bigr)\operatorname{vec}_{}(T).
\]
Thus the \(\mathcal N'\)-\(\mathcal M'\) bimodule action on \(\B(H,K)\) is
transported, under vectorisation, to the natural action of
\(\mathcal N'\otimes(\mathcal M')^{\op}\) on \(K\otimes\overline H\).

The following fact appears in \cite{Daws26} (see the discussion after Example 5.5 therein). Compare also with \cite[Proposition 2.23]{Wea12}, \cite[Proposition 7.11]{MRV18} and \cite[Corollary 5.28]{Daws24}. We include a proof for completeness.
\begin{proposition}\label{prop:rectangular-daws}
Let \(\mathcal M\subseteq \B(H)\) and \(\mathcal N\subseteq \B(K)\) be von Neumann algebras. There is a bijection between projections in 
\(\mathcal N\otimes\mathcal M^{\op}\) and
\(\mathcal N'\)-\(\mathcal M'\) operator bimodules of
\( \B(H,K)\). It is given by
\[
Q\longmapsto
\operatorname{vec}_{}^{-1}
\bigl(Q(K\otimes\overline H)\bigr).
\]
Conversely, if \(\mathcal U\subseteq \B(H,K)\) is an
\(\mathcal N'\)-\(\mathcal M'\) bimodule, then the corresponding projection 
 \(Q_{\mathcal U}\in\mathcal N\otimes\mathcal M^{\op}\)  is the orthogonal projection onto
\( 
\operatorname{vec}_{}(\mathcal U)
\subseteq K\otimes\overline H .
\)
\end{proposition}

\begin{proof}
By the preceding discussion, a subspace
\(\mathcal U\subseteq \B(H,K)\) is an \(\mathcal N'\)-\(\mathcal M'\) operator
bimodule if and only if \(\operatorname{vec}_{}(\mathcal U)\) is invariant
under
\( 
\mathcal N'\otimes(\mathcal M')^{\op}.
\)
Since this algebra is self-adjoint, invariance is equivalent to reducingness.
Thus the orthogonal projection \(Q_{\mathcal U}\) onto
\(\operatorname{vec}_{}(\mathcal U)\) commutes with
\(\mathcal N'\otimes(\mathcal M')^{\op}\). Hence
\[
Q_{\mathcal U}
\in
\bigl(\mathcal N'\otimes(\mathcal M')^{\op}\bigr)'
=
\mathcal N\otimes\mathcal M^{\op}
\]
inside \(\B(K\otimes\overline H)\). 

Conversely, if \(Q\in\mathcal N\otimes\mathcal M^{\op}\) is a projection, then
\(Q\) commutes with \(\mathcal N'\otimes(\mathcal M')^{\op}\). Hence
\( 
Q(K\otimes\overline H)
\)
is invariant under \(\mathcal N'\otimes(\mathcal M')^{\op}\). By the
vectorisation identity, the space
\( 
\operatorname{vec}_{}^{-1}
\bigl(Q(K\otimes\overline H)\bigr)
\)
is therefore an \(\mathcal N'\)-\(\mathcal M'\) operator bimodule. The two
constructions are inverse to one another.
\end{proof}

\subsection{Connection to projection-test quantum games}
\label{subsec:connect_coop}

Let \(\mathcal M=\mathcal X\otimes\mathcal Y\) and
\(\mathcal N=\mathcal A\otimes\mathcal B\). Throughout this work we use the canonical identification
$L^2(\mathcal X)\otimes L^2(\mathcal Y)
\cong L^2(\mathcal M)$.  Let
\(\lambda:L^2(\mathcal M)\to L^2(\mathcal N)\) be a rule operator quantum game and
set \(Q_\lambda=\Psi_{\M,\N}(\lambda)\in\mathcal N\otimes\mathcal M^{\op}\).  By
Corollary~\ref{cor:goldberg-rule-projection}, \(Q_\lambda\) is a projection.   

Choose faithful finite-dimensional representations
\( 
\mathcal X\subseteq \B(H_X)\), 
\(\mathcal Y\subseteq \B(H_Y)\), 
\( \mathcal A\subseteq \B(H_A)\) and
\( \mathcal B\subseteq \B(H_B).
\)
Put
\( 
H_{\rm in}=H_X\otimes H_Y\),
\( 
H_{\rm out}=H_A\otimes H_B,
\)
and represent \(\mathcal M\) and \(\mathcal N\) on these Hilbert spaces in the
natural tensor-product representations.    By
Proposition~\ref{prop:rectangular-daws}, \(Q_\lambda\) determines the concrete
operator bimodule
\( 
\mathcal U_\lambda
=
\operatorname{vec}^{-1}
\bigl(Q_\lambda(H_{\rm out}\otimes\overline{H_{\rm in}})\bigr)
\subseteq \B(H_{\rm in},H_{\rm out}).
\)
This is an \(\mathcal N'\)-\(\mathcal M'\) bimodule, where the commutants are
taken in \(\B(H_{\rm out})\) and \(\B(H_{\rm in})\), respectively.  Thus the
passage is
\[
\lambda
\overset{\Psi_{\M,\N}}{\longmapsto}
Q_\lambda\in\mathcal N\otimes\mathcal M^{\op}
\longmapsto
\mathcal U_\lambda\subseteq \B(H_{\rm in},H_{\rm out}).
\]

Conversely, every \(\mathcal N'\)-\(\mathcal M'\) operator bimodule
\(\mathcal U\subseteq \B(H_{\rm in},H_{\rm out})\) gives, by
Proposition~\ref{prop:rectangular-daws}, a unique projection
\(Q_{\mathcal U}\in\mathcal N\otimes\mathcal M^{\op}\).  Applying the inverse transform gives
\[
\lambda_{\mathcal U}:=\Psi_{\M,\N}^{-1}(Q_{\mathcal U})
:
L^2(\mathcal M)\to L^2(\mathcal N).
\]
Since \(Q_{\mathcal U}\) is a projection,
Theorem~\ref{thm:goldberg-daws-transform} implies that
\(\lambda_{\mathcal U}\) is a rule operator.  Hence, after fixing
faithful finite-dimensional representations,  rule operators are
equivalent to concrete \(\mathcal N'\)-\(\mathcal M'\) operator bimodules.

The  $\cl R$-projection-test quantum game  associated with \(\lambda\) is
\(G_\lambda=(\Pi_{\mathcal M},Q_{\lambda})\) with referee algebra $ \cl R= \cl M^{\rm op}$, input projection 
\( 
\Pi_{\mathcal M}:=P_{\operatorname{vec}(\mathcal M')} \in  \mathcal M\otimes \cl M^{\rm op},
\) 
and  accepting projection 
\( Q_\lambda:= \Psi_{\M,\N}(\lambda) \in \mathcal N \otimes \mathcal M^{\op}.\)

\subsection{Values and channels}

Let \(L^2(\mathcal M)\) and \(L^2(\mathcal N)\) be  realised from finite-dimensional tracial von Neumann algebras. For a
linear map \(f:L^2(\mathcal M)\to L^2(\mathcal N)\), define
\[
\widetilde f
:=
(1 \otimes m_{\mathcal N})
(1 \otimes f\otimes 1 )
(m^*_{\mathcal M}\otimes 1)
\in \B(L^2(\mathcal M)\otimes L^2(\mathcal N)).
\]
We say that \(f\) is \emph{completely positive} if \(\widetilde f\) is a
positive operator. Equivalently, there exist a finite-dimensional Hilbert space
\(H\) and a linear map \(g:L^2(\mathcal M)\otimes L^2(\mathcal N)\to H\) such that
\( 
\widetilde f=g^*g.
\)
We say that  $f$ is a  \emph{channel} if it is  completely positive and \emph{co-unital} in that
\( 
\eta^*_{\mathcal N}f=\eta^*_{\mathcal M}.
\) Equivalently, $ \langle \widehat 1, f(x)\rangle= \langle \widehat 1,x\rangle$, $ x\in L^{2}(\mathcal M)$.

Let
\( 
(\mathcal X,\tau_{\mathcal X})\),
\( (\mathcal Y,\tau_{\mathcal Y})\),
\((\mathcal A,\tau_{\mathcal A})\) and 
\((\mathcal B,\tau_{\mathcal B})
\)
be finite-dimensional tracial von Neumann algebras. 
A \emph{correlation} from \((\mathcal X,\mathcal Y)\) to
\((\mathcal A,\mathcal B)\) is a channel
\( 
f:L^{2}(\mathcal{X}) \otimes L^{2}(\mathcal Y)\to L^{2}(\mathcal{A}) \otimes L^{2}(\mathcal B).
\)

\medskip 

Since \(\tau_{\mathcal M}\) is tracial, \(\mathcal M^{\op}\) is faithfully
represented on \(L^2(\mathcal M)\) by right multiplication: we write
\(R(c^{\op})\widehat x=\widehat{xc}\).  Similarly, \(\mathcal N\) is
faithfully represented on \(L^2(\mathcal N)\) by left multiplication:
\(L(a)\widehat y=\widehat{ay}\).  We use the corresponding faithful
representation
\[
\rho:\mathcal N\otimes\mathcal M^{\op}
\to B(L^2(\mathcal M)\otimes L^2(\mathcal N)),
\qquad
\rho(a\otimes c^{\op})(\widehat x\otimes\widehat y)
=
\widehat{xc}\otimes\widehat{ay}.
\]
Equivalently, if
\(\Sigma:L^2(\mathcal M)\otimes L^2(\mathcal N)\to
L^2(\mathcal N)\otimes L^2(\mathcal M)\) is the tensor flip, then
\[
\rho(a\otimes c^{\op})
=
\Sigma^*(L(a)\otimes R(c^{\op}))\Sigma .
\]
Thus \(\rho\) is unitarily equivalent to the tensor product of the faithful
representations \(L\) and \(R\), and hence is faithful.

\begin{proposition} \label{prop:cptop}
Let \((\mathcal M,\tau_{\mathcal M})\) and
\((\mathcal N,\tau_{\mathcal N})\) be finite-dimensional tracial von Neumann
algebras. A linear map \(f:L^2(\mathcal M)\to L^2(\mathcal N)\) is completely
positive in the above sense if and only if
\( 
\Psi_{\M,\N}(f)\in(\mathcal N\otimes\mathcal M^{\mathrm{op}})_+.
\)
\end{proposition}

\begin{proof}
Let
\( 
\rho:\mathcal N\otimes\mathcal M^{\mathrm{op}}
\to B(L^2(\mathcal M)\otimes L^2(\mathcal N))
\) 
be the faithful representation given as above.
We claim that
\( 
\widetilde f=\rho(\Psi_{\M,\N}(f)).
\)
By linearity, it suffices to check this for
\(f=\widehat a\,\widehat b^{\,*}\), where
\(a\in\mathcal N\) and \(b\in\mathcal M\). In this case
\( 
f(\widehat x)=\widehat a\,\tau_{\mathcal M}(b^*x).
\)
Working blockwise, suppose first that \(\mathcal M=M_n\) and
\(\tau_{\mathcal M}=c\operatorname{Tr}\). Then
\[
m^*_{\mathcal M}(\widehat{E_{ij}})
=
c^{-1}\sum_t \widehat{\varepsilon_{i,t}}\otimes\widehat{\varepsilon_{t,j}}.
\]
Hence, for \(y\in\mathcal N\),
\[
\begin{aligned}
\widetilde f(\widehat{\varepsilon_{i,j}}\otimes\widehat y)
&=
c^{-1}\sum_t
\widehat{\varepsilon_{i,t}}\otimes
f(\widehat{\varepsilon_{t,j}}) \widehat y  \\
&=
c^{-1}\sum_t
\widehat{\varepsilon_{i,t}}\otimes
\widehat{a} \, \tau_{\mathcal M}(b^*\varepsilon_{t,j})\, \widehat y  \\
&= c^{-1}\sum_t
\widehat{\varepsilon_{i,t}}\, \tau_{\mathcal M}(\varepsilon_{t,j}b^*) \otimes
\widehat{a}\, \widehat y  \\
&=
\widehat{\varepsilon_{i,j}b^*}\otimes\widehat{ay}.
\end{aligned}
\]
Here we used the expansion $\widehat z=c^{-1}\sum_{r,s} \widehat \varepsilon_{r,s}\tau_{\mathcal M}(\varepsilon_{s,r}z)$ for an element $ z \in M_n$.
Thus, since the multiplication maps respect the block structure,  for arbitrary \(x\in\mathcal M\),
\[
\widetilde f(\widehat x\otimes\widehat y)
=
\widehat{xb^*}\otimes\widehat{ay}
=
\rho(a\otimes(b^*)^{\mathrm{op}})
(\widehat x\otimes\widehat y).
\]
Therefore
\( 
\widetilde{\widehat a\,\widehat b^{\,*}}
=
\rho\bigl(\Psi_{\M,\N}(\widehat a\,\widehat b^{\,*})\bigr),
\)
and hence \(\widetilde f=\rho(\Psi_{\M,\N}(f))\) for all \(f\). It follows that
\[
\widetilde f\geq 0
\quad\Longleftrightarrow\quad
\rho(\Psi_{\M,\N}(f))\geq 0.
\]
Since \(\rho\) is faithful, this is equivalent to
\( 
\Psi_{\M,\N}(f)\geq 0
\)
inside the finite-dimensional \(C^*\)-algebra
\(\mathcal N\otimes\mathcal M^{\mathrm{op}}\). Finally, positivity in a
 \(C^*\)-algebra is equivalent to being of the form
\(h^*h\). This proves the claim.
\end{proof}

If \(\mathcal M\) is equipped with a faithful trace \(\tau_{\mathcal M}\), we
identify \(\mathcal M\) with its predual through the trace pairing
\[
h\in\mathcal M
\longmapsto
\rho_h\in\mathcal M_*,
\qquad
\rho_h(x)=\tau_{\mathcal M}(xh),
\quad x\in\mathcal M .
\]
Under this identification, positive normal functionals correspond to positive
elements of \(\mathcal M\), and normal states correspond to positive
\(h\in\mathcal M\) satisfying \(\tau_{\mathcal M}(h)=1\).

Schrödinger-picture quantum channels are intrinsically written as completely positive
state-preserving maps
\( 
\Gamma:\mathcal M_*\to\mathcal N_* .
\)
Equivalently, the Banach adjoint
\( 
\Gamma^*:\mathcal N\to\mathcal M
\)
is normal, unital and completely positive.  When faithful traces
\(\tau_{\mathcal M}\) and \(\tau_{\mathcal N}\) are fixed, we also write
\(\Gamma_\tau:\mathcal M\to\mathcal N\) for the corresponding density-picture
map, defined by
\[
\tau_{\mathcal N}(\Gamma_\tau(h)y)
=
\tau_{\mathcal M}(h\Gamma^*(y)),
\qquad h\in\mathcal M,\ y\in\mathcal N .
\]
The map \(\Gamma_\tau\) is completely positive and trace-preserving, in the
sense that
\( 
\tau_{\mathcal N}\circ\Gamma_\tau=\tau_{\mathcal M} 
\) (see for example \cite[Proposition 7.11]{Daws24}).
When no confusion can arise, we suppress the subscript \(\tau\) and denote the
density-picture representative again by \(\Gamma\).

Since \(\mathcal M\) and \(\mathcal N\) are finite-dimensional, every element of
\(\B(\mathcal M,\mathcal N)\) is a finite sum of rank-one maps.  Using the trace
pairing, every functional on \(\mathcal M\) is of the form
\(\rho_{b}: x\mapsto \tau_{\mathcal M}(bx)\) for a unique \(b\in\mathcal M\).  Hence
\(B(\mathcal M,\mathcal N) \cong \mathcal N \otimes \mathcal M^{*}\) is spanned by the maps
\( 
\mathcal M\to\mathcal N\),
\( a \otimes \rho_b: x \mapsto \tau_{\mathcal M}(bx)a,
\)
with \(a\in\mathcal N\) and \(b\in\mathcal M\).

\begin{proposition}\label{prop:L2-channel-algebra-channel}
Let \((\mathcal M,\tau_{\mathcal M})\) and
\((\mathcal N,\tau_{\mathcal N})\) be finite-dimensional tracial von Neumann
algebras, and let \(f:L^2(\mathcal M)\to L^2(\mathcal N)\) be a channel. Let \(\Gamma: \mathcal M \to\mathcal N\) be the unique linear
map determined by \(f(\widehat x)=\widehat{\Gamma(x)}\). Then \(\Gamma\) is
completely positive and trace-preserving in the usual operator-algebraic sense.
Conversely, every completely positive trace-preserving map
\(\Gamma:\mathcal M \to\mathcal N\) gives a channel
\(f:L^2(\mathcal M)\to L^2(\mathcal N)\) by \(f(\widehat x)=\widehat{\Gamma(x)}\).
\end{proposition}

\begin{proof}
 Define
\( 
\Psi':B(\mathcal M,\mathcal N)\to
\mathcal N\otimes\mathcal M^{\op}
\)
on rank-one maps by
\( 
\Psi'(a \otimes \rho_b)
=
a\otimes b^{\op}.
\)
By \cite[Theorem 5.36]{Daws24}, \(\Psi'\) restricts to a bijection between
completely positive maps \(\mathcal M\to\mathcal N\) and positive elements of
\(\mathcal N\otimes\mathcal M^{\op}\).
By the previous proposition, \(f\) is completely positive  if and only if the associated Choi element
\(\Psi_{\M,\N}(f)\in\mathcal N\otimes\mathcal M^{\op}\) is positive. Under the
identification \(f(\widehat x)=\widehat{\Gamma(x)}\), this is precisely the
usual Choi element of \(\Gamma\), that is, \(\Psi'(\Gamma)=\Psi_{\M,\N}(f)\). Indeed, assume $ \Gamma =a \otimes \rho_b$. Then $ \Psi'(\Gamma) = a \otimes b^{\op}$ while $ f(\widehat x) = \widehat {\Gamma(x)}= \tau_{\mathcal M}(bx)\widehat a= \widehat a (\widehat{b^*})^*(\widehat x)$ so that $ f= \widehat a (\widehat{b^*})^*$ and $ \Psi_{\M,\N}(\widehat a (\widehat{b^*})^*)= a\otimes b ^{\op}$.
Hence \(f\) is completely positive
if and only if \(\Gamma\) is completely positive in the usual
operator-algebraic sense.

It remains to check the counit condition. Since \(f\) is a channel,
\(\eta^*_{\mathcal N}f=\eta^*_{\mathcal M}\). Thus, for every
\(x\in\mathcal M\),
\[
\tau_{\mathcal N}(\Gamma(x))
=
\langle \widehat{1_{\mathcal N}},\widehat{\Gamma(x)}\rangle
=
\langle \widehat{1_{\mathcal N}},f(\widehat x)\rangle
=
\langle \widehat{1_{\mathcal M}},\widehat x\rangle
=
\tau_{\mathcal M}(x).
\]
Therefore \(\Gamma\) is trace-preserving. The converse follows by reversing the arguments.
\end{proof}

Let $ \mathcal M \subseteq\B(H)$ be a von Neumann algebra. Recall that for an element $ a\in \mathcal M_{+}$, its support projection is defined as the smallest projection $p$ such that $a=ap=pa$. We denote this projection by $ s(a)$ and since $\mathcal M$ is a von Neumann algebra $ s(a)\in \mathcal M$. In particular for finite dimensional $\cl M$, $s(a)= P_{\operatorname{Ran} (a)}$.
\begin{lemma}\label{lem:support-Psi-id}
Let \((\mathcal M,\tau_{\mathcal M})\) and
\((\mathcal N,\tau_{\mathcal N})\) be finite-dimensional tracial von Neumann
algebras. Let \(f:L^2(\mathcal M)\to L^2(\mathcal N)\) and let
\(\Gamma_f:\mathcal M\to\mathcal N\) be defined by
\(f(\widehat x)=\widehat{\Gamma_f(x)}\). Then
\[
(\Gamma_f\otimes{\rm id}_{\mathcal M^{\op}})
\bigl(\Psi_{\M}(1_{L^2(\mathcal M)})\bigr)=\Psi_{\M,\N}(f).
\]
Moreover, if \(\mathcal M\subseteq \B(H)\) is any faithful 
representation and \(\Psi_{\M}(1_{L^2(\mathcal M)})\) is represented on
\(H\otimes\overline H\), then
\[
s\bigl(\Psi_{\M}(1_{L^2(\mathcal M)})\bigr)
=\Pi_{\mathcal M}
.
\]
\end{lemma}

\begin{proof}
For the first identity it suffices to check rank-one maps. Let
\(\Gamma_f=a \otimes \rho_b\), where \(\rho_b(x)=\tau_{\mathcal M}(bx)\). Then
\(f=\widehat a(\widehat{b^*})^*\) and \(\Psi_{\M,\N}(f)=a\otimes b^{\op}\). If
\((\widehat e_i)_i\) is an orthonormal basis of \(L^2(\mathcal M)\), then
\(\Psi_{\M}(1_{L^2(\mathcal M)})=\sum_i e_i\otimes(e_i^*)^{\op}\). Hence
\[
(\Gamma_f\otimes\id)(\Psi_{\M}(1_{L^2(\mathcal M)}))
=
\sum_i\tau_{\mathcal M}(be_i)a\otimes(e_i^*)^{\op}
=
a\otimes b^{\op},
\]
using \(\sum_i\tau_{\mathcal M}(be_i)e_i^*=b\). This proves the first claim.

For the support statement, since \(\mathcal M\subseteq \B(H)\) is already
represented faithfully, we may, up to a unitary, write
\[
H=\bigoplus_\alpha (K_\alpha\otimes E_\alpha),
\qquad
\mathcal M=
\bigoplus_\alpha (\B(K_\alpha)\otimes 1_{E_\alpha}),
\]
where \(K_\alpha=\mathbb C^{n_\alpha}\) and the \(E_\alpha\)'s are non-zero
multiplicity spaces.  Thus an element \(x=(x_\alpha)_\alpha\in\mathcal M\)
acts as \(\bigoplus_\alpha(x_\alpha\otimes1_{E_\alpha})\), and
\[
\mathcal M'
=
\bigoplus_\alpha (1_{K_\alpha}\otimes B(E_\alpha)).
\]
Write \(\tau_{\mathcal M}=\sum_\alpha c_\alpha\Tr_{K_\alpha}\), with
\(c_\alpha>0\).  The support computation is blockwise, so fix one summand
\(\B(K)\cong M_n\) with trace \(c\Tr_n\).  For matrix units
\((\varepsilon_{i,j})\),
\[
\Psi_{\M}(1_{L^2(M_n)})
=
c^{-1}\sum_{i,j}\varepsilon_{i,j}\otimes\varepsilon_{j,i}^{\op}.
\]
Represented on \((K\otimes E)\otimes\overline{(K\otimes E)}\), this is, up to
the canonical ordering of tensor legs,
\[
\frac{n}{c}\,\Omega_n\Omega_n^*\otimes1_{E\otimes\overline E},
\qquad
\Omega_n=n^{-1/2}\sum_i e_i\otimes\overline e_i.
\]
Its support is
\(\Omega_n\Omega_n^*\otimes1_{E\otimes\overline E}\).  On the other hand,
\(\operatorname{vec}(1_K\otimes B(E))
=
\mathbb C\Omega_n\otimes E\otimes\overline E\), with the same tensor-leg
ordering.  Hence on each summand
\(s(\Psi_{\M}(1_{L^2(M_n)}))=P_{\operatorname{vec}(M_n')}\).  Taking the direct sum
over \(\alpha\) gives
\( 
s\bigl(\Psi_{\M}(1_{L^2(\mathcal M)})\bigr)
=
\Pi_{\mathcal M}.
\)
\end{proof}

Let \(f,\lambda:L^2(\mathcal X)\otimes L^2(\mathcal Y)\to L^2(\mathcal A)\otimes L^2(\mathcal B)\), where $ \lambda$ is a quantum game and $ f$ is a  correlation. The \emph{winning probability} of \(f\) for the game \(\lambda\) is the scalar
\[
\omega(\lambda,f)
:=
\tau_{\mathcal M}(1)^{-1}
\eta_{\mathcal N}^*
m_{\mathcal N}(f\otimes\lambda)m_{\mathcal M}^*
\eta_{\mathcal M}(1).
\]
A correlation $ f$ is called \emph{perfect}~\cite{Goldberg26}  for the quantum game $ \lambda$ if 
\[
m_{\mathcal A\mathcal B}(f\otimes\lambda)m^*_{\mathcal X\mathcal Y}=f.
\]

\begin{lemma}\label{lem:perfect-losing-states}
Put
\(\mathcal M=\mathcal X\otimes\mathcal Y\),
\(\mathcal N=\mathcal A\otimes\mathcal B\), and
\( 
\lambda^\perp:=\eta_{\mathcal N}\eta_{\mathcal M}^*-\lambda .
\)
Let \(f:L^2(\mathcal M)\to L^2(\mathcal N)\) be a correlation. Then
\( 
0\leq \omega(\lambda,f)\leq 1.
\)
Moreover, the following are equivalent:
\begin{enumerate}
\item[\rm(i)] \(f\) is perfect for \(\lambda\);
\item[\rm(ii)]
\(m_{\mathcal N}(f\otimes\lambda^\perp)m_{\mathcal M}^*=0\);
\item[\rm(iii)] \(\omega(\lambda,f)=1\).
\end{enumerate}
\end{lemma}

\begin{proof}
Let \(Q_\lambda=\Psi_{\M,\N}(\lambda)\). Since
\(\Psi_{\M,\N}(\eta_{\mathcal N}\eta_{\mathcal M}^*)=1_{\mathcal N\otimes\mathcal M^{\op}}\),
we have \(\Psi_{\M,\N}(\lambda^\perp)=Q_\lambda^\perp\). By the trace identity for
\(\Psi_{\M,\N}\) from Theorem \ref{thm:goldberg-daws-transform} and the product formula
\(\Psi_{\M,\N}(m_{\mathcal N}(f\otimes\lambda)m_{\mathcal M}^*)=\Psi_{\M,\N}(f)Q_\lambda\),
\[
\omega(\lambda,f)
=
\tau_{\mathcal M}(1)^{-1}
(\tau_{\mathcal N}\otimes\tau_{\mathcal M^{\op}})
(\Psi_{\M,\N}(f)Q_\lambda).
\]
Since \(f\) is a correlation by Proposition \ref{prop:cptop}, \(\Psi_{\M,\N}(f)\geq0\), and \(Q_\lambda\) is a
projection. Thus, by traciality,
\[
(\tau_{\mathcal N}\otimes\tau_{\mathcal M^{\op}})
(\Psi_{\M,\N}(f)Q_\lambda)
=
(\tau_{\mathcal N}\otimes\tau_{\mathcal M^{\op}})
(Q_\lambda\Psi_{\M,\N}(f)Q_\lambda)\geq0.
\]
The same argument applied to \(Q_\lambda^\perp\) shows that the losing value
\( 
\tau_{\mathcal M}(1)^{-1}
(\tau_{\mathcal N}\otimes\tau_{\mathcal M^{\op}})
(\Psi_{\M,\N}(f)Q_\lambda^\perp)
\)
is also non-negative. Since \(Q_\lambda+Q_\lambda^\perp=1\), the winning and
losing values sum to
\[
\tau_{\mathcal M}(1)^{-1}
(\tau_{\mathcal N}\otimes\tau_{\mathcal M^{\op}})(\Psi_{\M,\N}(f))
=
\tau_{\mathcal M}(1)^{-1}\eta_{\mathcal N}^*f\eta_{\mathcal M}(1)
=1,
\]
where the last equality uses the counitality
\(\eta_{\mathcal N}^*f=\eta_{\mathcal M}^*\). Hence
\(0\leq\omega(\lambda,f)\leq1\).

Next, convolution with the total rule gives back \(f\). Indeed,
\[
m_{\mathcal N}
\bigl(f\otimes(\eta_{\mathcal N}\eta_{\mathcal M}^*)\bigr)
m_{\mathcal M}^*
=
m_{\mathcal N}(1\otimes\eta_{\mathcal N})
(f\otimes1)(1\otimes\eta_{\mathcal M}^*)m_{\mathcal M}^*
=f,
\]
using the unit and counit identities. Since
\(\lambda+\lambda^\perp=\eta_{\mathcal N}\eta_{\mathcal M}^*\), it follows
that
\[
f=
m_{\mathcal N}(f\otimes\lambda)m_{\mathcal M}^*
+
m_{\mathcal N}(f\otimes\lambda^\perp)m_{\mathcal M}^* .
\]
Thus, (i) and (ii) are equivalent. Also, (ii) implies that the losing value is
zero, and since the winning and losing values sum to \(1\), it gives
\(\omega(\lambda,f)=1\).

Conversely, assume \(\omega(\lambda,f)=1\). Then the losing value is zero, so
\( 
(\tau_{\mathcal N}\otimes\tau_{\mathcal M^{\op}})
(\Psi_{\M,\N}(f)Q_\lambda^\perp)=0.
\)
By traciality this is the trace of the positive operator
\(Q_\lambda^\perp\Psi_{\M,\N}(f)Q_\lambda^\perp\). Since the trace is faithful, we get
\(Q_\lambda^\perp\Psi_{\M,\N}(f)Q_\lambda^\perp=0\). Hence
\(\Psi_{\M,\N}(f)Q_\lambda^\perp=0\), because
\[
Q_\lambda^\perp\Psi_{\M,\N}(f)Q_\lambda^\perp
=
(\Psi_{\M,\N}(f)^{1/2}Q_\lambda^\perp)^*
(\Psi_{\M,\N}(f)^{1/2}Q_\lambda^\perp).
\]
Therefore
\[
\Psi_{\M,\N}\!\left(
m_{\mathcal N}(f\otimes\lambda^\perp)m_{\mathcal M}^*
\right)
=
\Psi_{\M,\N}(f)Q_\lambda^\perp
=
0.
\]
Since \(\Psi_{\M,\N}\) is a linear isomorphism, (ii) follows. This proves the
equivalence of (i), (ii), and (iii).
\end{proof}

We use the following elementary fact.
\begin{lemma}
\label{lem:faithful-test-state}
Let $(\mathcal M,\tau_{\mathcal M})$ and
$(\mathcal N,\tau_{\mathcal N})$ be finite-dimensional tracial von
Neumann algebras, let $\Gamma\colon\mathcal M\to\mathcal N$ be
positive, and let $P\in\mathcal P(\mathcal M)$ and
$Q\in\mathcal P(\mathcal N)$. Suppose that
$\rho_0\in\mathcal M_+$ is a normal state with
$s(\rho_0)=P$. Then the following are equivalent:
\begin{enumerate}
\item[(i)] $s(\Gamma(\rho))\leq Q$ for every normal state
$\rho\in\mathcal M_+$ satisfying $s(\rho)\leq P$;
\item[(ii)] $s(\Gamma(\rho_0))\leq Q$.
\end{enumerate}
\end{lemma}

\begin{proof}
Only (ii) $\Rightarrow$ (i) requires proof. Since
$s(\rho_0)=P$ and $\mathcal M$ is finite-dimensional, $\rho_0$ is
invertible in the corner $P\mathcal MP$. Hence
$\rho_0\geq\varepsilon P$ for some $\varepsilon>0$. Let $\rho$ be a normal state with $s(\rho)\leq P$. Then
$\rho\leq\|\rho\|P\leq\varepsilon^{-1}\|\rho\|\rho_0$.
Positivity of $\Gamma$ therefore gives
\[
0\leq
Q^\perp\Gamma(\rho)Q^\perp
\leq
\varepsilon^{-1}\|\rho\|\,
Q^\perp\Gamma(\rho_0)Q^\perp
=0.
\]
Thus $s(\Gamma(\rho))\leq Q$.
\end{proof}

Let
\(\mathcal X\subseteq \B(H_X)\),
\(\mathcal Y\subseteq \B(H_Y)\),
\(\mathcal A\subseteq \B(H_A)\), and
\(\mathcal B\subseteq \B(H_B)\) be faithful finite-dimensional representations.
Set \(\mathcal M=\mathcal X\otimes\mathcal Y\) and
\(\mathcal N=\mathcal A\otimes\mathcal B\). Then
\(\mathcal M\subseteq \B(H_{\rm in})\) and
\(\mathcal N\subseteq \B(H_{\rm out})\), where
\( 
H_{\rm in}=H_X\otimes H_Y\) and
\( H_{\rm out}=H_A\otimes H_B .
\)

Recall that for a rule operator quantum game \(\lambda\), the corresponding $\cl M^{\rm op}$-projection-test quantum game uses the accepting projection 
\(Q_\lambda=\Psi_{\M,\N}(\lambda)\in\mathcal N\otimes\mathcal M^{\op}\) represented on 
\(H_{\rm out}\otimes\overline{H_{\rm in}}\) while the corresponding input projection is
\( 
\Pi_{\mathcal M}=P_{\operatorname{vec}(\mathcal M')}
\subseteq
\mathcal M \otimes \mathcal M^{\op},
\) 
where \(\mathcal M'\) denotes the commutant of \(\mathcal M\) in
\(\B(H_{\rm in})\).

\begin{theorem} \label{thm:goldberg-projection-test}
Let \(f:L^2(\mathcal M)\to L^2(\mathcal N)\) be a channel and \(\Gamma_{f}:\mathcal M\to\mathcal N\) be the quantum channel determined by
\(f(\widehat x)=\widehat{\Gamma_{f}(x)}\). For a Goldberg rule
\(\lambda:L^2(\mathcal M)\to L^2(\mathcal N)\), set
\( 
\rho_{\mathcal M}:= \tau_{\mathcal M}(1_{\mathcal M})^{-1}\Psi_{\M}(1_{L^2(\mathcal M)})
\in \mathcal M\otimes\mathcal M^{\op},
\)
 then
\[
\omega_{G_\lambda}(\Gamma_f,\rho_{\mathcal M})
=
\left\langle
(\Gamma_{f}\otimes{\rm id}_{\mathcal M^{\op}})(\rho_{\mathcal M}),
Q_\lambda
\right\rangle
=
\,\omega(\lambda,f),
\] 
where 
\(G_\lambda\) is the corresponding projection-test quantum game.

 Moreover, \(f\) is perfect for \(\lambda\) if and only if
\(\Gamma_f\) is perfect for the projection-test \(G_\lambda\).
\end{theorem}
\begin{proof}
First observe that $\rho_{\mathcal M}$ is indeed a normal state
supported in $\Pi_{\mathcal M}$. 
Since $1_{L^2(\mathcal M)}$ is completely positive,
Proposition~4.6 gives $\Psi_{\M}(1_{L^2(\mathcal M)})\geq 0$, while Lemma~4.8
gives
\( 
 s(\Psi_{\M}(1_{L^2(\mathcal M)}))=\Pi_{\mathcal M}.
\)
Furthermore, Theorem~4.3(iii) yields
\begin{align*}
 &(\tau_{\mathcal M}\otimes
   \tau_{\mathcal M^{\mathrm{op}}})(\rho_{\mathcal M})  \\
 &\qquad =
 \tau_{\mathcal M}(1_{\mathcal M})^{-1}
 (\tau_{\mathcal M}\otimes
  \tau_{\mathcal M^{\mathrm{op}}})
 \bigl(\Psi_{\M}(1_{L^2(\mathcal M)})\bigr)                 \\
 &\qquad =
 \tau_{\mathcal M}(1_{\mathcal M})^{-1}
 \eta_{\mathcal M}^*
 1_{L^2(\mathcal M)}
 \eta_{\mathcal M}(1)                                  \\
 &\qquad =
 \tau_{\mathcal M}(1_{\mathcal M})^{-1}
 \langle\widehat{1_{\mathcal M}},
        \widehat{1_{\mathcal M}}\rangle
 =1.
\end{align*}
Thus $\rho_{\mathcal M}$ is a normal state and
$s(\rho_{\mathcal M})=\Pi_{\mathcal M}$. Then, we get
\[
\begin{aligned}
\left\langle
(\Gamma_f\otimes\id)(\rho_{\mathcal M}),Q_\lambda
\right\rangle
&=
\tau_{\mathcal M}(1)^{-1}
\left\langle \Psi_{\M,\N}(f),Q_\lambda\right\rangle  \\
&=
\tau_{\mathcal M}(1)^{-1}
(\tau_{\mathcal N}\otimes\tau_{\mathcal M^{\op}})
\bigl(\Psi_{\M,\N}(f)Q_\lambda\bigr)  \\
&=
\tau_{\mathcal M}(1)^{-1}
(\tau_{\mathcal N}\otimes\tau_{\mathcal M^{\op}})
\bigl(\Psi_{\M,\N}(m_{\mathcal N}(f\otimes\lambda)m_{\mathcal M}^*)\bigr)  \\
&=
\tau_{\mathcal M}(1)^{-1}
\eta_{\mathcal N}^*(m_{\mathcal N}(f\otimes\lambda)m_{\mathcal M}^* )\eta_{\mathcal M}(1)
\\
&=
\omega(\lambda,f).
\end{aligned}
\]
Here we used  Lemma \ref{lem:support-Psi-id}  and  Theorem \ref{thm:goldberg-daws-transform} (i),(iii).
This proves the value identity.

It remains to compare perfectness. Put
\(\lambda^\perp=\eta_{\mathcal N}\eta_{\mathcal M}^*-\lambda\). Then
\(\Psi_{\M,\N}(\lambda^\perp)=Q_\lambda^\perp\).  By Lemma \ref{lem:perfect-losing-states},
\(f\) is perfect for \(\lambda\) if and only if \(m_{\mathcal N}(f\otimes\lambda^{\perp})m_{\mathcal M}^*=0\).
Applying \(\Psi_{\M,\N}\), this is equivalent to
\( 
\Psi_{\M,\N}(f)Q_\lambda^\perp=0.
\)
As \(f\) is a channel, Proposition \ref{prop:cptop} says that \(\Psi_{\M,\N}(f)\geq0\). Thus, by traciality and faithfulness
of \(\tau_{\mathcal N}\otimes\tau_{\mathcal M^{\op}}\), we have that 
\(
\Psi_{\M,\N}(f)Q_\lambda^\perp=0\) is equivalent to 
\( 
\left\langle\Psi_{\M,\N}(f),Q_\lambda^\perp\right\rangle=0.
\)
On the other hand, by Lemma \ref{lem:support-Psi-id},
\( 
\Pi_{\mathcal M}
=
s\bigl(\Psi_{\M}(1_{L^2(\mathcal M)})\bigr).
\)
 Applying Lemma~\ref{lem:faithful-test-state} to
$\Gamma_f\otimes\operatorname{id}_{\mathcal M^{\mathrm{op}}}$,
with the normal state 
\( 
\rho_{\mathcal M}
\),
shows that $\Gamma_f$ is perfect for $G_\lambda$ if and only if
\[
\left\langle
(\Gamma_f\otimes\operatorname{id}_{\mathcal M^{\mathrm{op}}})
(\rho_{\mathcal M}),Q_\lambda^\perp
\right\rangle=0.
\]
This is equivalent
to
\[
\left\langle
(\Gamma_f\otimes\operatorname{id}_{\mathcal M^{\mathrm{op}}})
\bigl(\Psi_{\M}(1_{L^2(\mathcal M)})\bigr),
Q_\lambda^\perp
\right\rangle=0.
\]
By Lemma~\ref{lem:support-Psi-id}, the latter is exactly
\( 
\left\langle\Psi_{\M,\N}(f),Q_\lambda^\perp\right\rangle=0.
\)
Thus \(f\) is perfect for \(\lambda\) if and only if \(\Gamma_f\) is perfect
for the projection-test \(G_\lambda\).
\end{proof}

\section{Quantum versions of synchronicity} \label{sec:synchronicity}

In the classical setting, synchronicity \cite{PSSTW16} is a compatibility condition between the two
players' answers on equal questions.  A non-local game
\(G=(X,X,A,A,\lambda)\) is called \emph{synchronous} if
\(\lambda(a,b,x,x)=0\) whenever \(a\neq b\).  Thus, if both players receive
the same question, the rules force them to give the same answer in any winning
play.  Consequently, every perfect correlation \(p\) for a synchronous game is
itself synchronous, in the sense that \(p(a,b|x,x)=0\) for all \(x\in X\) and
all \(a\neq b\).

We first recall the notion of concurrency from the projection-lattice formalism
of \cite{BHTT23}.  Let \(H_X=\mathbb C^X\) and \(H_A=\mathbb C^A\), and set
\(J_X=\Omega_X\Omega_X^*\) and \(J_A=\Omega_A\Omega_A^*\), where
\(\Omega_X=|X|^{-1/2}\sum_{x\in X}e_x\otimes\overline e_x\) and
\(\Omega_A=|A|^{-1/2}\sum_{a\in A}e_a\otimes\overline e_a\).  A quantum game
\( 
\varphi:\mathcal P(M_X\otimes M_X^{\op})
\rightarrow
\mathcal P(M_A\otimes M_A^{\op})
\)
is called \emph{concurrent} if \(\varphi(J_X)=J_A\).  Thus concurrency means
that the rule sends the quantum diagonal of the input system, represented by the
maximally entangled projection \(J_X\), to the corresponding quantum diagonal of
the output system.

The same definition extends naturally to arbitrary 
finite-dimensional von Neumann algebras.  Let
\(\mathcal X\subseteq \B(H_X)\) and \(\mathcal A\subseteq \B(H_A)\), and write
\(\Pi_{\mathcal X}=P_{\operatorname{vec}(\mathcal X')}\) and
\(\Pi_{\mathcal A}=P_{\operatorname{vec}(\mathcal A')}\), where the vectorisation
is taken in \(H_X\otimes\overline{H_X}\) and \(H_A\otimes\overline{H_A}\),
respectively.  We call a projection-lattice rule
\[
\varphi:\mathcal P(\mathcal X\otimes\mathcal X^{\op})
\rightarrow
\mathcal P(\mathcal A\otimes\mathcal A^{\op})
\]
\emph{concurrent} if \(\varphi(\Pi_{\mathcal X})=\Pi_{\mathcal A}\).  When
\(\mathcal X=M_X\) and \(\mathcal A=M_A\) are represented canonically, one has
\(\Pi_{\mathcal X}=J_X\) and \(\Pi_{\mathcal A}=J_A\), so this recovers the
definition above.

Assume that $\mathcal X$ and $ \mathcal A$ are equipped with faithful traces  and let
\( 
\Gamma: \mathcal X \otimes \mathcal X^{\op}
\rightarrow
\mathcal A \otimes \mathcal A^{\op}
\)
be a quantum channel.  We call  $\Gamma$ \emph{concurrent} if it is perfect for the concurrent part of the rule, that is, 
\(\langle\Gamma(\rho),\Pi_{\mathcal A}^{\perp}\rangle=0\) for every normal state
\(\rho\) with \(s(\rho)\leq\Pi_{\mathcal X}\). Equivalently,  \( 
s(\Gamma(\rho))\leq\Pi_{\mathcal A}
\)
for every normal state $\rho$ satisfying
$s(\rho)\leq\Pi_{\mathcal X}$. 
\begin{proposition} \label{prop:concurrency-equivalences}
Let $\Gamma\colon\mathcal X\otimes\mathcal X^{\mathrm{op}}
\to\mathcal A\otimes\mathcal A^{\mathrm{op}}$ be a quantum channel,
and let $\Gamma(\rho)=\sum_iT_i\rho T_i^*$ be a Kraus decomposition.
The following are equivalent:
\begin{enumerate}
\item[(i)] $\Gamma$ is concurrent;
\item[(ii)] $s(\Gamma(\Pi_{\mathcal X}))\leq\Pi_{\mathcal A}$;
\item[(iii)] $\Pi_{\mathcal A}^{\perp}T_i\Pi_{\mathcal X}=0$
for every $i$.
\end{enumerate}
\end{proposition}

\begin{proof}
Let
$d=(\tau_{\mathcal X}\otimes\tau_{\mathcal X^{\mathrm{op}}})
(\Pi_{\mathcal X})$ and set
$\rho_0=d^{-1}\Pi_{\mathcal X}$. Then
$\rho_0$ is a normal state with
$s(\rho_0)=\Pi_{\mathcal X}$. By
Lemma~\ref{lem:faithful-test-state}, $\Gamma$ is concurrent if and
only if
$s(\Gamma(\rho_0))\leq\Pi_{\mathcal A}$. Since
$\Gamma(\rho_0)=d^{-1}\Gamma(\Pi_{\mathcal X})$, this is equivalent
to
$s(\Gamma(\Pi_{\mathcal X}))\leq\Pi_{\mathcal A}$. Thus
(i) and (ii) are equivalent.

Finally,
\[
\Pi_{\mathcal A}^{\perp}\Gamma(\Pi_{\mathcal X})
\Pi_{\mathcal A}^{\perp}
=
\sum_i
(\Pi_{\mathcal A}^{\perp}T_i\Pi_{\mathcal X})
(\Pi_{\mathcal A}^{\perp}T_i\Pi_{\mathcal X})^*.
\]
Since $\Gamma(\Pi_{\mathcal X})$ is positive, (ii) holds if
and only if the left-hand side vanishes. Since every summand on the
right is positive, this is equivalent to
$\Pi_{\mathcal A}^{\perp}T_i\Pi_{\mathcal X}=0$ for every $i$.
Thus (ii) and (iii) are equivalent.
\end{proof}

We next compare concurrency with Goldberg's notion of synchronicity. A linear map
\(f:L^2(\mathcal X\otimes\mathcal X^{\op})\to
L^2(\mathcal A\otimes\mathcal A^{\op})\) is called \emph{synchronous}
\cite{Goldberg26} if
\[
m_{\mathcal A}^*m_{\mathcal A}fm_{\mathcal X}^*m_{\mathcal X}
=
fm_{\mathcal X}^*m_{\mathcal X}.
\]

\begin{proposition}
\label{prop:Goldberg-concurrency}
Let
\( 
f\colon
L^2(\mathcal X)\otimes L^{2}(\mathcal X^{\mathrm{op}})
\longrightarrow
L^2(\mathcal A) \otimes L^2(\mathcal A^{\mathrm{op}})
\)
be a channel, and let $\Gamma_f$ be its associated quantum channel.
If $f$ is synchronous, then $\Gamma_f$ is concurrent.

If moreover $\mathcal X$ and $\mathcal A$ are equipped with the special
Frobenius normalisation, the converse also holds.
\end{proposition}

\begin{proof}
Put
$\mathcal M=\mathcal X\otimes\mathcal X^{\mathrm{op}}$ and
$\mathcal N=\mathcal A\otimes\mathcal A^{\mathrm{op}}$, and set
$h_{\mathcal X}=\Psi_{\mathcal X}
(1_{L^2(\mathcal X)})$, with $h_{\mathcal A}$ defined similarly.
The operator
$m_{\mathcal X}^*m_{\mathcal X}$ is right multiplication by
$h_{\mathcal X}$. Indeed, it suffices to check this blockwise. If
$\mathcal X=M_n$ and $\tau_{\mathcal X}=c\operatorname{Tr}$, then
\( 
h_{\mathcal X}
=
c^{-1}\sum_{i,j=1}^n
\varepsilon_{i,j}\otimes\varepsilon_{j,i}^{\mathrm{op}},
\)
and
\[
m_{\mathcal X}^*m_{\mathcal X}
\bigl(
\widehat{\varepsilon_{a,b}}\otimes
\widehat{\varepsilon_{c,d}^{\mathrm{op}}}
\bigr)
=
\delta_{b,c}c^{-1}\sum_{j=1}^n
\widehat{\varepsilon_{a,j}}\otimes
\widehat{\varepsilon_{j,d}^{\mathrm{op}}}.
\]
On the other hand,
\begin{align*}
(\varepsilon_{a,b}\otimes\varepsilon_{c,d}^{\mathrm{op}})
h_{\mathcal X}
&=
c^{-1}\sum_{i,j}
\varepsilon_{a,b}\varepsilon_{i,j}\otimes
\varepsilon_{c,d}^{\mathrm{op}}
\varepsilon_{j,i}^{\mathrm{op}}\\
&=
\delta_{b,c}c^{-1}\sum_j
\varepsilon_{a,j}\otimes\varepsilon_{j,d}^{\mathrm{op}}.
\end{align*}
Thus $m_{\mathcal X}^*m_{\mathcal X}(z)=zh_{\mathcal X}$ for
$z\in\mathcal M$, and similarly
$m_{\mathcal A}^*m_{\mathcal A}(w)=wh_{\mathcal A}$ for
$w\in\mathcal N$. 

Now let $z\in\mathcal M$. We have
\begin{align*}
 &m_{\mathcal A}^*m_{\mathcal A}
 f m_{\mathcal X}^*m_{\mathcal X}(\widehat z)\\
 &\qquad =
 m_{\mathcal A}^*m_{\mathcal A}
 f(\widehat{zh_{\mathcal X}})\\
 &\qquad =
 m_{\mathcal A}^*m_{\mathcal A}
 \bigl(\widehat{\Gamma_f(zh_{\mathcal X})}\bigr)\\
 &\qquad =
 \widehat{\Gamma_f(zh_{\mathcal X})h_{\mathcal A}},
\end{align*}
whereas
\( 
 f m_{\mathcal X}^*m_{\mathcal X}(\widehat z)
 =
 f(\widehat{zh_{\mathcal X}})
 =
 \widehat{\Gamma_f(zh_{\mathcal X})}.
\)
Consequently, the  synchronicity condition of $ f$
is equivalent to
\[
 \Gamma_f(zh_{\mathcal X})h_{\mathcal A}
 =
 \Gamma_f(zh_{\mathcal X}),
 \qquad
 z\in\mathcal X\otimes\mathcal X^{\mathrm{op}}.
\]

By Lemma~\ref{lem:support-Psi-id},
$s(h_{\mathcal X})=\Pi_{\mathcal X}$ and
$s(h_{\mathcal A})=\Pi_{\mathcal A}$. 
Since $s(h_{\mathcal X})=\Pi_{\mathcal X}$ and $ \cl M$ is finite dimensional,
$h_{\mathcal X}$ is invertible in the corner
$\Pi_{\mathcal X}\mathcal M\Pi_{\mathcal X}$ and one has
$\mathcal Mh_{\mathcal X}=\mathcal M\Pi_{\mathcal X}$.
Multiplying
$\Gamma_f(zh_{\mathcal X})h_{\mathcal A}
=\Gamma_f(zh_{\mathcal X})$
on the right by $\Pi_{\mathcal A}^{\perp}$ gives
$\Gamma_f(zh_{\mathcal X})\Pi_{\mathcal A}^{\perp}=0$, since
$h_{\mathcal A}\Pi_{\mathcal A}^{\perp}=0$. Hence
$\Gamma_f(y)\Pi_{\mathcal A}^{\perp}=0$ for every
$y\in\mathcal M\Pi_{\mathcal X}$. In particular, if
$s(\rho)\leq\Pi_{\mathcal X}$, then
$s(\Gamma_f(\rho))\leq\Pi_{\mathcal A}$. Thus $\Gamma_f$ is
concurrent.

Under the special Frobenius normalisation,
$h_{\mathcal X}=\Pi_{\mathcal X}$ and
$h_{\mathcal A}=\Pi_{\mathcal A}$. Conversely, suppose that
$\Gamma_f$ is concurrent. By
Proposition~\ref{prop:concurrency-equivalences},
\( 
\Pi_{\mathcal A}^{\perp}T_i\Pi_{\mathcal X}=0\) for every \(i\).
Taking adjoints gives
$\Pi_{\mathcal X}T_i^*\Pi_{\mathcal A}^{\perp}=0$. Hence, for
every $z\in\mathcal M$,
\[
\Gamma_f(z\Pi_{\mathcal X})\Pi_{\mathcal A}^{\perp}
=
\sum_iT_i z\Pi_{\mathcal X}T_i^*
\Pi_{\mathcal A}^{\perp}
=0.
\]
Therefore
$\Gamma_f(z\Pi_{\mathcal X})\Pi_{\mathcal A}
=\Gamma_f(z\Pi_{\mathcal X})$ for every $z\in\mathcal M$.
Since $h_{\mathcal X}=\Pi_{\mathcal X}$ and
$h_{\mathcal A}=\Pi_{\mathcal A}$, this is precisely Goldberg
synchronicity.
\end{proof}

Finally, we compare the preceding notions with the synchronicity condition of
Bochniak--Kasprzak--So{\l}tan \cite{BKS23}.  Let \((\mathcal X,\tau_{\cl X})\) and
\((\mathcal A, \tau_{\cl A})\) again be tracial von Neumann algebras. 
Write first their abstract decompositions as
\[
\mathcal X \cong\bigoplus_{\ell=1}^{r_{\mathcal X}}M_{m_\ell}, 
\qquad
\mathcal A\cong\bigoplus_{k=1}^{r_{\mathcal A}}M_{n_k}.
\]
In this abstract form the BKS diagonal projections are, up to the fixed transpose
convention on the opposite leg,
\[
\Pi_{\mathcal X}^{\rm BKS}
=
\sum_{\ell=1}^{r_{\mathcal X}}\frac1{m_\ell}
\sum_{s,t=1}^{m_\ell} f^\ell_{s,t}\otimes(f^\ell_{t,s})^{\op},
\qquad
\Pi_{\mathcal A}^{\rm BKS}
=
\sum_{k=1}^{r_{\mathcal A}}\frac1{n_k}
\sum_{i,j=1}^{n_k} e^k_{i,j}\otimes(e^k_{j,i})^{\op}.
\]
Equivalently, in the multiplicity-free realisation of the \(\ell\)-th summand,
if
\(\Omega_\ell^{\mathcal X}=m_\ell^{-1/2}\sum_s e_s^\ell\otimes
\overline{e_s^\ell}\), then
\(\Pi_{\mathcal X}^{\rm BKS}=\sum_\ell \Omega_\ell^{\mathcal X}(\Omega_\ell^{\mathcal X})^*\), and
similarly for \(\mathcal A\).

Let \(\omega_{\mathcal X}\) be the canonical diagonal state on
\(\mathcal X\otimes\mathcal X^{\op}\), given by
\[
\omega_{\mathcal X}(z)
=
\frac1{r_{\mathcal X}}
\sum_{\ell=1}^{r_{\mathcal X}}
\langle\Omega_\ell^{\mathcal X},z\Omega_\ell^{\mathcal X}\rangle .
\]
 If
\(\Gamma:\mathcal X\otimes\mathcal X^{\op}
\to \mathcal A\otimes\mathcal A^{\op}\) is a quantum channel,
with unital completely positive adjoint \(\Gamma^*\), then $\Gamma$ is called \emph{BKS-synchronous}~\cite[Proposition 6.3]{BKS23} if
\[
\omega_{\mathcal X}
\bigl(\Gamma^*(\Pi_{\mathcal A}^{\rm BKS})\bigr)=1.
\]

Now choose concrete faithful representations with multiplicities. Up to unitary equivalence, we may write
\[
H_{\mathcal X}=
\bigoplus_{\ell=1}^{r_{\mathcal X}}\mathbb C^{m_\ell}\otimes\mathbb C^{d_\ell},
\qquad
\mathcal X=
\bigoplus_{\ell=1}^{r_{\mathcal X}}M_{m_\ell}\otimes 1_{d_\ell},
\]
and similarly
\(H_{\mathcal A} =
\bigoplus_k\mathbb C^{n_k}\otimes\mathbb C^{c_k}\) and
\(\mathcal A=\bigoplus_k M_{n_k}\otimes 1_{c_k}\). Then
\(\mathcal X'=\bigoplus_\ell 1_{m_\ell}\otimes M_{d_\ell}\) and
\(\mathcal A'=\bigoplus_k 1_{n_k}\otimes M_{c_k}\). Under the 
action on \(H_{\mathcal X}\otimes\overline{H_{\mathcal X}}\), the abstract BKS
projection \(\Pi_{\mathcal X}^{\mathrm{BKS}}\) is represented by
\[
\sum_{\ell=1}^{r_{\mathcal X}}
\Omega_{m_\ell}\Omega_{m_\ell}^*\otimes
1_{\mathbb C^{d_\ell}\otimes\overline{\mathbb C^{d_\ell}}}
=
P_{\operatorname{vec}(\mathcal X')}= \Pi_{\cl X},
\]
and analogously \(\Pi_{\mathcal A}^{\mathrm{BKS}}\) is represented by 
\(\Pi_{\cl A}\).

\begin{proposition} \label{prop:BKS-concurrency}
Let
\(
\mathcal X\otimes\mathcal X^{\mathrm{op}}
\subseteq
\mathcal B(H_{\mathcal X}\otimes\overline{H_{\mathcal X}})
\)
and
\(
\mathcal A\otimes\mathcal A^{\mathrm{op}}
\subseteq
\mathcal B(H_{\mathcal A}\otimes\overline{H_{\mathcal A}})
\)
be the concrete representations described above. Let
\( 
\Gamma\colon
\mathcal X\otimes\mathcal X^{\mathrm{op}}
\rightarrow
\mathcal A\otimes\mathcal A^{\mathrm{op}}
\)
be a quantum channel. Then the following are equivalent:
\begin{enumerate}
\item[(i)] $\Gamma$ is BKS-synchronous;
\item[(ii)] $\Gamma$ is concurrent.
\end{enumerate}
\end{proposition}

\begin{proof}
Let $\tilde\rho_{\omega_{\mathcal X}}$ be the density of the
transported BKS diagonal state. By the preceding identification,
$s(\tilde\rho_{\omega_{\mathcal X}})=\Pi_{\mathcal X}$. By trace duality, the BKS condition is equivalent to
$\langle\Gamma(\tilde\rho_{\omega_{\mathcal X}}),
\Pi_{\mathcal A}\rangle=1$. Since
$\Gamma(\tilde\rho_{\omega_{\mathcal X}})$ is a state, this is
equivalent to
$s(\Gamma(\tilde\rho_{\omega_{\mathcal X}}))
\leq\Pi_{\mathcal A}$. By
Lemma~\ref{lem:faithful-test-state} and Proposition \ref{prop:concurrency-equivalences}, this is equivalent to
concurrency.
\end{proof}


\section{Examples} \label{sec:examples}

Before we move on to examples, we single out some classes of quantum non-local games that fall under the projection-lattice setting and one example that does not. Recall that throughout, $H_{\rm in}, H_{\rm out}$ are assumed finite dimensional while $ H_{R}$ not necessarily.

\begin{proposition} \label{prop:rank-ones-refl}
Let \(G=(\psi,\gamma\gamma^*)\) be a rank-one projection-test game, with states
\(\psi\in H_{\rm in}\otimes H_R\) and
\(\gamma\in H_{\rm out}\otimes H_R\). Then its winning transformation space
\[
\mathcal U_{\psi,\gamma}
=
\{T\in \B(H_{\rm in},H_{\rm out}):
(\gamma\gamma^*)^\perp(T\otimes 1_{H_R})\psi=0\}
\]
is reflexive. Consequently, \(G\) determines a projection-lattice game whose
winning transformation space is \(\mathcal U_{\psi,\gamma}\).
\end{proposition}

\begin{proof}
Identify $\psi$ and $\gamma$ with Hilbert--Schmidt operators
$A_\psi\colon\overline{H_R}\to H_{\rm in}$ and
$A_\gamma\colon\overline{H_R}\to H_{\rm out}$, respectively. Then
$(T\otimes1_{H_R})\psi=\operatorname{vec}(TA_\psi)$, and hence
$T\in\mathcal U_{\psi,\gamma}$ if and only if
$TA_\psi\in\mathbb C A_\gamma$. Put $L=\operatorname{ran}A_\psi$. 

Suppose first that $A_\gamma$ does not vanish on $\ker A_\psi$. If
$TA_\psi=\alpha A_\gamma$, evaluation on $\ker A_\psi$ forces
$\alpha=0$. Therefore
\( 
\mathcal U_{\psi,\gamma}
=\{T\in\mathcal B(H_{\rm in},H_{\rm out}):T|_L=0\},
\)
which is reflexive.
Now suppose that $A_\gamma$ vanishes on $\ker A_\psi$. There is then a
well-defined operator $C\colon L\to H_{\rm out}$ satisfying
$CA_\psi=A_\gamma$, and
\( 
\mathcal U_{\psi,\gamma}
=\{T\in\mathcal B(H_{\rm in},H_{\rm out}):T|_L\in\mathbb C C\}.
\)
We use the elementary fact that the one-dimensional operator space
$\mathbb C C$ is reflexive. Indeed, suppose that
$B\colon L\to H_{\rm out}$ satisfies $Bx\in\mathbb C Cx$ for every
$x\in L$. Then $\ker C\subseteq\ker B$ and
$\operatorname{ran}B\subseteq\operatorname{ran}C$. Thus $B$ and $C$
induce maps
\( 
\widetilde B\colon L/\ker C\to\operatorname{ran}C\) and 
\( 
\widetilde C\colon L/\ker C\to\operatorname{ran}C,
\)
with $\widetilde C$ invertible. The operator
$\widetilde C^{-1}\widetilde B$ leaves every one-dimensional subspace of
$L/\ker C$ invariant and is therefore a scalar multiple of the
identity. Hence $B\in\mathbb C C$.

Finally, let $S\in\operatorname{Ref}(\mathcal U_{\psi,\gamma})$. For
every $x\in L$, we have
$Sx\in[\mathcal U_{\psi,\gamma}x]=\mathbb C Cx$. By the preceding
observation, $S|_L\in\mathbb C C$, and therefore
$S\in\mathcal U_{\psi,\gamma}$. Thus
$\mathcal U_{\psi,\gamma}$ is reflexive. The final assertion follows
from the correspondence between reflexive winning transformation spaces
and projection-lattice games.
\end{proof}

\begin{proposition}\label{prop:one-sided-masa-modules-reflexive}
Let \(H_{\rm in}\) and \(H_{\rm out}\) be finite-dimensional Hilbert spaces, and
let \(\mathcal U\subseteq \B(H_{\rm in},H_{\rm out})\) be an operator space. If
\(\mathcal U\) is a left masa module (resp. right masa module), then
\(\mathcal U\) is reflexive.
\end{proposition}

\begin{proof}
We prove the left module case.  Choose orthonormal bases and identify
\(H_{\rm in}=\mathbb C^n\), \(H_{\rm out}=\mathbb C^m\).  Let
\(\mathcal D_m\subseteq M_m\) be the diagonal masa and assume
\(\mathcal D_m\mathcal U\subseteq\mathcal U\).  If \(\varepsilon_{i,i}\) denotes the
\(i\)-th diagonal matrix unit, then \(\varepsilon_{i,i}\mathcal U\subseteq\mathcal U\),
and hence every \(T\in\mathcal U\) decomposes row by row as
\(T=\sum_i \varepsilon_{i,i}T\).  Thus
\[
\mathcal U=\bigoplus_{i=1}^m \varepsilon_{i,i}\mathcal U .
\]

For each \(i\), set
\( 
F_i=\{e_i^*T:T\in\mathcal U\}\subseteq(\mathbb C^n)^* .
\)
Equivalently, \(F_i\) is the space of possible \(i\)-th rows of operators in
\(\mathcal U\).  Since \(\varepsilon_{i,i}\mathcal U\subseteq\mathcal U\), we have
\(\varepsilon_{i,i}\mathcal U=e_iF_i\), where \(e_if:x\mapsto f(x)e_i\), $ x \in \bb C^n$.  Therefore
\( 
\mathcal U=\bigoplus_{i=1}^m e_iF_i .
\)

We now compute its preannihilator under the trace pairing between
\(\B(\mathbb C^n,\mathbb C^m)\) and \(\B(\mathbb C^m,\mathbb C^n)\).  If
\(S\in \B(\mathbb C^m,\mathbb C^n)\), then for \(T=\sum_i e_if_i\in\mathcal U\),
\( 
\operatorname{Tr}(TS)=\sum_{i=1}^m f_i(Se_i).
\)
Hence \(S\in\mathcal U_\perp\) if and only if \(Se_i\in F_i^\perp\) for every
\(i\).  Thus
\( 
\mathcal U_\perp=\bigoplus_{i=1}^m F_i^\perp e_i^* .
\)
In particular, every \(S\in\mathcal U_\perp\) decomposes as
\( 
S=\sum_{i=1}^m (Se_i)e_i^*,
\)
and each summand belongs to \(\mathcal U_\perp\) and has rank at most one.
Thus \(\mathcal U_\perp\) is spanned by its rank-one elements.  By Larson's 
rank-one preannihilator criterion \cite{Lar82} (see also
\cite[Corollary 9.3]{Erdos86}), \(\mathcal U\) is reflexive.

The right module case is obtained similarly by decomposing column by column.
\end{proof}

\begin{corollary} \label{cor:q2c_c2q}
Every quantum-to-classical (resp. classical-to-quantum) projection-test quantum game  determines a projection-lattice
game with the same winning transformation space.
\end{corollary}

We now exhibit an example of a projection-test quantum game whose winning transformation space is not reflexive.

\begin{example} \label{ex:not_lattice} \rm
Let \(H=\mathbb C^2\), put \(H_{\rm in}=H_{\rm out}=H\), and let
\(H_R=\overline H\).  Fix an orthonormal basis \(e_1,e_2\) of \(H\), and set
\( 
\Omega
=
\frac{1}{\sqrt 2}
\bigl(e_1\otimes \overline e_1+e_2\otimes \overline e_2\bigr)
\in H\otimes \overline H .
\)
Consider the projection-test game \(G=(\Omega,Q)\), where
\( 
Q
=
1_{H\otimes \overline H}-\Omega\Omega^* .
\)
Then
\[
\mathcal U_{\Omega,Q}
=
\{T\in \B(H):Q^\perp(T\otimes 1_{\overline H})\Omega=0\}
=
\{T\in M_2:\operatorname{Tr}(T)=0\}.
\]
Indeed, \(Q^\perp=\Omega\Omega^*\), and
\( 
\left\langle \Omega,(T\otimes 1_{\overline H})\Omega\right\rangle
=
\frac{1}{2}\operatorname{Tr}(T).
\)
Thus the winning condition is exactly \(\operatorname{Tr}(T)=0\).

We claim that \(\mathcal U_{\Omega,Q}\) is not reflexive.  Put $\mathcal U=\mathcal U_{\Omega, Q}$. 
By Larson's criterion \cite{Lar82}, \(\mathcal U\) is reflexive only if
\(\mathcal U_\perp\) is spanned by its rank-one elements.
In the present example,
\(
\mathcal U=\{T\in M_2:\operatorname{Tr}(T)=0\}
\) 
and its preannihilator is \(\mathcal U_\perp=\mathbb C I_2\).  Since \(I_2\) has rank two,
\(\mathcal U_\perp\) contains no non-zero rank-one operators. Hence
\(\mathcal U_\perp\) is not spanned by rank-one elements and therefore
\(\mathcal U\) is not reflexive.
\end{example}

\subsection{Coherent state exchange game}
We now describe perhaps the first example of a quantum input/output non-local game; namely, the coherent state exchange game of 
Leung--Toner--Watrous  \cite{LTW13} and explain how it fits the settings examined here.  In our notation, the input
spaces are \(H_X=H_Y=\mathbb C^3\), the output spaces are
\(H_A=H_B=\mathbb C^2\), and the referee space is \(H_R=\mathbb C^2\). Thus
\(H_{\rm in}=H_X\otimes H_Y\) and
\(H_{\rm out}=H_A\otimes H_B\).  Put
\( 
\phi=\frac{1}{\sqrt2}(e_1^X\otimes e_1^Y+e_2^X\otimes e_2^Y)
\)
and define the input vector, in  \(H_{\rm in}\otimes H_R\), by
\[
\psi_{\rm coh}
=
\frac{1}{\sqrt2}(e_0^X\otimes e_0^Y)\otimes e_0^R
+
\frac{1}{\sqrt2}\phi\otimes e_1^R .
\]
The referee sends the \(H_X\)-register to Alice and the \(H_Y\)-register to
Bob. The players return the qubit registers \(H_A\) and \(H_B\). The accepting
vector in  \(H_{\rm out}\otimes H_R\), is
\[
\gamma
=
\frac{1}{\sqrt2}(e_0^A\otimes e_0^B)\otimes e_0^R
+
\frac{1}{\sqrt2}(e_1^A\otimes e_1^B)\otimes e_1^R .
\]
The accepting projection is \(Q_{\gamma}=\gamma\gamma^*\in \B(H_{\rm out}\otimes H_R)\).
Hence, the coherent state exchange game of LTW is precisely the rank-one quantum game \(G_{\rm LTW}=(\psi_{\rm coh},Q_{\gamma})\).

Operationally, the referee gives Alice and Bob a single quantum input which
is a superposition of two components: one in which the players hold the
product state \(e_0^X\otimes e_0^Y\), and one in which they hold the
entangled state \(\phi\).  To win, their operation should send the first
component to \(e_0^A\otimes e_0^B\) and the second component to
\(e_1^A\otimes e_1^B\).  Their transformations must be done coherently, that is, the players cannot measure which component they received,
nor leave any private record of it.  The referee's reference register
\(H_R\), together with the final rank-one test against \(\gamma\), detects
precisely whether this coherence has been preserved. The winning transformation space is, therefore,
\[
\mathcal U_{\rm LTW}
=
\{T\in \B(H_{\rm in},H_{\rm out}):
Q_\gamma^\perp(T\otimes 1_{H_R})\psi_{\rm coh}=0\}.
\]

Since \(Q_\gamma\) is the projection onto \(\mathbb C\gamma\), the winning
condition is equivalent to
\((T\otimes 1_{H_R})\psi_{\rm coh}\in\mathbb C\gamma\).  Hence, there is a
scalar \(\alpha\in\mathbb C\) such that
\[
T: e_0^X\otimes e_0^Y \mapsto \alpha (e_0^A\otimes e_0^B),
\qquad
\phi \mapsto \alpha (e_1^A\otimes e_1^B) .
\]
Let
\( 
L=\operatorname{span}\{e_0^X\otimes e_0^Y,\phi\}\),
\(K=\operatorname{span}\{e_0^A\otimes e_0^B,e_1^A\otimes e_1^B\},
\)
and let \(C:L\to K\) be the linear isomorphism determined by
\(C:e_0^X\otimes e_0^Y \mapsto e_0^A\otimes e_0^B\) and \(C: \phi \mapsto e_1^A\otimes e_1^B\).  Then
\[
\mathcal U_{\rm LTW}
=
\{T\in \B(H_{\rm in},H_{\rm out}):T|_L\in\mathbb C C\}.
\]

 By Proposition \ref{prop:rank-ones-refl}, since the game is rank-one, its operator space is reflexive and it determines a projection-lattice rule; 
\[
\varphi_{\rm LTW}:\mathcal P(\B(H_{\rm in}))
\to
\mathcal P(\B(H_{\rm out}))
\]
given by
\[
\varphi_{\rm LTW}(P)
=
\begin{cases}
P_{C(PH_{\rm in})}, & PH_{\rm in}\subseteq L,\\
1_{H_{\rm out}}, & PH_{\rm in}\nsubseteq L,
\end{cases}
\]
where \(P_{C(PH_{\rm in})}\) denotes the projection onto the subspace
\(C(PH_{\rm in})\subseteq K\).  This rule says that, on input supports lying
inside the subspace \(L\), the output support is forced to be the
corresponding image under \(C\); once the input support has a component
outside \(L\), no restriction remains.

With this rule, one recovers the same space.  Indeed, if \(T|_L=\alpha C\),
then \(T(PH_{\rm in})\subseteq C(PH_{\rm in})\) whenever
\(PH_{\rm in}\subseteq L\), while the rule is automatic otherwise.  Conversely,
if \(T\) satisfies the lattice condition, then applying it to the rank-one
projection \(P_x\), for \(0\neq x\in L\), gives
\(Tx\in\mathbb C Cx\).  Hence \(C^{-1}T|_L\) leaves every line in \(L\)
invariant and is therefore a scalar multiple of the identity.  Thus
\(T|_L\in\mathbb C C\), and consequently
\( 
\mathcal U_{\varphi_{\rm LTW}}
=
\mathcal U_{\rm LTW}.
\)

\subsection{Quantum XOR games}
XOR games were introduced in quantum information theory by Cleve, Høyer,
Toner and Watrous \cite{CHTW04}, although they already appear implicitly in
the earlier works of Bell \cite{Bell64} and Clauser--Horne--Shimony--Holt
\cite{CHSH69}. A classical XOR game is a two-player one-round game in which
the referee chooses a question pair \((x,y)\in X\times Y\) according to a
distribution \(\pi\), sends \(x\) to Alice and \(y\) to Bob, and receives
bits \(a,b\in\{0,1\}\). The winning condition depends only on the parity
\(a\oplus b\): for each \((x,y)\) there is a prescribed value
\(c(x,y)\in\{0,1\}\), and the players win precisely when
\(a\oplus b=c(x,y)\).

Quantum XOR games, introduced by Regev and Vidick \cite{RV15}, are
two-player one-round games with quantum inputs and binary classical outputs.
The referee chooses a label \(i\in I\) with probability \(\pi(i)\), prepares
a bipartite quantum state on \(H_X\otimes H_Y\), sends the two registers to
Alice and Bob, and asks them to return bits \(a,b\in\{0,1\}\). The winning
condition depends only on the parity of the answers: each label \(i\) comes
with a prescribed value \(c(i)\in\{0,1\}\), and the players win precisely
when \(a\oplus b=c(i)\). 

They fit into the projection-test formalism in direct analogy with classical
games. Assume, for simplicity, that the quantum questions are pure states
\(\xi_i\in H_X\otimes H_Y\). Set
\(H_{\rm in}=H_X\otimes H_Y\), \(H_{\rm out}=\mathbb C^2\otimes\mathbb C^2\),
and \(H_R=\mathbb C^I\), where \(H_R\) stores a coherent copy of the
classical label. The input vector is
\[
\psi_\pi=\sum_{i\in I}\sqrt{\pi(i)}\,\xi_i\otimes e_i
\in H_{\rm in}\otimes H_R .
\]
Let \(P_0,P_1\in \B(H_{\rm out})\) be the parity projections
\[
P_0=\varepsilon_{00,00}+\varepsilon_{11,11},
\qquad
P_1=\varepsilon_{01,01}+\varepsilon_{10,10},
\]
where $ \varepsilon_{ab,a'b'}= \varepsilon_{a,a'} \otimes \varepsilon_{b,b'}$.
Thus \(P_0\) projects onto the even-parity output subspace and \(P_1\) onto
the odd-parity output subspace. The accepting projection is
\[
Q_{\rm XOR}
=
\sum_{i\in I} P_{c(i)}\otimes \varepsilon_{i,i}
\in (D_2 \otimes D_2)\otimes \B(H_R).
\]
Equivalently, in the \(i\)-th referee block, the accepted output subspace is
the even-parity subspace if \(c(i)=0\), and the odd-parity subspace if
\(c(i)=1\).
Hence, a quantum XOR game is the quantum game
\(G_{\rm qXOR}=(\psi_\pi,Q_{\rm XOR})\).

Since the outputs are classical bits, the accepting projection is diagonal
with respect to the computational basis of \(H_{\rm out}\). Hence
\(\mathcal U_{\rm qXOR}\) is a left module over the output diagonal masa, and
is therefore reflexive by Proposition~\ref{prop:one-sided-masa-modules-reflexive}. We record the corresponding projection-lattice rule.
Set
\(L_\varepsilon=\operatorname{span}\{\xi_i:c(i)=\varepsilon\}\), for
\(\varepsilon=0,1\).  Assuming that the chosen probability distribution has full support, the corresponding winning transformation space is
\[
\mathcal U_{\rm qXOR}
=
\{T\in \B(H_{\rm in},H_{\rm out}):
T(L_0)\subseteq P_0H_{\rm out},\;
T(L_1)\subseteq P_1H_{\rm out}\}.
\]
Equivalently, \(P_1TP_{L_0}=0\) and \(P_0TP_{L_1}=0\).  Thus, if
\(L_0\cap L_1\neq0\), every winning transformation vanishes on
\(L_0\cap L_1\).

The associated projection-lattice map
\(\varphi_{\rm qXOR}:\mathcal P(\B(H_{\rm in}))\to
\mathcal P(\B(H_{\rm out}))\) is given by
\[
\varphi_{\rm qXOR}(P)
=
\begin{cases}
0, & PH_{\rm in}\subseteq L_0\cap L_1,\\
P_0, & PH_{\rm in}\subseteq L_0 \text{ and } PH_{\rm in}\nsubseteq L_1,\\
P_1, & PH_{\rm in}\subseteq L_1 \text{ and } PH_{\rm in}\nsubseteq L_0,\\
1_{H_{\rm out}}, & PH_{\rm in}\nsubseteq L_0 \text{ and } PH_{\rm in}\nsubseteq L_1 .
\end{cases}
\]
This rule assigns to an input support the smallest parity output support
compatible with the constraints \(L_0\to P_0H_{\rm out}\) and
\(L_1\to P_1H_{\rm out}\).  Hence, if
\(T\in\mathcal U_{\rm qXOR}\), then
\(\varphi_{\rm qXOR}(P)^\perp TP=0\) for every projection \(P\).  Conversely,
applying the lattice condition to rank-one projections \(P_x\), with
\(x\in L_0\) or \(x\in L_1\), forces, respectively
\(Tx\in P_0H_{\rm out}\) and \(Tx\in P_1H_{\rm out}\).  Therefore,
\( 
\mathcal U_{\varphi_{\rm qXOR}}
=
\mathcal U_{\rm qXOR}.
\)
In particular, the quantum XOR projection-test determines a 
projection-lattice game.

\subsection{Classical-to-quantum graph colouring games}

Let \(G\) be a finite simple graph with vertex set \(X\), and write
\(x\sim y\) when \(\{x,y\}\in E(G)\).  The classical \(A\)-colouring game is
the synchronous non-local game in which the referee sends vertices
\(x,y\in X\), the players return colours \(a,b\in A\), and the winning
conditions are \(x=y\Rightarrow a=b\) and \(x\sim y\Rightarrow a\neq b\).

Following Todorov--Turowska \cite{TT24} and
Brannan--Harris--Todorov--Turowska \cite{BHTT23}, the
classical-to-quantum version keeps the input classical but replaces the
output diagonal algebra by the full matrix algebra.  Let
\(\Omega_A=|A|^{-1/2}\sum_{a\in A}e_a\otimes \overline e_a\) and
\(J_A=\Omega_A\Omega_A^*\in M_A\otimes M_A^{\op}\).  The \emph{classical-to-quantum
\(A\)-colouring game} of \(G\) is the zero-preserving join-continuous map
\[
\varphi_G^A:\mathcal P(D_X\otimes D_X^{\op})\to\mathcal P(M_A\otimes M_A^{\op})
\]
determined on atoms by
\[
\varphi_G^A(\varepsilon_{x,x}\otimes (\varepsilon_{y,y})^{\op})
=
\begin{cases}
J_A, & x=y,\\
J_A^\perp, & x\sim y,\\
1_{A}\otimes 1_A^{\op}, & \text{otherwise}.
\end{cases}
\]
Thus, the quantisation consists in replacing the classical equality
projection \(J_A^{\rm cl}=\sum_{a\in A}\varepsilon_{a,a}\otimes (\varepsilon_{a,a})^{\op}\in
D_A\otimes D_A^{\op}\) by the  equality projection
\(J_A=\Omega_A\Omega_A^*\in M_A\otimes M_A^{\op}\).  The input rule remains
classical, while the output predicate is promoted from a diagonal relation to
a genuine quantum projection. This game is concurrent/synchronous. Indeed,
\(\Pi_{D_X}=J_X^{\rm cl}
=\bigvee_{x\in X}\varepsilon_{x,x}\otimes
(\varepsilon_{x,x})^{\op}\), whereas
\(\Pi_{M_A}=J_A\). Hence join-continuity gives
\[
\varphi_G^A(\Pi_{D_X})
=
\bigvee_{x\in X}J_A
=
J_A
=
\Pi_{M_A}.
\]

This game also has a projection-test realisation.  Put
\(H_{\rm in}=\mathbb C^X\otimes\overline{\mathbb C^X}\),
\(H_{\rm out}=\mathbb C^A\otimes\overline{\mathbb C^A}\), and
\(H_R=\mathbb C^X\otimes\overline{\mathbb C^X}\).  Let \(\pi\) be a full-support
probability distribution on \(X\times X\), and set
\[
\psi_\pi
=
\sum_{x,y\in X}\sqrt{\pi(x,y)}
(e_x\otimes \overline e_y)_{\rm in}\otimes(e_x\otimes \overline e_y)_R .
\]
For \(x,y\in X\), let \(Q_{x,y}=J_A\) if \(x=y\), \(Q_{x,y}=J_A^\perp\) if
\(x\sim y\), and \(Q_{x,y}=1_{A}\otimes 1_A^{\op}\) otherwise, and define
\[
Q_G^A
=
\sum_{x,y\in X}Q_{x,y}\otimes \varepsilon_{(x,y),(x,y)}
\in \B(H_{\rm out}\otimes H_R).
\]
Then \(G_A^{\rm cq}=(\psi_\pi,Q_G^A)\) is a projection-test realisation of
the lattice-map game \(\varphi_G^A\).  Indeed, for
\(T\in \B(H_{\rm in},H_{\rm out})\), the condition
\((Q_G^A)^\perp(T\otimes 1_{H_R})\psi_\pi=0\) is equivalent, since \(\pi\)
has full support and the referee basis vectors are mutually orthogonal, to
\(Q_{x,y}^\perp T(e_x\otimes \overline e_y)=0\) for every \(x,y\in X\).  Equivalently,
\(T(e_x\otimes \overline e_y)\in
\varphi_G^A(\varepsilon_{x,x}\otimes (\varepsilon_{y,y})^{\op})H_{\rm out}\) for every \(x,y\).
Thus the projection-test game tests all classical input atoms at once, and
its winning transformation space is precisely the operator space determined
by the classical-to-quantum lattice rule.

\subsection{Quantum graph homomorphism games}

The ``classical" graph homomorphism game \cite{MR16} associated with finite graphs
\(G\) and \(H\) is the synchronous non-local game whose perfect deterministic
strategies are exactly graph homomorphisms \(G\to H\).  The referee sends
vertices \(x,y\in V(G)\) to Alice and Bob, respectively, and the players
respond with vertices \(a,b\in V(H)\).  They win if  they satisfy
\( 
x=y\Rightarrow a=b\) and
\( 
x\sim_G y\Rightarrow a\sim_H b .
\)

 Non-commutative graphs were introduced by
Duan--Severini--Winter \cite{DSW13} as operator systems
\(\mathcal S\subseteq M_n\), $n\in \bb N$.  Stahlke \cite{Stahlke16} equivalently used operator anti-systems, namely selfadjoint trace-zero subspaces of \(M_n\) as it is more natural when studying homomorphisms in the non-commutative setting. 

\subsubsection{The Todorov--Turowska  version}
We  recall the quantum graph homomorphism game of
Todorov--Turowska \cite{TT24},  
(see also \cite{BHTT23}).  We use a slightly modified version of the
Todorov--Turowska formulation fitting our vectorisation convention. Let \(H_X=\mathbb C^X\).  For an operator anti-system \(\mathcal S^0\subseteq M_X\), put
\(U_{\mathcal S^0}=\operatorname{vec}(\mathcal S^0)\subseteq H_X\otimes\overline{H_X}\), where
\(\operatorname{vec}:M_X\to H_X\otimes\overline{H_X}\) is the
Hilbert--Schmidt vectorisation.  Selfadjointness of \(\mathcal S^0\) becomes invariance
of \(U_{\mathcal S^0}\) under the antiunitary flip
\(\xi\otimes\overline\eta\mapsto\eta\otimes\overline\xi\), while
\(\operatorname{Tr}(\mathcal S^0)=0\) becomes \(U_{\mathcal S^0}\perp\Omega_X\), where
\(\Omega_X=|X|^{-1/2}\sum_{x\in X} e_x\otimes\overline e_x
       \in H_X\otimes\overline{H_X}\).  If \(G\) is a classical graph,
then the off-diagonal edge space
\(\mathcal S_G^0=\operatorname{span}\{\varepsilon_{x,y}:x\sim_G y\}\) gives
\(U_{\mathcal S_G^0}=\operatorname{span}\{e_x\otimes\overline e_y:x\sim_G y\}\).

Let \(\mathcal S^0_G\subseteq M_X\) and \(\mathcal S^0_H\subseteq M_A\) be operator anti-systems, and
let
\(e_G\in M_X\otimes M_X^{\op}\) and \(e_H\in M_A\otimes M_A^{\op}\) be the
orthogonal projections onto \(\rm vec(\mathcal S_{G}^0)\) and \(\rm vec(\mathcal S_{H}^0)\).  The quantum graph homomorphism
game \(G \to H\) is the zero-preserving join-continuous map
\(\varphi_{G \to H}:\mathcal P(M_X\otimes M_X^{\op})\to
\mathcal P(M_A\otimes M_A^{\op})\) defined by
\[
\varphi_{G \to H}(P)
=
\begin{cases}
0, & P=0,\\
e_H, & 0\neq P\leq e_G,\\
I_A\otimes I_A^{\op}, & \text{otherwise}.
\end{cases}
\]
For the present rule, the only non-trivial projections are the non-zero
\(P\leq e_G\), and on all of them \(\varphi_{G \to H}(P)=e_H\).  Hence the
condition becomes \(e_H^\perp TP=0\) for every non-zero \(P\leq e_G\), which
is equivalent to \(e_H^\perp Te_G=0\).  Therefore
\[
\mathcal U_{\varphi_{G \to H}}
=
\{T\in \B(H_X\otimes\overline{H_X},H_A\otimes\overline{H_A})
: e_H^\perp Te_G=0\}.
\]

\subsubsection{Goldberg's general quantum graph formulation}

Let \(\mathcal X \subseteq \B(H_{X})\) be a finite-dimensional tracial von Neumann algebra, with \(L^2\)-space
\(L^2(\mathcal X)\). A \emph{quantum graph} $G$ \cite{MRV18} on
\(\mathcal X\) consists of a quantum adjacency operator, that is,  a self-adjoint operator \(A_G\in B(L^2(\mathcal X))\) such that
\[
m_\mathcal X(A_G\otimes A_G)m_\mathcal X^*=A_G,
\]
\[
(1\otimes \eta_\mathcal X^*m_\mathcal X)(1\otimes A_G\otimes 1)
(m_\mathcal X^*\eta_\mathcal X\otimes 1)=A_G.
\]
It is called \emph{irreflexive} if moreover
\[
m_\mathcal X(A_G\otimes 1)m_\mathcal X^*=0 .
\]

We now translate quantum graphs from operators to projections and bimodules \cite[Theorem 7.7]{MRV18}; see also \cite{Daws24}. Let \(A_G\in \B(L^2(\mathcal X))\) be a quantum graph and set
\(e_G=\Psi_{\cl X}(A_G)\). By Theorem~\ref{thm:goldberg-daws-transform}, the
quantum-graph relations, together with the self-adjointness of \(A_G\),
imply that \(e_G\) is a projection satisfying
\(\sigma_{\mathcal X}(e_G)=e_G\). Hence, by
Proposition~\ref{prop:rectangular-daws}, \(e_G\) corresponds to the
\(\mathcal X'\)-\(\mathcal X'\) operator bimodule
\( 
\mathcal S_G^0
=
\operatorname{vec}^{-1}
\big(e_G(H_{X}\otimes \overline{H_{X}})\big).
\)
The symmetry condition \(\sigma_{\mathcal X}(e_G)=e_G\)
translates to \((\mathcal S_G^0)^*=\mathcal S_G^0\). Moreover, the irreflexivity condition
\(m_{\mathcal X}(A_G\otimes 1)m_{\mathcal X}^*=0\) translates to
\(e_G\Psi_{\cl X}(1_{L^2(\mathcal X)})=0\). Note that
\( 
\Pi_{\mathcal X}=s(\Psi_{\cl X}(1_{L^2(\mathcal X)}))
\)
is the projection corresponding to the vectorised diagonal bimodule \(\mathcal X'\), and
the previous relation is equivalent to
\( 
e_G\Pi_{\mathcal X}=0.
\) Thus, under this
correspondence, an irreflexive quantum graph is represented by a
self-adjoint \(\mathcal X'\)-bimodule
\(\mathcal S_G^0\subseteq \B(H_{X})\) satisfying
\(\mathcal S_G^0\perp \mathcal X'\) with respect to the Hilbert--Schmidt
inner product. In the $\mathcal X= M_X$ case, we retrieve the operator anti-system formulation as earlier.

Let \(G\) and \(H\) be irreflexive quantum graphs on finite-dimensional
quantum sets \(\mathcal X\) and \(\mathcal A\), with adjacency operators
\(A_G\in \B(L^2(\mathcal X))\) and \(A_H\in \B(L^2(\mathcal A))\). The \emph{quantum graph
homomorphism game}~\cite{Goldberg26}  from \(G\) to \(H\) is the synchronous rule operator
\[
\lambda_{G\to H}:L^2(\mathcal X)\otimes L^2(\mathcal X^{\op})\to L^2(\mathcal A)\otimes L^2(\mathcal A^{\op})
\]
defined by
\[
\begin{aligned}
\lambda_{G\to H}
&:=
\bigl((A_H\otimes 1)m_\mathcal A^*\eta_\mathcal A\bigr)
\bigl(\eta_\mathcal X^*m_\mathcal X(A_G\otimes 1)\bigr)
+
(m_\mathcal A^*\eta_\mathcal A)(\eta_\mathcal X^*m_\mathcal X)  \\
&\quad
+
(\eta_\mathcal A\otimes\eta_\mathcal A)(\eta_\mathcal X^*\otimes\eta_\mathcal X^*)
-
(\eta_\mathcal A\otimes\eta_\mathcal A)\bigl(\eta_\mathcal X^*m_\mathcal X(A_G\otimes 1)\bigr)
-
(\eta_\mathcal A\otimes\eta_\mathcal A)(\eta_\mathcal X^*m_\mathcal X).
\end{aligned}
\]

We now describe the corresponding projection-test game.  Let
\(\mathcal X\subseteq \B(H_X)\) and \(\mathcal A\subseteq \B(H_A)\) be faithful
finite-dimensional representations, and put
\( 
\mathcal M_G=\mathcal X\otimes\mathcal X^{\op}\),
\( 
\mathcal N_H=\mathcal A\otimes\mathcal A^{\op}.
\)
We represent \(\mathcal M_G\) on
\(H_{\rm in}:=H_X\otimes\overline{H_X}\) and \(\mathcal N_H\) on
\(H_{\rm out}:=H_A\otimes\overline{H_A}\).  Set
\(e_G:=\Psi_{\mathcal X}(A_G)\) and
\(e_H:=\Psi_{\mathcal A}(A_H)\), represented on
\(H_{\rm in}\) and \(H_{\rm out}\), respectively.  Also write
\(\Pi_{\mathcal X}:=P_{\operatorname{vec}(\mathcal X')}\) and
\(\Pi_{\mathcal A}:=P_{\operatorname{vec}(\mathcal A')}\).

Applying \(\Psi_{\mathcal M_G,\mathcal N_H}\) to the five summands in
\(\lambda_{G\to H}\) gives
\[
\begin{aligned}
\Psi_{\mathcal M_G,\mathcal N_H}(\lambda_{G\to H})
&=
e_H\otimes e_G^{\op}
+
\Psi_{\mathcal A}(1_{L^2(\mathcal A)})
\otimes
\Psi_{\mathcal X}(1_{L^2(\mathcal X)})^{\op}
+
1_{\mathcal N_H}\otimes1_{\mathcal M_G}^{\op} \\
&\qquad \qquad \qquad \qquad \qquad \qquad \qquad 
-
1_{\mathcal N_H}\otimes e_G^{\op}
-
1_{\mathcal N_H}\otimes
\Psi_{\mathcal X}(1_{L^2(\mathcal X)})^{\op}.
\end{aligned}
\]
Under the special Frobenius normalisation \(mm^*=1\),
Lemma~\ref{lem:support-Psi-id} gives
\( 
\Psi_{\mathcal X}(1_{L^2(\mathcal X)})=\Pi_{\mathcal X}\),
\( 
\Psi_{\mathcal A}(1_{L^2(\mathcal A)})=\Pi_{\mathcal A}.
\)
Thus, if \(Q_{G\to H}:=\Psi_{\mathcal M_G,\mathcal N_H}(\lambda_{G\to H})\), then
\[
Q_{G\to H}
=
e_H\otimes e_G^{\op}
+
\Pi_{\mathcal A}\otimes\Pi_{\mathcal X}^{\op}
+
1_{\mathcal N_H}\otimes
(1_{\mathcal M_G}-e_G-\Pi_{\mathcal X})^{\op}.
\]
Since \(G\) is irreflexive, \(e_G\Pi_{\mathcal X}=0\).  Hence
\(e_G\), \(\Pi_{\mathcal X}\), and
\(1_{\mathcal M_G}-e_G-\Pi_{\mathcal X}\) are pairwise orthogonal
projections, and therefore \(Q_{G\to H}\) is a projection.

The associated projection-test game is \( 
G_{G\to H}=
\bigl(\Pi_{\mathcal M_{G}},Q_{G\to H}\bigr)
\) with referee algebra  
\(\cl R= \cl M_{G}^{\rm op}\). 
Let \(\mathcal U_{G\to H}\subseteq \B(H_{\rm in},H_{\rm out})\) be the
corresponding winning operator space:
\[
\mathcal U_{G\to H}
=
\{T:Q_{G\to H}^{\perp}(T\otimes 1_{\cl M_{G}^{\rm op}})
\Pi_{\mathcal M_{G}} =0\}.
\]
  Using
\((F\otimes E^{\op})\operatorname{vec}(S)=\operatorname{vec}(FSE)\), the block
formula for \(Q_{G\to H}\) gives
\[
\mathcal U_{G\to H}
=
\{T\in \B(H_{\rm in},H_{\rm out}):
e_H^\perp T e_G=0,\ 
\Pi_{\mathcal A}^{\perp}T\Pi_{\mathcal X}=0\}.
\]
In particular, \(\mathcal U_{G\to H}\) is an
\(\mathcal N_H'\)-\(\mathcal M_G'\) operator bimodule and by Proposition~\ref{prop:Goldberg-concurrency}, equivalently by Proposition~\ref{prop:concurrency-equivalences}, the game is concurrent/synchronous. Finally, \(\mathcal U_{G\to H}\) is reflexive as an operator space.  Indeed, if
\(S\in\operatorname{Ref}(\mathcal U_{G\to H})\) and
\(\xi\in\operatorname{Ran}e_G\), then
\(\mathcal U_{G\to H}\xi\subseteq\operatorname{Ran}e_H\), hence
\(S\xi\in\operatorname{Ran}e_H\).  Thus \(e_H^\perp S e_G=0\).  Similarly,
\(\Pi_{\mathcal A}^{\perp}S\Pi_{\mathcal X}=0\).  Hence
\(S\in\mathcal U_{G\to H}\), and so
\(\operatorname{Ref}(\mathcal U_{G\to H})=\mathcal U_{G\to H}\).
Therefore the Goldberg graph homomorphism rule determines a reflexive
operator bimodule, and hence a projection-lattice quantum game.

\subsubsection{Comparison in the full matrix algebra case}

Assume now that \(\mathcal X=M_X= \B(H_X)\) and
\(\mathcal A=M_A= \B(H_A)\), where \(H_X=\mathbb C^X\) and
\(H_A=\mathbb C^A\). Recall that \(\Tr_X\) denotes the unnormalised trace. Hence, for
\(\mathcal X=M_X\), the special Frobenius normalisation is obtained from the
trace \(\tau_X=|X|\Tr_X\), and similarly for \(M_A\).  Put
\(H_{\rm in}=H_X\otimes\overline{H_X}\),
\(H_{\rm out}=H_A\otimes\overline{H_A}\), and
\(H_R=\overline{H_{\rm in}}\).  Then
\(\mathcal M_G=M_X\otimes M_X^{\op}=\B(H_{\rm in})\),
\(\mathcal N_H=M_A\otimes M_A^{\op}=\B(H_{\rm out})\) and $ \cl R = \cl M_{G}^{\rm op}= \cl B(\overline H_{\rm in})$.

Let \(e_G=P_{\operatorname{vec}(\mathcal S_G^0)}\) and
\(e_H=P_{\operatorname{vec}(\mathcal S_H^0)}\).  Since
\(\mathcal X'=\mathbb C1\) and \(\mathcal A'=\mathbb C1\), we have
\(P_{\operatorname{vec}(\mathcal X')}=J_{X}\) and
\(P_{\operatorname{vec}(\mathcal A')}=J_{A}\) while the special normalisation gives $ \Psi_{\cl X}(1_{L^2(\cl X)})= J_X$ and $ \Psi_{\cl A}(1_{L^2(\cl A)})= J_A$.
Thus the Goldberg accepting projection is
\[
Q_{\rm Gold}
=
e_H\otimes e_G^{\op}
+
J_{A}\otimes J_{X}^{\op}
+
1_{H_{\rm out}}\otimes
(1_{H_{\rm in}}-e_G-J_{X})^{\op}.
\]
Since \(\mathcal M_G'= \mathbb C1_{H_{\rm in}}\), the input projection is
\(\Omega_{\rm in} \Omega_{\rm in}^*\), where
\[
\Omega_{\rm in}
=
|X|^{-1}
\sum_{x,y\in X}
(e_x\otimes\overline e_y)
\otimes
\overline{(e_x\otimes\overline e_y)} .
\]
Hence the Goldberg projection-test game is
\(G^{\rm Gold}_{G\to H}=(\Omega_{\rm in},Q_{\rm Gold})\), and its winning
space is
\[
\mathcal U_{\rm Gold}
=
\{T\in \B(H_{\rm in},H_{\rm out}):
e_H^\perp T e_G=0,\ 
J_{A}^{\perp}TJ_{X}=0\}.
\]

On the other hand, the Todorov--Turowska quantum graph homomorphism rule, as defined
in \cite{TT24,BHTT23},  imposes the support condition
\[
\mathcal U_{\rm TT}
=
\{T\in \B(H_{\rm in},H_{\rm out}): e_H^\perp T e_G=0\}.
\]
In the same full-vectorisation projection-test model it is realised by
\(P=\Omega_{\rm in} \Omega_{\rm in}^*\) and
\[
Q_{\rm TT}
=
e_H\otimes e_G^{\op}
+
1_{H_{\rm out}}\otimes(1_{H_{\rm in}}-e_G)^{\op}.
\]
Thus
\[
\mathcal U_{\rm Gold}
=
\mathcal U_{\rm TT}
\cap
\{T:T\Omega_X\in\mathbb C\Omega_A\},
\qquad
Q_{\rm Gold}
=
Q_{\rm TT}
-
J_{A}^{\perp}\otimes J_{X}^{\op}.
\]
Therefore the edge part of Goldberg's rule agrees with the
Todorov--Turowska quantum graph homomorphism rule, while the extra
\(J_{A}\otimes J_{X}^{\op}\) term enforces the 
synchronicity condition.

\subsection{Quantum graph isomorphism games}

The classical graph isomorphism game of Atserias et al.~\cite{AMRSSV19}
is a synchronous non-local game played with two finite graphs \(G\) and \(H\).  The question set for
each player is \(V(G)\sqcup V(H)\); questions from \(G\) must be answered in
\(H\), and questions from \(H\) must be answered in \(G\).  The winning rule
requires equality, adjacency and non-adjacency to be preserved, together with
the usual consistency conditions across the two graphs.  Thus a perfect
deterministic strategy is exactly a pair of mutually inverse graph
homomorphisms, equivalently a graph isomorphism \(G\cong H\).

\subsubsection{The Brannan et al. version}

We recall the isomorphism version of the quantum graph homomorphism game \cite{BHTT24}.  Let
\(H_X=\mathbb C^X\) and \(H_A=\mathbb C^A\).  For operator anti-systems
\(\mathcal S_G^0\subseteq M_X\) and \(\mathcal S_H^0\subseteq M_A\), let
\(e_G\in M_X\otimes M_X^{\op}\) and
\(e_H\in M_A\otimes M_A^{\op}\) be the orthogonal projections onto
\(\operatorname{vec}(\mathcal S_G^0)\subseteq H_X\otimes\overline{H_X}\) and
\(\operatorname{vec}(\mathcal S_H^0)\subseteq H_A\otimes\overline{H_A}\),
respectively.

At the level of winning transformations, the isomorphism game imposes the
homomorphism condition in both directions:
\[
e_H^\perp T e_G=0,
\qquad
e_H T e_G^\perp=0.
\]
Equivalently, \(T\operatorname{Ran}(e_G)\subseteq\operatorname{Ran}(e_H)\)
and \(T\operatorname{Ran}(e_G)^\perp\subseteq\operatorname{Ran}(e_H)^\perp\).
Thus the winning transformation space is
\[
\mathcal W_{\rm TT}
=
\{T\in \B(H_X\otimes\overline{H_X},H_A\otimes\overline{H_A}):
e_H^\perp T e_G=0,\ e_HTe_G^\perp=0\}.
\]

The corresponding projection-lattice rule is the zero-preserving
join-continuous map
\(\varphi_{G\cong H}:\mathcal P(M_X\otimes M_X^{\op})\to
\mathcal P(M_A\otimes M_A^{\op})\) given by
\[
\varphi_{G\cong H}(P)
=
\begin{cases}
0, & P=0,\\
e_H, & 0\neq P\leq e_G,\\
e_H^\perp, & 0\neq P\leq e_G^\perp,\\
1_{A}\otimes 1_A^{\rm op}, & \text{otherwise}.
\end{cases}
\]
Indeed, testing projections below \(e_G\) gives \(e_H^\perp T e_G=0\), while
testing projections below \(e_G^\perp\) gives \(e_HTe_G^\perp=0\).  Hence the
associated operator space is precisely \(\mathcal W_{\rm TT}\).

\subsubsection{Goldberg's quantum graph isomorphism game}

Let \(G\) and \(H\) be irreflexive quantum graphs on finite-dimensional
tracial von Neumann algebras \(\mathcal X\) and \(\mathcal A\), with adjacency operators
\(A_G\in \B(L^2(\mathcal X))\) and \(A_H\in \B(L^2(\mathcal A))\).  Put
\(\mathcal M_G=\mathcal X\otimes\mathcal X^{\op}\) and
\(\mathcal N_H=\mathcal A\otimes\mathcal A^{\op}\).  The \emph{quantum graph isomorphism game} rule~\cite{Goldberg26} is
\[
\lambda_{G\cong H}
=
m_{\mathcal N_H}
(\lambda_{G\to H}\otimes\lambda_{H\to G}^*)
m_{\mathcal M_G}^*,
\]
where \(\lambda_{H\to G}^*:L^2(\mathcal M_G)\to L^2(\mathcal N_H)\) is the
adjoint of the reverse homomorphism rule. We note that $\lambda_{G\cong H}$ is bisynchronous, that is, both $\lambda_{G\cong H}$ and $ \lambda_{G\cong H}^*$ are synchronous \cite[Proposition 5.14]{Goldberg26}. 

Choose faithful finite-dimensional representations
\(\mathcal X\subseteq \B(H_X)\) and \(\mathcal A\subseteq \B(H_A)\), and set
\(H_{\rm in}=H_X\otimes\overline{H_X}\),
\(H_{\rm out}=H_A\otimes\overline{H_A}\), and
\(H_R=\overline{H_{\rm in}}\).  Let
\(e_G=\Psi_{\mathcal X}(A_G)\) and
\(e_H=\Psi_{\mathcal A}(A_H)\), represented on \(H_{\rm in}\) and
\(H_{\rm out}\).  Put
\[
\Pi_{\mathcal X}=P_{\operatorname{vec}(\mathcal X')},
\qquad
\Pi_{\mathcal A}=P_{\operatorname{vec}(\mathcal A')},
\qquad
\Theta_G=1_{\mathcal M_G}-e_G-\Pi_{\mathcal X},
\qquad
\Theta_H=1_{\mathcal N_H}-e_H-\Pi_{\mathcal A}.
\]
Assuming the special Frobenius normalisation for \(\mathcal X\) and \(\mathcal A\),
Theorem~\ref{thm:goldberg-daws-transform} gives
\[
Q_{G\cong H}:=\Psi_{\mathcal M_G,\mathcal N_H}(\lambda_{G\cong H})
=
e_H\otimes e_G^{\op}
+
\Pi_{\mathcal A}\otimes\Pi_{\mathcal X}^{\op}
+
\Theta_H\otimes\Theta_G^{\op}.
\]
Thus the projection-test game is
\( 
\bigl(\Pi_{\mathcal M_{G}},Q_{G\cong H}\bigr),
\)
with referee algebra \(\cl R=\cl M_{G}^{\rm op}\).  Its winning operator space is
\[
\mathcal W_{\rm Gold}
=
\{T\in \B(H_{\rm in},H_{\rm out}):
e_H^\perp T e_G=0,\ 
\Pi_{\mathcal A}^\perp T\Pi_{\mathcal X}=0,\ 
\Theta_H^\perp T\Theta_G=0\}.
\]
 In particular,
\(\mathcal W_{\rm Gold}\) is an
\(\mathcal N_H'\)-\(\mathcal M_G'\) operator bimodule and by Proposition~\ref{prop:Goldberg-concurrency} (equiv. Proposition~\ref{prop:concurrency-equivalences}), the game is concurrent/synchronous.

A similar proof  as in the homomorphism case shows that
\(\mathcal W_{\rm Gold}\) is reflexive.  

\subsubsection{Comparison in the full matrix algebra case}

Assume now that \(\mathcal X=M_X\subseteq \B(H_X)\) and
\(\mathcal A=M_A\subseteq \B(H_A)\), where \(H_X=\mathbb C^X\) and
\(H_A=\mathbb C^A\), and use the special traces
\(\tau_X=|X|\Tr_X\) and \(\tau_A=|A|\Tr_A\).  Then
\(\Pi_{\mathcal X}=J_{X}\) and
\(\Pi_{\mathcal A}=J_{A}\), where
\(\Omega_X=|X|^{-1/2}\sum_x e_x\otimes\overline e_x\) and
\(\Omega_A=|A|^{-1/2}\sum_a e_a\otimes\overline e_a\).  Moreover
\(\mathcal M_G=\B(H_{\rm in})\), so the input projection is
\(\Omega_{\rm in} \Omega_{\rm in}^*\), the projection onto the maximally entangled vector in
\(H_{\rm in}\otimes\overline{H_{\rm in}}\).

In this case the Todorov--Turowska isomorphism rule gives
\[
\mathcal W_{\rm TT}
=
\{T:e_H^\perp T e_G=0,\ e_HTe_G^\perp=0\},
\]
or, in the full-vectorisation projection-test model,
\[
Q_{\rm TT}
=
e_H\otimes e_G^{\op}
+
e_H^\perp\otimes(e_G^\perp)^{\op},
\qquad
P=\Omega_{\rm in} \Omega_{\rm in}^*.
\]
Goldberg's rule instead splits the complements of \(e_G\) and \(e_H\) into
the diagonal and remaining parts.  With
\(\Theta_G=1-e_G-J_{X}\) and
\(\Theta_H=1-e_H-J_{A}\), it is
\[
Q_{\rm Gold}
=
e_H\otimes e_G^{\op}
+
J_{A}\otimes J_{X}^{\op}
+
\Theta_H\otimes\Theta_G^{\op}.
\]
Consequently,
\[
\mathcal W_{\rm Gold}
=
\mathcal W_{\rm TT}
\cap
\{T:T\Omega_X\in\mathbb C\Omega_A,\ 
T\operatorname{Ran}\Theta_G\subseteq\operatorname{Ran}\Theta_H\}.
\]
Equivalently,
\[
Q_{\rm Gold}
=
Q_{\rm TT}
-
J_{A}\otimes\Theta_G^{\op}
-
\Theta_H\otimes J_{X}^{\op}.
\]
Thus, in the full matrix case, Goldberg's isomorphism rule refines the
Todorov--Turowska isomorphism rule by separating the edge-complement sector
into the diagonal/synchronicity block and the remaining block.

\subsection{Quantum-to-classical graph homomorphism games}

We recall the quantum-to-classical graph homomorphism game of
Brannan--Ganesan--Harris \cite{BGH22}. In this example the authors used ``reflexive" (all loops) quantum graphs, that is 
\((\mathcal S,\mathcal X,M_n)\)  such that  \(\mathcal X\subseteq M_n\) is a
non-degenerate finite-dimensional von Neumann algebra and
\(\mathcal S\subseteq M_n\) is an operator system which is an \(\mathcal X'\)-bimodule.  Let \(G\) be a finite
classical graph with vertex set \(A:=V(G)\).

We first recall the auxiliary basis obtained in \cite[Proposition 4.1]{BGH22}.  Choose a
decomposition
\( 
\mathbb C^n=K_1\oplus\cdots\oplus K_r
\)
such that \(\mathcal X\) acts irreducibly on each \(K_i\), and let \(P_{K_i}\)
be the projection onto \(K_i\).  A \emph{quantum edge basis} for
\(\mathcal S\) is an orthonormal basis \(F=\{Y_\alpha\}_\alpha\) of
\(\mathcal S\), with respect to the unnormalised Hilbert--Schmidt inner
product, which contains an orthonormal basis of \(\mathcal X'\) chosen to
include the normalised projections \((\dim K_i)^{-1/2}P_{K_i}\), and has each
\(Y_\alpha\) supported in a single block \(P_{K_i}M_nP_{K_j}\).

Fix an orthonormal basis \(v_1,\ldots,v_n\) of \(\mathbb C^n\) adapted to this
decomposition.  If
\(Y_\alpha=\sum_{p,q}y_{\alpha,pq}v_pv_q^*\), set
\[
\xi_\alpha
=
\operatorname{vec}(Y_\alpha)
=
\sum_{p,q}y_{\alpha,pq}v_p\otimes  \overline v_q
\in \mathbb C^n\otimes \overline{\mathbb C^n} .
\]
The referee samples one of the vectors \(\xi_\alpha\), sends the two tensor
legs to Alice and Bob, and receives classical outputs
\(a,b\in A\).  The winning rule is
\[
Y_\alpha\in\mathcal X' \Rightarrow a=b,
\qquad
Y_\alpha\perp\mathcal X' \Rightarrow a\sim_G b.
\]
Thus the input system is quantum, while the output system is classical.

We now express this game as a projection test.  Put
\( 
H_{\rm in}=\mathbb C^n\otimes \overline{\mathbb C^n}\) and
\( 
H_{\rm out}=\mathbb C^A\otimes\overline{\mathbb C^A}.
\)
For each \(\alpha\), define the diagonal accepting projection
\(Q_\alpha\in D_A\otimes D_A^{\op}\) by
\[
Q_\alpha=
\begin{cases}
\displaystyle\sum_{a\in A}
\varepsilon_{a,a}\otimes(\varepsilon_{a,a})^{\op},
& Y_\alpha\in\mathcal X',\\[1.2em]
\displaystyle\sum_{a\sim_G b}
\varepsilon_{a,a}\otimes(\varepsilon_{b,b})^{\op},
& Y_\alpha\perp\mathcal X'.
\end{cases}
\]
Let \(H_R=\bb C^F\), with orthonormal basis \((f_\alpha)_\alpha\).  If
\(\pi\) is a full-support probability distribution on \(F\), set
\[
\psi_\pi
=
\sum_\alpha\sqrt{\pi_\alpha}\,\xi_\alpha\otimes f_\alpha,
\qquad
Q_{\rm BGH}
=
\sum_\alpha Q_\alpha\otimes f_\alpha f_\alpha^* .
\]
Then \(G_{\rm BGH}=(\psi_\pi,Q_{\rm BGH})\) is a projection-test
realisation of the game.  Indeed, for
\(T\in \B(H_{\rm in},H_{\rm out})\), the condition
\[
Q_{\rm BGH}^\perp(T\otimes 1_{H_R})\psi_\pi=0
\]
is equivalent, since \(\pi\) has full support and the vectors
\((f_\alpha)_\alpha\) are orthonormal, to
\( 
Q_\alpha^\perp T\xi_\alpha=0\) 
for every \(\alpha .
\)
Hence the winning transformation space is
\[
\mathcal U_{\rm BGH}
=
\{T\in \B(H_{\rm in},H_{\rm out}):
Q_\alpha^\perp T\xi_\alpha=0\ \text{for all }\alpha\}.
\]
Moreover, \(\mathcal U_{\rm BGH}\) is an
\((D_A\otimes D_A^{\op})\)-\((\mathcal X'\otimes(\mathcal X')^{\op})\)
bimodule. 

Let $\Gamma$ be a perfect strategy for $ G_{\rm BGH}$ with Kraus operators $(T_i)$. Since the quantum edge basis contains an orthonormal basis of \(\mathcal X'\), the conditions
\(Q_\alpha^\perp T_i\xi_\alpha=0\) for the basis elements \(Y_\alpha\in\mathcal X'\) give
\[
\Pi_{D_A}^\perp T_i\Pi_{\mathcal X}=0,
\]
for every Kraus operator \(T_i\), where \(
\Pi_{\mathcal X}:=P_{\operatorname{vec}(\mathcal X')}\) and
\(\Pi_{D_A}:=P_{\operatorname{vec}(D_A)}
=
\sum_{a\in V(G)}\varepsilon_{a,a}\otimes(\varepsilon_{a,a})^{\op}
\). Hence, by Proposition~\ref{prop:concurrency-equivalences} the BGH game is concurrent/synchronous.

Finally, since the game is quantum-to-classical, Corollary \ref{cor:q2c_c2q} implies the game determines a projection-lattice game.  We omit the explicit projection-lattice rule, since its case-by-case formula
over the full projection lattice is technical and adds little to the comparison.

\subsection*{Acknowledgments}
This work was motivated by a discussion on the definitions of quantum non-local games initiated by Adina Goldberg during the Operator Algebras and Quantum Information program at Institut Mittag-Leffler. The author thanks Adina Goldberg for discussions and useful comments on this manuscript, and is grateful to Ivan G. Todorov and Lyudmila Turowska for their helpful feedback. The author also thanks Aristides Katavolos for bringing \cite{Lar82} to their attention. This material is based upon work supported by the Swedish Research Council under grant no. 2021-06594, while the author was in residence at Institut Mittag-Leffler in Djursholm, Sweden, during the Spring semester of 2026. The author warmly thanks the Institute for its hospitality and excellent working conditions. The author acknowledges the use of ChatGPT/Codex (GPT-5.6 Sol, OpenAI) during manuscript preparation to assist with consistency checks and literature searches. All mathematical statements, proofs and references were reviewed and written by the author, who takes full responsibility for the content.

\subsection*{Conflict of Interest}
The author has no conflicts to disclose.

\subsection*{Data Availability}
Data sharing is not applicable to this article as no new data were created or analysed in this study.



\end{document}